%% file: boscia.tex
\documentclass{article}

\usepackage[T1]{fontenc}
\usepackage[utf8]{inputenc}
\usepackage[english]{babel}
\usepackage{newtxtext}

\usepackage{amsmath}
\usepackage{amssymb}
\usepackage{bm}
\usepackage{bbm}
\usepackage{spalign}
\usepackage{enumerate}
\usepackage{fix-cm}
\usepackage{paralist}

\usepackage{charter}
\newcommand{\package}{\texttt{Boscia.jl}}

\usepackage[obeyspaces]{url}
\usepackage{subcaption}

\usepackage{graphicx}
\usepackage{float}

\graphicspath{{new_images/}}

\usepackage{fancyvrb}

\usepackage[theme=grayscale]{jlcode}

\usepackage[section]{algorithm}
\usepackage{algorithmicx} 
\usepackage{algpseudocode}

\usepackage{xspace}

\usepackage[figuresright]{rotating}
\usepackage{csvsimple}
\usepackage{makecell}
\usepackage{caption}

\usepackage{xcolor} 

\usepackage[colorinlistoftodos]{todonotes}

\usepackage[scaled=0.9]{FiraMono}

\usepackage{array}
\newcolumntype{H}{>{\setbox0=\hbox\bgroup}c<{\egroup}@{}}

\definecolor{brandeisblue}{rgb}{0.0, 0.44, 1.0}
\newcommand{\revision}[1]{{ #1 }}
\newcommand{\revisionTwo}[1]{{ #1}}

\usepackage{booktabs}

\usepackage{amsthm}
\usepackage{jmlr2e}
\usepackage{natbib}
\setcitestyle{authoryear,round,citesep={;},aysep={,},yysep={;}}
\usepackage[vvarbb]{newtxmath}
\usepackage{multirow} 
\usepackage[noabbrev,capitalize,nameinlink]{cleveref}[0.21]

\newtheorem{assumption}{Assumption}

\newtheorem{theorem}{Theorem}
\newtheorem{definition}{Definition}

\definecolor{RED}{rgb}{1,0,0}\definecolor{BLUE}{rgb}{0,0,1}

\newif\ifarxiv
\arxivtrue

\title{Convex mixed-integer optimization with Frank-Wolfe methods}

\author{\name Deborah Hendrych \email \href{mailto:Hendrych@zib.de}{hendrych@zib.de}\\
\addr Freie Universit\"at Berlin, Germany\\
\addr Zuse Institute Berlin, Germany
\AND
\name Hannah Troppens \email \href{mailto:Troppens@zib.de}{troppens@zib.de}\\
\addr Freie Universit\"at Berlin, Germany\\
\addr Zuse Institute Berlin, Germany
\AND
\name Mathieu Besançon \email \href{mailto:besancon@zib.de}{besancon@zib.de} \\
\addr Zuse Institute Berlin, Germany
\AND
\name Sebastian Pokutta \email \href{mailto:pokutta@zib.de}{pokutta@zib.de} \\
\addr Technische Universit\"at Berlin, Germany\\
Zuse Institute Berlin, Germany
}

\DeclareMathOperator*{\argmax}{argmax}
\DeclareMathOperator*{\argmin}{argmin}
\newcommand{\innp}[2]{\left\langle #1, #2 \right\rangle}
\newcommand{\norm}[1]{\left\| #1 \right\|}
\newcommand{\vx}{\mathbf{x}}
\newcommand{\vvv}{\mathbf{v}}

\newcommand{\vd}{\mathbf{d}}
\newcommand{\vc}{\mathbf{c}}
\newcommand{\vl}{\mathbf{l}}
\newcommand{\vu}{\mathbf{u}}

\newcommand{\Xset}{{\ensuremath{\mathcal{X}}}\xspace}

\begin{document}

\maketitle

\begin{abstract}
Mixed-integer nonlinear optimization encompasses a broad class of problems that present both theoretical and computational challenges.
We propose a new type of method to solve these problems based on a branch-and-bound algorithm with convex node relaxations.
These relaxations are solved with a Frank-Wolfe algorithm over the convex hull of mixed-integer feasible points instead of the continuous relaxation
via calls to a mixed-integer linear solver as the linear minimization oracle.
The proposed method computes feasible solutions while working on a single representation of the polyhedral constraints,
leveraging the full extent of mixed-integer linear solvers without an outer approximation scheme
and can exploit inexact solutions of node subproblems.
\end{abstract}

\section{Introduction}

Mixed-integer nonlinear optimization problems (MINLP) are a challenging class combining both
combinatorial structures and nonlinearities which can model a broad range of problems arising in engineering,
transportation, and more generally operations and other application contexts.
Combinatorial constraints can capture rich properties required for a solution, e.g.,~solutions that must be a path,
cycle, or tour in a graph, or solutions with maximum support or guaranteed sparsity.
The dominant algorithmic frameworks for solving such problems are combinations of branch-and-bound (BnB) with spatial and
integer branching, and cutting planes added iteratively.

We focus in this paper on mixed-integer convex problems in which the nonlinear constraints and objectives are convex and present Boscia, a new
algorithmic framework for solving these problems that exploit recent advances in so-called Frank-Wolfe (FW) or Conditional Gradient (CG) methods.
The problem class we consider is of the type:
\begin{align}\label{prob:definition}
\min_{\vx \in \mathbb{R}^n}\, f(\vx) \,\,\text{s.t.}\, \vx \in \Xset,
\end{align}
where \Xset is a compact nonconvex set admitting a \emph{boundable linear minimization oracle} (BLMO), i.e.,~a set over which optimizing
a linear function can be done efficiently (comparatively to the original problem), even when bounds are added or modified.
Formally, we consider we have access to an oracle taking new bounds $(\mathbf{l}, \mathbf{u})$ and a direction $\mathbf{d}$:
\begin{align}\tag{B-LMO}
(\mathbf{l},\mathbf{u},\mathbf{d}) \in \mathbb{R}^n\times \mathbb{R}^n\times \mathbb{R}^n \rightarrow \argmin_{\vvv \in \mathbb{R}^n}\, \innp{\vvv}{\mathbf{d}} \,\,\text{s.t.}\,\, \vvv \in \Xset \cap \left[\mathbf{l},\mathbf{u}\right].
\end{align}
We assume that one source of nonconvexity of \Xset stems from integrality constraints and
will formalize that requirement in the following section. Given $J \subseteq \left\{1\dots n\right\}$, we will denote:
\begin{align*}
    & \Xset = \overline{\Xset} \cap \mathbb{Z}_J,\\
    \text{where }\, & \mathbb{Z}_J = \left\{\vx\in \mathbb{R}^n, \vx_j \in \mathbb{Z}\, \forall j \in J \right\},
\end{align*}
with $\overline{\Xset}$ a continuous relaxation of \Xset (which is in general not unique).
A typical example for \Xset includes polyhedral and integrality constraints, and combinatorial structures (which we will detail in the following sections).
The objective $f: \text{conv}(\Xset) \rightarrow \mathbb{R}$ is a differentiable Lipschitz-smooth convex function with a gradient
accessed as an oracle, i.e.,~we do not require an expression graph.
The proposed framework extends to nonconvex objectives when the nonconvexity affects only binary variables and exact convexifiers can be built by adding convex terms vanishing on $\{0,1\}$ \citep{crama2022quadratization}.
It can also be extended to non-smooth generalized self-concordant functions using the methods proposed in \citet{carderera2021simple,carderera2024scalable}.
In the more general non-convex case, the framework can be used as a heuristic to compute high-quality solutions.

Several families of MINLP methods have been developed over the last decades, rooted in both nonlinear and discrete optimization approaches;
see \citet{kronqvist2019review} for a recent review of methods for convex MINLPs.
Some solvers such as \texttt{Knitro}~\citep{byrd2006k} implement a BnB algorithm that solves a nonlinear relaxation at each node of the tree.
These subproblems are however expensive: typical algorithms are interior points or sequential quadratic programming, both relying on expensive second-order information and yielding high-accuracy solutions.
An alternative family of approaches was introduced in \citet{quesada1992lp}, solving Linear Programs (LPs) at each node and nonlinear problems only to compute
feasible solutions and additional cuts when integer solutions are found.
An LP-based method is also at the core of the \revision{solution process of \texttt{COUENNE}~\citep{belotticouenne} and} \texttt{SCIP} \citep{bestuzheva2021scip,bestuzheva2023enabling}, with additional techniques handling nonconvex nonlinearities.
The solvers \texttt{SHOT} \citep{lundell2022supporting,kronqvist2016extended} and \texttt{Pavito.jl} \citep{pavito} use single-tree polyhedral outer approximation approaches, relying on the Mixed-Integer Programming (MIP) solver to build the main relaxation of the problem,
similar to the framework proposed in \citet{bertsimas2021unified} for mixed-binary convex problems.
\texttt{Bonmin}~\citep{BONAMI2008186} uses a hybrid approach with both outer polyhedral approximations and BnB.
\revision{As noteworthy commercial MINLP solvers, \texttt{Antigone}~\citep{misener2014antigone} hybridizes MIP and LP relaxations and NLP subproblems to compute feasible solutions,
while \texttt{BARON}~\citep{sahinidis1996baron} primarily solves LP relaxations in a BNB with polyhedral under- and overestimators for nonlinearities, but exploits a portfolio of MIP and NLP bounds to selectively improve on the LP relaxation at nodes~\citep{kilincc2018exploiting}.}
Early termination techniques have been exploited for other convex relaxations.
In \citet{fletcher1998numerical} and \citet{buchheim2016feasible}, mixed-integer quadratic problems are handled with early termination and
warm starts in the context of active set methods, while \citet{liang2020early} consider infeasible interior-point methods.
In \citet{chen2022early}, mixed-integer conic problems are tackled with ADMM using inexact termination
and a safe dual bound recovery in a BnB framework.
Unlike ADMM, Frank-Wolfe is a primal approach with a lower bound computed at each iteration that can be used to terminate
node processing and derive a global tree bound. Furthermore, our BLMO-based solver avoids projections onto cones which can be expensive (requiring eigendecompositions)
or numerically unstable (for exponential and positive semidefinite cones for instance).
It is important to note that in many MINLP solution approaches, nonlinearities are handled through separation (via gradient cuts or supporting hyperplane cuts).
Separation is not possible if the nonlinearity is in the objective and most solvers therefore allow for
nonlinearities in the constraints only, requiring an epigraph formulation of the objective function
losing the initial structure of Problem \eqref{prob:definition}.
Some lines of work focused on specialized BnB approaches for some Sparse Regression or best subset selection problems
\citep{bertsimas2016best,hazimeh2020fast,gomez2021mixed,moreira2022alternating}.
These approaches however rely on the properties of the quadratic loss and/or of the constraint set and do not extend to other loss functions (e.g.,~logistic or Poisson regression),
additional constraints (arbitrary linear constraints on the features, grouped subset selection, integer predictor),
or other MINLPs of the form \eqref{prob:definition}.

Ideas for warm-starting Frank-Wolfe variants based on the active set representation of the solution (i.e.,~as a convex combination of vertices)
have been explored in the past;
in \citet{moondra2021reusing}, an Away-step FW (AFW) variant leveraging active sets from previous iterations in an online optimization setting.
In \citet{joulin2014efficient}, FW or variants thereof are called over the integer hull of combinatorial problems,
calling a MIP solver as BLMO and additional roundings to compute feasible solutions of good quality.
The hardness of optimization over the integer hull motivated \emph{lazified} conditional gradient approaches in \citet{braun2017lazifying},
\revision{i.e., FW-based approaches which at some steps replace the Linear Minimization Oracle (LMO) with an inexact procedure yielding an improving direction, as long as these directions provide sufficient progress.
The control of the expected progress ensures lazified methods converge at a rate similar to their exact counterpart, while drastically reducing the number of linear minimization oracle calls.}
We however generalize lazification techniques to consider the whole tree.
BnB with AFW has been explored in \citet{buchheim2018frank} for a portfolio problem,
each subproblem is solved inexactly to compute lower bounds from the FW gap, and each branching performs fixing to all possible values of the integer variable instead of adding bounds.
In contrast, we leverage the Blended Pairwise Conditional Gradient (BPCG) algorithm \citep{tsuji2021sparser} which enables us to aggressively reuse information (from computed vertices);
its convergence speed is typically higher than AFW while the iterates remain sparser. This sparsity greatly benefits branching as the fractionality of solutions is lower in most cases.
We point interested readers to \citet{CGFWSurvey2022} for a general overview of FW methods.

\subsubsection*{Our approach and contribution}
We propose a different solution approach to Problem~\eqref{prob:definition},
consisting at its heart of a branch-and-bound process over the convex hull of the feasible region with inexact relaxation solutions via first-order methods.
A key feature is that at each node, FW solves the nonlinear subproblem over the convex hull of integer-feasible solutions or \emph{integer hull} and not over the continuous relaxation.
This tight relaxation is made available by solving a MIP as the LMO within the FW solving process, thus resulting in vertices of the integer hull.
No other access to the feasible region (e.g., complete description as linear inequalities)
or the objective function (e.g., Hessian information or expression graph) is required.
In particular, we do not require epigraph formulations, and we retain the original polyhedral structure;
see also Figure~\ref{fig:minlpframeworks} for a schematic overview of the Boscia algorithm.
The central novel elements of our approach are:
\begin{enumerate}
	\item Instead of solving weaker continuous relaxations at each node of the branch-and-bound tree,
	we optimize over the convex hull of integer-feasible solutions by solving mixed-integer linear problems in the BLMO,
    obtaining stronger relaxations,	several feasible solutions at each node, and reducing the size of the tree;
	\item We replace exact convex solvers with a Frank-Wolfe-based error-adaptive solution process, i.e.,~a solution process
	for which the amount of computations performed is increasing for higher accuracies;
	\item We leverage the active set representation of the solution at each node to warm-start the children iterates and reduce the number of calls to the Mixed-Integer Programming (MIP) solver;
	\item We develop new lazification techniques and strong branching strategies for our framework;
    \item We exploit the convexity and Frank-Wolfe gap to tighten bounds at each node, further reducing the search space.
\end{enumerate}

One motivation for outer approximation is the ability to leverage a MIP solver through single-tree approaches
i.e.,~the MIP solver is called only once and coordinates the whole solving process with separation callbacks.
This approach however carries several issues:
\revision{in its pure form, only near-feasible and near-optimal solutions are obtained in the limit towards the end of the solution process. Solvers thus typically rely on MIP heuristics for primal feasible solutions.}
Furthermore, the separation constraints for the nonlinear feasible set, often linear cutting planes, \revision{can be dense for multivariate nonlinear expressions}.
These dense cuts slow down the MIP solving process and yield numerically ill-conditioned LP relaxations.
We face none of these issues since we do not rely on outer approximations at all but rather on nonlinear relaxations.
While this may seem inefficient at first and naive realizations would require a very large number of MIP subproblems
to be solved, our framework allows utilizing recent advances in FW and MIP methods to considerably reduce the number of MIP subproblems.
Specifically, our method is the first one, to the best of our knowledge, that exploits the flexibility and advances of modern MIP solvers \revision{within the NLP solving process}, maintaining feasibility throughout the process.


We complement the above with extensive computational results to demonstrate the effectiveness of our approach.
The algorithm is available as a registered Julia package \package{}\footnote{\url{https://github.com/ZIB-IOL/Boscia.jl}} open-sourced under the MIT license.

\begin{figure}
    \centering
    \begin{subfigure}[t]{0.2\textwidth}
        \centering
        \includegraphics[width=0.8\textwidth]{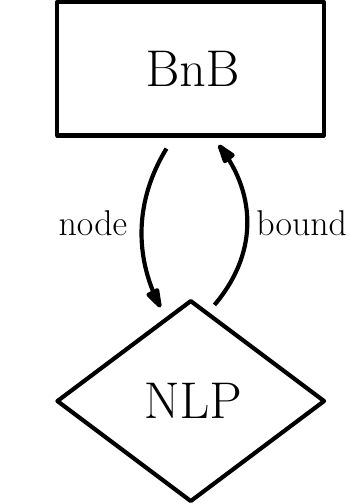}
        \caption{BnB with NLP nodes}
        \label{fig:branchnlp}
    \end{subfigure}
    \hfill
    \begin{subfigure}[t]{0.2\textwidth}
        \centering
        \includegraphics[width=0.8\textwidth]{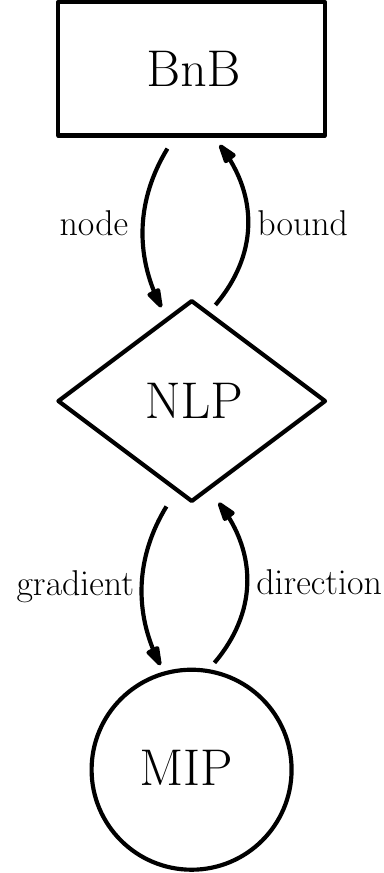}
        \caption{Our approach}
        \label{fig:branchwolfe}
    \end{subfigure}
    \hfill
    \begin{subfigure}[t]{0.4\textwidth}
        \centering
        \includegraphics[width=0.8\textwidth]{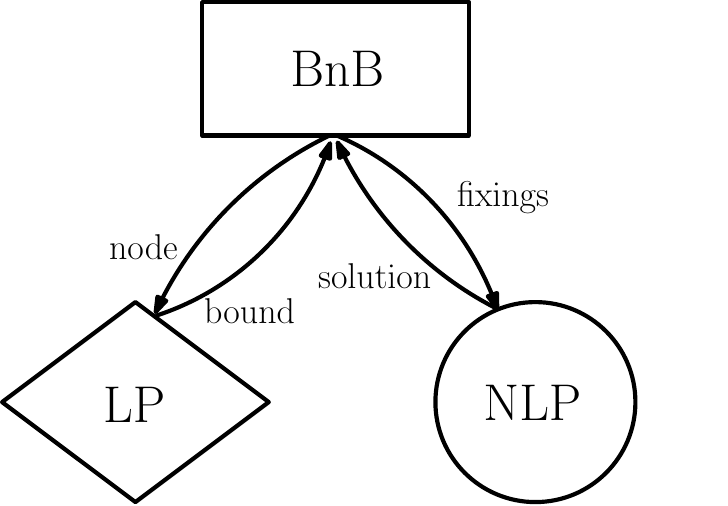}
        \caption{LP-based BnB and outer approximation}
        \label{fig:lpbased}
    \end{subfigure}
       \caption{
        Three main algorithmic frameworks for MINLPs. Diamond blocks represent the nodal relaxations in the given framework.
        \cref{fig:branchnlp} corresponds to a classic BnB framework on top of NLP relaxations, \cref{fig:lpbased} represents
        the mechanism of LP-based MINLP frameworks and outer approximations, \cref{fig:branchwolfe} is our proposed approach with the linearized models
        solved as MIPs within the Frank-Wolfe algorithm on top of which we perform a branch-and-bound.
        \revisionTwo{Leveraging a fully-fledged MIP solver also makes Boscia directly benefit from their algorithmic improvements.}
        }
       \label{fig:minlpframeworks}
\end{figure}

\section{Nonlinear branch-and-bound over integer hulls with linear oracles}\label{sec:integerhull}

In this section, we present the main paradigm for our framework.
Our approach is primarily based on branch-and-bound, the feasible region handled at each node is identical except for
updated fixings and bounds on the variables that must take integer values.
At each node, we solve a convex relaxation with updated local bounds for the integer variables using a FW-based algorithm providing a relaxed primal solution and lower bound.
This family of algorithms only requires a gradient oracle for the objective function $f$, and not a full analytical expression nor Hessian information.
It accesses the feasible region $P$ through the LMO, which, given a direction $\vd$, solves $\vvv \leftarrow \argmin_{\vx\in P} \langle \vd, \vx \rangle$, where $\vvv$ is an extreme point or vertex of the feasible set.
This means in particular that the feasible region does not have to be given in closed form, e.g., with constraint matrix and right-hand side, as long as the LMO is well-defined with a suitable algorithm implementing it.
Instead of optimizing over the continuous relaxation of Problem~\eqref{prob:definition} with updated bounds at each node,
we optimize over the convex hull of feasible solutions at each node. For that purpose, the LMO called within the FW procedure
is the MIP solver with bounds updated for the node and objective given by the gradient information at the current solution.
This design choice is illustrated in \cref{fig:branchingparadigm}.
\begin{figure}[ht]
    \centering
    \begin{subfigure}[t]{0.35\textwidth}
        \centering
        \includegraphics[width=\textwidth]{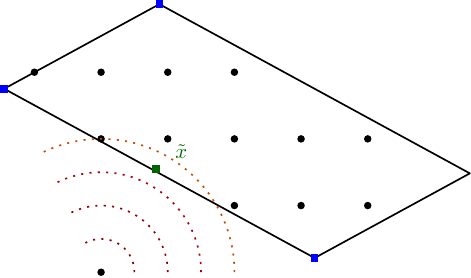}
        \caption{Baseline problem}
        \label{fig:baselineprob}
    \end{subfigure}
    \begin{subfigure}[t]{0.35\textwidth}
        \centering
        \includegraphics[width=\textwidth]{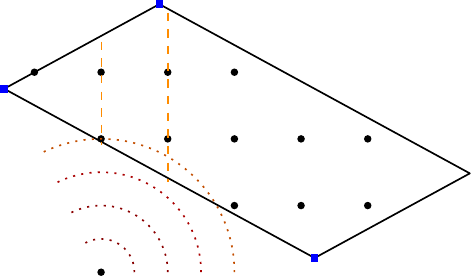}
        \caption{Branching over continuous relaxation}
        \label{fig:branchcontinuous}
    \end{subfigure}
    \begin{subfigure}[t]{0.35\textwidth}
        \centering
        \includegraphics[width=\textwidth]{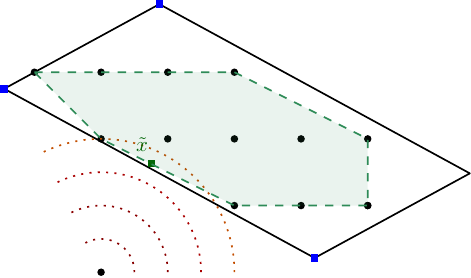}
        \caption{Optimizing over the convex hull}
        \label{fig:convexhull}
    \end{subfigure}
    \begin{subfigure}[t]{0.35\textwidth}
        \centering
        \includegraphics[width=\textwidth]{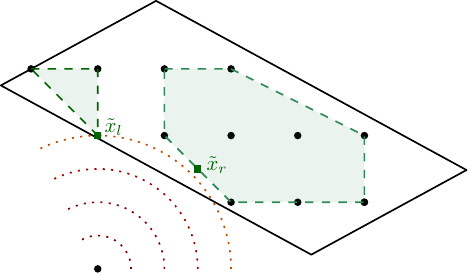}
        \caption{Branching over the convex hull}
        \label{fig:convexhull2}
    \end{subfigure}
\caption{
Branching over the convex hull. \cref{fig:baselineprob} shows the baseline problem with the level curves of the objective, polyhedron, the optimum over the relaxation $\tilde{x}$,
and examples of potentially active vertices for a near-optimal solution (the top vertex is dropped in an optimal solution).
Branching over the continuous relaxation is shown in \cref{fig:branchcontinuous}. Our approach is shown in \cref{fig:convexhull} and \cref{fig:convexhull2},
optimizing over the convex hull with the help of the MIP solver. Branching only once results in an optimal solution 
in the left part in \cref{fig:convexhull2}, with the right part being pruned once $\tilde{x}_l$ is found.
}
\label{fig:branchingparadigm}
\end{figure}

The main algorithm we propose is presented in \cref{alg:boscia}, containing at its core a BnB procedure
with approximate subproblems.
The procedure \texttt{near-optimal\_relaxation\_solve} solves the relaxation to approximate optimality with a specified FW gap
and returns a tuple including the relaxation $\hat{\vx}$, FW gap $g$, active and shadow sets $\mathcal{A}$, $\mathcal{S}$,
and set of primal solutions $\mathcal{H}$ computed through heuristics of our methods or of the MIP solver.
More details on the subproblem solution method are provided in \cref*{app:bpcg}.
$\texttt{best\_bound\_node}(\mathcal{N})$ returns the node with the lowest dual bound.

\begin{algorithm}
\caption{Boscia algorithm for Problem~\eqref{prob:definition} }\label{alg:boscia}
\begin{algorithmic}[1]
\Require Primal-dual tolerance $\delta$, feasible set \Xset as a boundable LMO, $J$ the set of integer variables, objective $f$, initial point $\vvv_0$, FW gap tolerance $\{\varepsilon_t\}_{t\geq 0}$.
\State $\vl^{(0)}, \vu^{(0)} \gets \texttt{global\_bounds}(\Xset)$
\State $\hat{\vx}^{(0)} \gets \vvv_0$
\State $\mathrm{UB} \gets f(\vvv_0)$
\State $g_0 \gets \max_{\vvv\in\Xset} \innp{\nabla f(\vx_0)}{\vvv_0 - \vvv}$
\State $n_0 \gets (l^{(0)}, u^{(0)}, \mathcal{A}_0, \mathcal{S}_0, g_0, f(\hat{\vx}^{(0)}) - g_0)$
\State $\mathcal{N}_0 \gets \{n_0\}$
\State $\mathrm{UB} \gets \mathrm{min}\{\mathrm{UB}, \min_{\vvv\in \mathcal{A}_0 \cup \mathcal{S}_0 \cup \mathcal{H}_0} f(\vvv) \}$
\State $\mathrm{LB} \gets f(\hat{\vx}^{(0)}) - g_0$\
\State $t \gets 0$
\While{$\mathrm{UB} - \mathrm{LB} > \delta$ and $\mathcal{N}_t \neq \emptyset$}
\State $n_t \gets \texttt{best\_bound\_node}(\mathcal{N}_t)$
\State $(\vl^{(t)}, \vu^{(t)}, \bar{\mathcal{A}}_t, \bar{\mathcal{S}}_t, \bar{g}_t, \bar{b}_t) \gets n_t$
\State $\bar{\mathcal{N}}_t \gets \mathcal{N}_t \,\backslash\, \{n_t\} $
\State $(\hat{\vx}^{(t)}, g_t, \mathcal{A}_t, \mathcal{S}_t, \mathcal{H}_t)\gets \texttt{near-optimal\_relaxation\_solve}(\vl^{(t)}, \vu^{(t)}, \bar{\mathcal{A}}_t, \bar{\mathcal{S}}_t, \mathrm{UB}, \varepsilon_t)$
\State $\mathrm{UB} \gets \mathrm{min}\{\mathrm{UB}, \min_{\vvv\in \mathcal{A}_0 \cup \mathcal{S}_0 \cup \mathcal{H}_0} f(\vvv) \}$
\State $b_t \gets f(\hat{\vx}^{(t)}) - g_t$
\If{$b_t > \mathrm{UB}$}
\State prune suboptimal node
\ElsIf{$\hat{\vx}^{(t)} \in \Xset$}\Comment{integer point, no further branching}
\State $\mathrm{UB} \gets \mathrm{min}\{\mathrm{UB}, f(\hat{\vx}^{(t)})\}$
\State close node
\Else\Comment{fractional node to split}
\State $\hat{\vl}^{(t)}, \hat{\vu}^{(t)} \gets \texttt{dual\_bound\_tightening}({\vl}^{(t)}, {\vu}^{(t)}, \hat{\vx}^{(t)}, \nabla f(\hat{\vx}^{(t)}), \mathrm{UB}, b_t)$
\State $j \in \{j \in J \,|\, \hat{\vx}^{(t)}_j \notin \mathbb{Z}\}$\Comment{find variable to branch on}
\State $\vl^{(t)}_r\ \gets \hat{\vl}^{(t)}\ + \mathbf{e}_{j} \left(\lceil \hat{\vx}^{(t)}_j \rceil - \hat{\vl}^{(t)}_j\right) $ \Comment{compute bounds for left and right node}
\State $\vu^{(t)}_l \gets \hat{\vu}^{(t)} + \mathbf{e}_{j} \left( \lfloor \hat{\vx}^{(t)}_j \rfloor - \hat{\vu}^{(t)}_j  \right) $
\State $\mathcal{A}_l, \mathcal{S}_l, \mathcal{A}_r, \mathcal{S}_r \gets \texttt{partition\_vertices}(\mathcal{A}_t, \mathcal{S}_t, j)$
\State $n_l \gets (\vl^{(t)}, \vu^{(t)}_l, \mathcal{A}_l, \mathcal{S}_l, {g}_t, b_t)$
\State $n_r \gets (\vl^{(t)}_r, \vu^{(t)}, \mathcal{A}_r, \mathcal{S}_r, {g}_t, b_t)$
\State $\mathcal{N}_{t+1} \gets \bar{\mathcal{N}}_t \cup \{n_l, n_r\}$
\EndIf
\State $t \gets t + 1$
\EndWhile
\end{algorithmic}
\end{algorithm}

\paragraph{Constraint representability with integer hulls.}
Optimizing over the integer hull at each node automatically makes all vertices computed across FW iterations feasible for the original problem.
The feasible region itself is tightened, greatly reducing the tree size for problems that have a weak continuous relaxation.
This approach also allows us to integrate other nonconvex constraints
as long as the integer points of their convex hull respect the combinatorial constraint, an assumption we formalize below.

\begin{assumption}[non-extreme integer points feasibility]\label{A-CONV}
    Let \Xset be the feasible set, $\overline{\Xset}$ its continuous relaxation,
    $\mathrm{conv}(\Xset)$ the convex hull of \Xset, $\mathbb{Z}_J$ the set of points respecting integrality constraints, and $(\mathbf{l},\mathbf{u})$ the local bounds at the node.
    We assume that:
    \begin{align*}
        x \in \mathrm{conv}(\Xset) \cap \mathbb{Z}_J \cap \left[\mathbf{l},\mathbf{u}\right] \Rightarrow x \in \Xset.
    \end{align*}
    In other words, there exists no integer point in the convex hull of \Xset which is not feasible.
\end{assumption}
Assumption~\ref{A-CONV} ensures that if the relaxed solution obtained at a node is integer-feasible, it is feasible for \Xset since it is obtained as a convex combination of extreme points of \Xset intersected with local bounds.
On a relaxed solution, the algorithm thus branches if $x\notin\Xset$ or closes the node with an admissible solution, it cannot terminate on an integer point that would not be feasible for \Xset,
ensuring its exactness and global convergence.
We illustrate this condition with two families of constraints that respect this condition and one that does not.
We also highlight that the assumption remains extremely flexible, capturing essentially most combinatorial structures and allowing Boscia to run even on some nonconvex MINLPs.

\paragraph*{Indicators.}
Indicator constraints are nonconvex constraints of the form $(z, \vx) \in \{0,1\} \times \mathbb{R}^n: z = 1 \Rightarrow \innp{\mathbf{a}}{\vx} \leq b$.
They can represent, among others, logical, disjunctive, or SOS1 constraints and are crucial in MIP and MINLP applications.
Vertices of the integer hull respect the binary constraint, thus $z\in\left\{0,1\right\}$ on a vertex.
If for a given solution, vertices of the active set all have $z=0$ or all have $z=1$,
the indicator constraint is respected; otherwise, the relaxed solution $\hat{z}$ is fractional and will be branched on.
This means that our framework natively handles indicator constraints by branching on the binary variable only.
This extends to indicators with convex functions or \textit{convex superindicator} expressed as:
\begin{align*}
(z, \vx) \in \{0,1\} \times \mathbb{R}^n: z = 1 \Rightarrow g(\vx) \leq 0   
\end{align*}
with $g$ convex. Indeed, vertices are feasible by definition, so $z=0$ or $z=1, g(\vvv) \leq 0$ for any vertex $\vvv$ in the active set.
If $z=1$ in the solution, convexity implies that $g(\sum_k \lambda_k \vvv_k) \leq 0$
with $\lambda_k$ weights of the vertices in the final convex combination.

\paragraph*{Logical constraints.}

Other constraints that are directly representable are logical constraints (\texttt{or}, \texttt{and}, and \texttt{xor} constraints) or more broadly arbitrary constraints of pure binary variables.
Although these can typically be represented with linear constraints, most MIP solvers include specialized forms for these
which helps presolving, propagation, conflict analysis and separation.

\paragraph*{Special Ordered Sets.}

Special Ordered Sets of type 1 (SOS1) and 2 (SOS2) are defined for a vector of variables $\vx \in \mathbb{R}^n $as:
\begin{align*}
    & \vx \in \mathrm{SOS1} \Leftrightarrow \|\vx\|_0 \leq 1,\\
    & \vx \in \mathrm{SOS2} \Leftrightarrow \|\vx\|_0 \leq 2, x_i x_j = 0 \; \forall (i,j) \in [n]^2, |i-j| > 1.
\end{align*}
For SOS1 or SOS2 constraints over non-binary variables,
integrality of discrete variables is not sufficient to ensure feasibility, they would therefore require additional branching or reformulations.
An example with an SOS1 constraint violating the assumption is illustrated in \cref{subfig:notholdscomb}.

Our framework can also incorporate nonlinear constraints in the feasible set \Xset, as long as Assumption~\ref{A-CONV} is respected.
From the computational perspective, however, we must note that this makes sense only if the BLMO subproblems remain inexpensive compared to the original problem.
We illustrate examples where Assumption~\ref{A-CONV} may or may not hold in \cref{fig:assumption}.

\begin{figure}
	\centering
	\begin{subfigure}[t]{0.32\textwidth}
		\centering
		\includegraphics[width=0.8\textwidth]{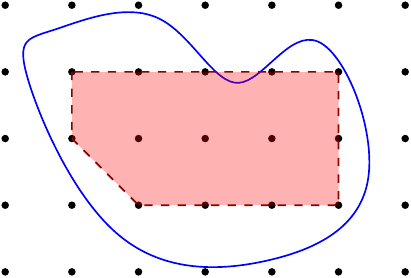}
        \caption{Nonconvex problem where Assumption~\ref{A-CONV} holds and $\mathrm{conv}(\Xset) \not\subseteq \overline{\Xset}$.}\label{subfig:holds}
	\end{subfigure}
	\hfill
	\begin{subfigure}[t]{0.32\textwidth}
		\centering
		\includegraphics[width=0.8\textwidth]{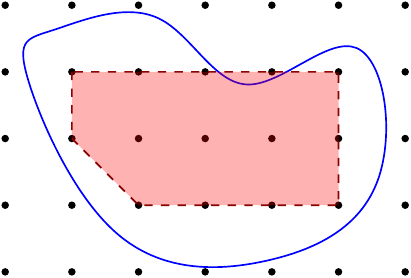}
        \caption{Example with a nonconvex problem where Assumption~\ref{A-CONV} does not hold.}\label{subfig:notholds}
	\end{subfigure}
	\begin{subfigure}[t]{0.32\textwidth}
		\centering
		\includegraphics[width=0.8\textwidth]{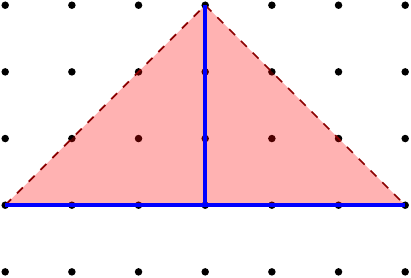}
        \caption{Example with a combinatorial constraint (equivalent to SOS1) where Assumption~\ref{A-CONV} does not hold.}\label{subfig:notholdscomb}
	\end{subfigure}
    \caption{Positive and negative illustrations for Assumption~\ref{A-CONV}. The blue solid curve represents the boundary of $\overline{X}$, the red dashed shaded area is $\mathrm{conv}(\Xset)$.}
	\label{fig:assumption}
\end{figure}

\paragraph{Blended Pairwise Conditional Gradient (BPCG).}
The classic FW algorithm calls an LMO once per iteration. Even though this is acceptable in contexts where
the LMO is inexpensive compared to gradient evaluations or in large dimensions in which storage of vertices is an issue,
it is not the best-suited one when optimizing over polytopes. The first reason is the descent direction
which can be significantly misaligned with the negative of the gradient.
The second and essential one in our context is that the LMO is particularly expensive (consisting of solving a MIP).
Instead, we employ methods that exploit the \emph{active set} decomposition of the iterate $\vx$, i.e.,~its representation as a convex combination of vertices obtained from the LMO $\vx_t := \sum_k \lambda_k \vvv_k$.
Typical algorithms leveraging this representation are the Away Frank-Wolfe, Pairwise Conditional Gradient, and Blended Conditional Gradient algorithms \citep{guelat1986some,lacoste2015global,braun2019blended}.
The variant we focus on is the recently proposed
Blended Pairwise Conditional Gradient (BPCG) algorithm \citep{tsuji2021sparser} which blends Frank-Wolfe steps in which
a vertex is added to the active set with pairwise steps which do not require any LMO call
and only perform inner products between vertices and the gradient.
Empirically, BPCG produces iterates with a smaller active set across all iterations compared to other FW algorithms.
We use more specifically a modification of the lazified BPCG from \citep{tsuji2021sparser} with the novel adaptive step size from \cite{pokutta2024frank}
implemented in \texttt{FrankWolfe.jl}; the algorithm is presented in detail in the supplement.

\paragraph{Active and shadow set branching.}
We accelerate the solving process of individual nodes through warm starts.
At the end of each BPCG call, the optimal solution is given as a convex combination of
vertices in an active set. All these vertices are obtained by calls to the BLMO and
are thus extreme points of the integer hull.
When branching, the vertices can be partitioned to form
starting active sets for the left and right children. As the relaxed solution is fractional
in the variable branched on, it always contains at least one vertex for each of the children.
The weights are also transferred to each child for the corresponding vertex and renormalized.
BPCG can also drop vertices from the active set, in a single run of the algorithm,
a dropped vertex is not part of the optimal active set and will not be re-added.
However, we solve subproblems across different nodes.
As such, a dropped vertex can be relevant for another subproblem further down the tree when the bounds have been updated, adding computational burden by redundant calls to the BLMO.
We therefore maintain a \emph{shadow set}, i.e.,~a set of discarded vertices collected within each node when they are removed from the active set.
The shadow set can be partitioned in a fashion similar to the active set.
At all non-root nodes, an extra lazification layer is added to BPCG: if no suitable pair of vertices can be found in the pairwise step,
the algorithm searches the shadow set for a vertex with a suitable inner product with the gradient, using identical criteria as in the lazy BPCG algorithm.
Only if this additional step fails does the algorithm perform a BLMO call computing a new vertex.
By propagating both the active and shadow set, we significantly reduce the number of calls to the BLMO, ensuring we never recompute a given vertex twice within a run for our algorithm.

\section{Branch-and-bound techniques for error-adaptive convex relaxations}\label{sec:TechniquesErrorAdaptiveness}

In this section, we present the techniques derived from branch-and-bound techniques that can be used in our framework.
In particular, we highlight the advantages of the error-adaptivity of BPCG, i.e.,~the property of the computational cost decreasing
gradually when the error tolerance increases, and how the various components of modern MIP solvers can be used to enhance the higher-level branch-and-bound process.
MIP solvers are implementations of complex algorithms with several key components including presolving, heuristics, bound tightening, and conflict analysis.
This is usually also a main motivation for single-tree MINLP solvers handling nonlinearities as a special case of separation within a MIP solving process.
We can however also leverage this additional information about the feasible set within our framework by transferring it across nodes.

\paragraph*{FW gap-based termination.}
Frank-Wolfe methods produce at each iteration a primal feasible iterate and a so-called FW gap upper-bounding
the optimality gap. They feature many inexpensive iterations
(unlike e.g.,~interior points which require few expensive iterations),
with a dual bound gradually increasing with iterations.
We leverage this fact to adapt the solving process at each node.
Whenever the dual bound reaches the objective value of the best incumbent (or becomes close to
the gap tolerance), the node can be terminated early.
A significant part of the solving process can be avoided in this manner which is not necessarily
the case when nodes are processed with other nonlinear solvers.
Pushing this further, the solution process of a node can be stopped at any point to produce a dual bound,
even if the solution is not an exact but approximate optimum.
The only strict requirement is that the dual bound after solving the current node becomes high enough to
yield progress in the overall tree dual bound.

\paragraph*{Tree state-dependent termination and evolving error.}
We implement additional termination criteria for node processing which do not guarantee the node should be pruned
but that save on the total number of iterations.
One of them is the number of open nodes with a lower bound (obtained from the parent)
that is lower than the dual bound of the node being processed during the BPCG run.
This set of more promising nodes will be processed before the children of the current node are.
If the number of these more promising nodes is high enough,
there is a higher probability for the children of the current node not to be processed at all, for instance,
if an incumbent is found elsewhere that prunes them.
We also add an adaptive Frank-Wolfe gap criterion, increasing the precision with the depth in the BnB tree:
$\varepsilon_n = \varepsilon_0\, \rho^n$,
with $\rho \in \left]0, 1\right]$ where $n$ is the depth of the current node, starting from the root at depth 0,
and clipping when reaching a sufficient depth.
With some convex solvers, using a reduced precision could slow down the search, as approximate solutions might exhibit a much higher fractionality when a weaker stopping criterion is used, requiring more branching. In the case of BPCG however, the iterates are a convex combination of a few integer extreme points.
Furthermore, if the optimal solution at a node is an extreme point itself, the process converges rapidly by dropping
the other extreme points from the active set once the optimal one has been added.
As such, solutions obtained from low-precision runs do not necessarily exhibit a higher fractionality.

\paragraph{Strong branching.}
When tackling large discrete problems, the choice of variables to branch on can yield drastic differences
in how fast the lower and upper bounds evolve and the overall size of the branch-and-bound tree. A powerful technique to estimate the lower bound increase when branching on a given variable is \emph{strong branching}
which solves the children subproblems of a given node for all candidate variables, selecting the variable to branch on
that improves the lowest lower bound across children. Other techniques like pseudo-cost branching or recent
machine learning approaches
try to construct surrogate models to avoid solving the multiple convex subproblems induced by strong branching \citep{nair2020solving}.

We propose a new family of branching techniques that leverage the properties of FW algorithms to obtain a partial estimate
of the lower bound improvement while greatly reducing the cost.
In our context, strong branching with nonlinear subproblems over the integer hull would be too costly to perform variable selection. However, we can
\begin{inparaenum}[(a)]
\item relax the strong branching over the integer hull to the continuous relaxation (solving LPs for the LMO instead of MIPs),
\item run few iterations of the subproblem (and/or set a very high FW gap tolerance).
\end{inparaenum}
In the limit case, on the one extreme end this corresponds to performing a single FW iteration for the strong branching estimation and the optimal value of a single linear problem is used to select the branching variable. On the other extreme end, the complete high-accuracy solve would correspond to using the continuous relaxation of the node. By carefully limiting the precision we can interpolate between these two regimes.

\revision{This new partial strong branching technique will not be beneficial for all problems, i.e.,~it can still be too computationally expensive. 
Thus, we also developed a hybrid strong branching technique 
which performs partial strong branching only to a certain depth in the Branch-and-Bound tree and then switches to the default strategy, i.e.,~most fractional branching. 
The idea is that strong branching is effective mainly for the first couple of branching decisions.}

\paragraph*{Dual fixing and tightening.}
In a lot of subproblems, many variables are at their (local) bounds. We can then leverage convexity and primal solutions to deduce safe \emph{dual fixings}
or at least \emph{dual tightening} of the bounds,
a procedure originally established for MIPs going back to Dantzig (see \citet{grotschel2008george} and \citet{achterberg2020presolve}) and extended to mixed-integer conic problems in \citet{gally2018framework} or MINLPs in \citet{mahajan2021minotaur}.
Dual fixing in these frameworks is a presolving technique, applied in a static manner to a problem to reduce its size before the solving phase.
Our approach is slightly different though as we do not explicitly use dual solutions as customary but rather convexity of the objective and availability of the FW gap.
These connect to a property of Frank-Wolfe methods developed in \citet{braun2021dual}, namely that for every Frank-Wolfe vertex $\vvv$ in a polytope, the dual gap is equal to the inner product of the dual variables with the slack of the polytope, in the case dual solutions exists.
Our approach is more general and even works when dual solutions are not readily available e.g., when the BLMO is given via a MIP.
We confine ourselves here to tightening the upper bound (potentially fixing it to the lower bound); the case for tightening the lower bound is analogous.

\begin{theorem}\label{the:dualtightening}
Let us assume that the bounds are $\left[\vl,\vu\right] \supseteq \Xset$ and that we have a relaxed solution $\vx^{(t)}$ and a variable $j \in J$
such that $\vx^{(t)}_j = \vl_j$ and $\nabla f(\vx^{(t)})_j \geq 0$.
Then, if there exists $M \in \{1, \dots, \vu_j-\vl_j\}$, such that:
\begin{align}\label{eq:dualtightening}
    M \nabla f(\vx^{(t)})_j > \mathrm{UB} - f(\vx^{(t)}) + g_{-j}(\vx^{(t)})
\end{align}
with $g_{-j}(\cdot)$ the masked Frank-Wolfe gap of coordinate $j$:
\begin{align*}
g_{-j}(\vx) := \max_{\vvv \in \Xset} \sum_{k \neq j} (\nabla f(\vx^{(t)}))_k (\vx^{(t)} - \vvv)_k,
\end{align*}
then the upper bound can be tightened to $\vx^{(t)}_j \leq \vl_j + M - 1$.
\end{theorem}
\begin{proof}
By convexity, we have:
\begin{align*}
& \innp{\nabla f(\vx^{(t)})}{\vx - \vx^{(t)}} \leq  f(\vx) - f(\vx^{(t)}) \,\,\forall \vx \in \Xset.
\end{align*}
If $\vx_j > \vl_j$, we can rewrite $\vx_j = \vl_j + M,$ with $M \in \{1, \dots, \vu_j-\vl_j\}$. Then:
\begin{align*}
    & \nabla f(\vx^{(t)})_j (\vx - \vx^{(t)})_j + \sum_{k\neq j} \nabla f(\vx^{(t)})_k (\vx - \vx^{(t)})_k \\
    =& \, M \nabla f(\vx^{(t)})_j + \sum_{k\neq j} \nabla f(\vx^{(t)})_k (\vx - \vx^{(t)})_k \leq  f(\vx) - f(\vx^{(t)}) \quad \forall \vx \in \Xset\\
\end{align*}
Let us denote $\Xset_M := \{\vx \in \Xset \,|\, \vx_j \geq \vl_j + M\}$ for $M \in \{1, \dots, \vu_j-\vl_j\}$.
Taking the minimum over all solutions $\vx \in \Xset_M$, we obtain:
\begin{align*}
    \min_{\vx \in \Xset_M} f(\vx) - f(\vx^{(t)}) \geq\, & \min_{\vx \in \Xset_M} (\vx - \vx^{(t)})_j \nabla f(\vx^{(t)})_j + \min_{\vx \in \Xset_M} \sum_{k\neq j} \nabla f(\vx^{(t)})_k (\vx - \vx^{(t)})_k \\
    \geq \, &  M \nabla f(\vx^{(t)})_j - \max_{\vx \in \Xset_M} \sum_{k\neq j} \nabla f(\vx^{(t)})_k (\vx^{(t)} - \vx)_k \\
    \geq\, & M \nabla f(\vx^{(t)})_j - \max_{\vx \in \Xset} \sum_{k\neq j} \nabla f(\vx^{(t)})_k (\vx^{(t)} - \vx)_k = M \nabla f(\vx^{(t)})_j - g_{-j}(\vx^{(t)}).
\end{align*}
The last inequality holds since $\Xset_M \subseteq \Xset$.
If an optimal solution is in $\Xset_M$, then
\begin{align}\label{eq:validcomplement}
    \mathrm{UB} - f(\vx^{(t)}) \geq \min_{\vx \in \Xset} f(\vx ) - f(\vx^{(t)}) = \min_{\vx \in \Xset_M} f(\vx ) - f(\vx^{(t)}) \geq M \nabla f(\vx^{(t)})_j - g_{-j}(\vx^{(t)})
\end{align}
Therefore, if \eqref{eq:validcomplement} does not hold, i.e.:
\begin{align*}
    & M \nabla f(\vx^{(t)})_j > \mathrm{UB} - f(\vx^{(t)}) + g_{-j}(\vx^{(t)}),
\end{align*}
we can deduce $\vx_j \leq \vl_j + M - 1$.
\end{proof}
We exploit \cref{the:dualtightening} in two ways within our algorithm.
First, we strengthen the local bounds at nodes, using the final relaxed solution.
Because we incrementally add bounds on integer variables as the algorithm runs further down in the tree, more variables terminate at their bounds, leading to further fixings.
We note that tightening reduces the upper bound when a variable is at its lower bound, which implies that all vertices of the active set are at the same bound,
and that the active set vertices can be partitioned as described in \cref{sec:integerhull},
unlike shadow set vertices which may be discarded by the tightening.
Tightening not only reduces the number of nodes to solve the problem but may also accelerate BLMO calls through the reduction of the integer variable domains.
We also provide a stronger version when a strong convexity constant $\mu$ is known for the objective in \cref{the:dualstronger} of \cref*{app:bounds}.

Second, we exploit the root-node relaxation $\bar{\vx}$, for which we store the gradient $\nabla f(\bar{\vx})$ and the dual bound $f(\bar{\vx}) - g(\bar{\vx})$.
Whenever the upper bound decreases with a new primal solution, we can compute the masked Frank-Wolfe gap, verify whether the inequality \eqref{eq:dualtightening} holds,
and create a set of tighter bounds that are used to preemptively prune out nodes.
\revision{This method may be extended in future work to incorporate other tightening techniques, such as the Lagrangian bound from optimization-based bound tightening from \citet{gleixner2017three}.}

\paragraph*{Strong convexity node bound tightening.}
For several important classes of problems, a strong convexity parameter $\mu$ is known for the objective.
In such cases, we can exploit the minimum increase of the objective from a current fractional solution $\hat{\vx}$
to preemptively prune nodes or increase their bound.
We will illustrate this by branching down on (i.e.,~flooring) the fractional variable $j$.
The same reasoning applies when branching up (i.e.,~ceiling).
Let us denote $\vx^*_l$ the optimal solution of the subproblem created after branching on the variable, and $\hat{J} \subseteq J \backslash \{j\}$ the set of other integer variables at fractional values.
Our goal is to derive a lower bound for $f(\vx^*_l)$ which potentially allows us to prune the node.
\begin{align*}
    f(\vx^*_l) & \geq f(\hat{\vx}) + \frac{\mu}{2} \norm{\vx^*_l - \hat{\vx}}^2_2 + \innp{\nabla f(\hat{\vx})}{\vx^*_l - \hat{\vx}}\\
    & \geq f(\hat{\vx}) + \frac{\mu}{2} (\hat{\vx}_j - \lfloor\hat{\vx}_j\rfloor )^2 + \frac{\mu}{2} \sum_{k\in\hat{J}} \min \left\{ (\hat{\vx}_k - \lfloor\hat{\vx}_k\rfloor )^2, (\lceil\hat{\vx}_k\rceil - \hat{\vx}_k )^2  \right\}   - \innp{\nabla f(\hat{\vx})}{\hat{\vx} - \vx^*_l} \\
    & \geq f(\hat{\vx}) + \frac{\mu}{2} (\hat{\vx}_j - \lfloor\hat{\vx}_j\rfloor )^2 + \frac{\mu}{2} \sum_{k\in\hat{J}} \min \left\{ (\hat{\vx}_k - \lfloor\hat{\vx}_k\rfloor )^2, (\lceil\hat{\vx}_k\rceil - \hat{\vx}_k )^2  \right\} - \max_{\vvv\in\Xset} \innp{\nabla f(\hat{\vx})}{\hat{\vx} - \vvv} \\
    & = f(\hat{\vx}) + \frac{\mu}{2} (\hat{\vx}_j - \lfloor\hat{\vx}_j\rfloor )^2 + \frac{\mu}{2} \sum_{k\in\hat{J}} \min \left\{ (\hat{\vx}_k - \lfloor\hat{\vx}_k\rfloor )^2, (\lceil\hat{\vx}_k\rceil - \hat{\vx}_k )^2  \right\} - g(\hat{\vx}).
\end{align*}
We also note that the strong convexity parameter can be computed only on the integer variables, which can result in a stronger value of $\mu$ than on the whole function.
Furthermore, strongly convex terms can be added to binary variables when they do not change the value of $f$ on $\{0,1\}$,
for instance by replacing a linear expression $\sum_{i} \vc_i \vx_i$ with $\vc \geq 0$ by a quadratic one $\sum_{i} \vc_i \vx_i^2$.

The obtained improved bounds have several use cases;
they can be used directly to find a branching decision removing one of the two branches or maximizing a lower bound on the improvement.
Given the form of the inequality, the branching decision maximizing the lower bound improvement is equivalent to most fractional branching.
We can also increase the lower bound of the node as the maximum of the bound improvement over all variables (i.e., the bound obtained from the most fractional branching).

The technique is complementary to dual bound tightening: dual bound tightening can be leveraged when a variable is at one of its bounds, tightening the variable domain,
while strong convexity improves the dual bound based on fractional variables.
A more generic version of the technique only assuming \emph{sharpness} (Hölderian error bound) of the function is developed in \cref*{app:bounds}.

\section{Computational experiments}\label{sec:CompuationalExperiments}

We use the \texttt{FrankWolfe.jl} framework \citep{besanccon2022frankwolfe,fwreleasepaper} to solve the node subproblems.
All features specific to our BnB framework are implemented in the open-source Julia package \package{}, the documentation of which is available at \url{https://zib-iol.github.io/Boscia.jl/stable/}.
\revisionTwo{The experiments were performed with the package v0.1.23.}
The BnB core structures are implemented in \texttt{Bonobo.jl} and
\package{} supports all MIP solvers through \texttt{MathOptInterface} \citep{legat2020mathoptinterface}.
The underlying MIP solver used in the experiments is \texttt{SCIP} \revision{0.11.14} \citep{bestuzheva2021scip} run single-threaded,
strong branching uses \texttt{HiGHS} \revision{v1.9.0} \citep{Huangfu_Parallelizing_the_dual_2018} as the LP solver.
\revision{All experiments were carried out in a 64-core compute node equipped with an Intel Xeon Gold 6338 2.00GHz and 1024GB RAM. The time limit was half an hour, i.e.,~1800 s.}
We use Julia \revision{1.10.2}. The package versions used are \texttt{FrankWolfe.jl} \revision{v.0.3.4}, \texttt{Bonobo.jl} \revision{v0.1.3}, \texttt{SCIP.jl} \revision{v0.11.14}.

\subsection*{Problem classes}
These problems and additional examples are available in the package repository.\newline

\emph{Sparse Regression Problems.} Sparsity is a desirable property for
prediction models for reasons of robustness, explainability, computational efficiency or other underlying motivations.
Our framework allows for solving all cardinality-constrained regression models including
linear regression, sparse Poisson regression, and logistic regression, as long as the loss function is convex and its gradient available.
\revision{In \package{}, there are examples of regression models where the predictor coefficients themselves are constrained to take integer values.}
The flexible computational model further allows for richer constraint structures encoding more domain knowledge than sparsity alone.
In particular, we can encode group sparsity and constraints similar to group lasso \citep{friedman2010note}.

\emph{Sparse Tailed Regression.} Another example of constraint structure we can directly encode is thresholded or two-tailed cardinality-constrained optimization
as considered in \citet{ahntractable}: given a compact convex set \Xset, a convex loss $f$ and a cardinality penalty $\lambda$,
the sparse tailed cardinality-penalized model they propose is:
\begin{align*}
    \min_{\vx \in \Xset} f(\vx) + \lambda \|\max \{|\vx_i| - \tau_i, 0\}_{i\in \left[n\right]}\|_0 + \mu \|\vx\|_2^2
\end{align*}
which can be reformulated using indicator constraints as:
\begin{align}\label{prob:threscard}\tag{TCMP}
\min_{\vx, \mathbf{z}, \mathbf{s}}\,\,\, & f(\vx) - \lambda \sum_i \mathbf{z}_i + \mu \|\vx\|_2^2\\
\text{s.t.} \,\, & \mathbf{z}_i = 1 \Rightarrow \mathbf{s}_i \leq 0 \nonumber\\
& \mathbf{s}_i \geq \vx_i - \tau_i \nonumber\\
& \mathbf{s}_i \geq -\vx_i - \tau_i \nonumber\\
& \vx \in \Xset, \mathbf{z} \in \{0,1\}^n, \mathbf{s} \in \Xset \cap \mathbb{R}^n_+.\nonumber
\end{align}

When the indicator variable $\mathbf{z}_i$ equals one and the slack $\mathbf{s}_i$ are zero, $\vx_i$ is constrained in the interval $\left[-\tau_i, \tau_i\right]$,
the model hence penalizes the number of predictor values above a threshold in absolute value, while ignoring all smaller values.
We can provide guaranteed optimal, near-optimal solutions, relaxations and bounds for Problem~\eqref{prob:threscard}
or a cardinality-constrained version without additional assumptions, replacing the complementarity constraints of
\citet{ahntractable} with polyhedral and combinatorial constraints.


\emph{Portfolio Optimization.} We revisit the example of \citet{buchheim2018frank},
selecting a portfolio with budget $b$, integrality requirements on shares for some assets, and with a
generic convex differentiable risk penalty term $h(\cdot)$:
\begin{align*}
\min_{\vx}\,\, h(\vx) = \, \vx^T M \vx - \innp{r}{\vx}\,\,\,
\text{s.t. }\, \innp{c}{\vx} \leq b,\, \vx_j \in \mathbb{Z}\, \forall j \in J.
\end{align*}
\revision{We investigate both the case where all variables are integers, the Pure Integer Portfolio Problem, and where only half of the variables are integers, the Mixed Integer Portfolio Problem.}

\emph{MIPLIB Instances.} Our framework can process standard instance formats out of the box
through the \texttt{MathOptInterface} library \citep{legat2020mathoptinterface}.
We take instances from the MIPLIB 2017~\citep{miplib} with a bounded feasible region and set a convex quadratic objective 
that minimizes the sum of distances to several random vertices. This ensures a sufficient curvature and a continuous optimum in the interior,
avoiding a single MIP call to be optimal.
The instances used are \texttt{pg5\_34}, \texttt{ran14x18-disj-8}, \texttt{neos5} and \texttt{22433}.
\revision{The purpose of the MIPLIB instances is to have problems with computationally expensive BLMOs.}

\subsection*{Results}

To evaluate \package, we compare its performance on the aforementioned problems to four other solution methods.

\paragraph*{\texttt{BnB-Ipopt}.} First, we have implemented a nonlinear BnB approach using the \texttt{Ipopt} interior point solver for the nonlinear node relaxations.
\revision{For NLPs of moderate size, \texttt{Ipopt} is one of the most robust and well-tested open-source interior point solvers.}
\revision{The purpose of comparisons against \texttt{BnB-Ipopt} is primarily to show the impact of FW-based relaxations for different types of instances compared to an interior point method.
We use \texttt{Bonobo.jl} as the base BNB framework on top of Ipopt, which is the same one as the one used for \package{}.}

\paragraph*{\texttt{SCIP+OA}.} The next solution method is an Outer Approximation approach utilizing the open source solver \texttt{SCIP} \revision{with a custom constraint handler to generate the cutting planes from the gradient of the objective.}
\revision{Note that \texttt{SCIP} could be used as-is but would require the objective function in an expression format. 
The purpose of this setup is to keep the objective function as a first-order oracle.}

\paragraph*{} The implementation of the two previous setups can be found on the \package{} repository in the \\ \verb|rerun_experiments| branch (\url{https://github.com/ZIB-IOL/Boscia.jl/tree/rerun_experiments}).

\revision{\paragraph*{\texttt{Pavito}.} Apart from our customized approaches, we also want to compare \package{} to established solvers. 
One of those is the mixed-integer convex solver \texttt{Pavito} \citep{pavito} written in \texttt{Julia}.
As \texttt{SCIP+OA}, it is an Outer Approximation approach using \texttt{SCIP} as a MIP solver and \texttt{Ipopt} as the NLP solver.}

\revisionTwo{\paragraph*{\texttt{SHOT}.} Another well established Outer Approximation solver is \texttt{SHOT} \citep{lundell2022supporting,kronqvist2016extended}.
Observe that \texttt{SHOT} requires the nonlinear objective in expression form and stores it in expression trees.
Thus, it explicitly knows the objective and can utilize this knowledge in the problem-solving process.
In contrast, \package{} only has access to the objective and its gradient via an oracle.
The oracle structure is useful in cases where the expression format is hard to generate.
An example is the A-criterion for Optimal Design Problems which amounts to minimizing the trace of the inverse of a matrix, see \citet{design_of_experiments_boscia_23}.
This difference in access and required information on the objective should be noted when comparing the performance of \package{} to \texttt{SHOT}.

In addition, \citet{design_of_experiments_boscia_23} showcases how combinatorial solution methods can be used as the BLMO within \package{}. 
Another application of the solver is showcased in \citet{SharmaNetworkDesignTrafficFrankWolfe24} which tackles a network design problem for traffic assignment with congestion effects. 
The mixed-integer nonlinear formulation solved with \package{} was compared to polyhedra approximation-based algorithms using \texttt{SCIP} directly as a MINLP solver and mixed-integer conic formulations.
The experiments showed that \texttt{SCIP} could suffer from numerical instabilities in the LP resulting from the polyhedral approximation \citep{SharmaNetworkDesignTrafficFrankWolfe24}.}

The summary of the performance comparison \revisionTwo{of our experiments} can be found in \cref{tab:SummaryOfComparison}. 
For a more detailed comparison for each problems class, see \cref{tab:SummaryByDifficultyMIPLIP,tab:SummaryByDifficultyPortfolio,tab:SummaryByDifficultyRegression,tab:SummaryByDifficultyTailedReg} in \cref{app:DetailedExperiments}. 
Due to the indicator constraints present in the Sparse Tailed Regression Problem, the instances can only be solved by the \texttt{SCIP+OA} framework and \package{}.

\begin{table} \scriptsize
    \centering
    \begin{tabular}{cl HrrH HrrH HrrH HrrH | HrrH} 
        \toprule
        \multicolumn{2}{l}{} & \multicolumn{4}{c}{Boscia} & \multicolumn{4}{c}{B\&B Ipopt} & \multicolumn{4}{c}{SCIP+OA} & \multicolumn{4}{c}{Pavito} & \multicolumn{4}{c}{SHOT}\tabularnewline 
        
        \cmidrule(lr){3-6}
        \cmidrule(lr){7-10}
        \cmidrule(lr){11-14}
        \cmidrule(lr){15-18}
        \cmidrule(lr){19-22}

        Problem & \thead{\# \\ inst.} & \thead{\# \\ solved} & \thead{\% \\ solved} & \thead{Time (s)} & \thead{Relative \\ Gap} & \thead{\# \\ solved} & \thead{\% \\ solved} & \thead{Time (s)} & \thead{Relative \\ Gap} & \thead{\# \\ solved} & \thead{\% \\ solved} & \thead{Time (s)} & \thead{Relative \\ Gap} & \thead{\# \\ solved} & \thead{\% \\ solved} & \thead{Time (s)} & \thead{Relative \\ Gap} & \thead{\# \\ solved} & \thead{\% \\ solved} & \thead{Time (s)} & \thead{Relative \\ Gap}  \\

        \midrule
        \thead{MIPLIB \texttt{22433}}   & 15           & 15         & \textbf{100 \%}        & 5.33       & 0.0            & 15        & \textbf{100 \%}       & 30.32     & 0.0             & 15         & \textbf{100 \%}        & \textbf{1.69} & 0.0  & 15         & \textbf{100 \%}        & 6.64       & 0.0   & 15       & \textbf{100 \%}      & 19.77    & 0.0       \\
        \midrule
        \thead{MIPLIB \texttt{neos5}}  & 15           & 14         & 93 \%         & 30.66      & \textbf{0.01}  & 15        & \textbf{100 \%} & \textbf{19.0}      & \textbf{0.01} & 5          & 33 \%         & 517.4      & 0.04           & 5          & 33 \%         & 460.85     & 0.05  & 15 &\textbf{100 \%} & \textbf{0.51} & \textbf{0.01}    \\
        \midrule
        \thead{MIPLIB \texttt{pg5\_34}} & 15           & 0          & 0 \%          & 1800.0     & 0.02           & 0         & 0 \%         & 1800.07   & Inf           & 0          & 0 \%          & 1800.04    & 0.18           & 0          & 0 \%          & 1800.0     & Inf   & 12       & \textbf{80 \%} & \textbf{685.26} & \textbf{0.01}      \\
        \midrule
        \thead{MIPLIB \texttt{ran14x18-disj-8}} & 15           & 0          & 0 \%          & 1802.86    & 0.09           & 0         & 0 \%         & 1800.16   & Inf           & 0          & 0 \%          & 1800.09    & 0.61           & 1          & \textbf{7 \%}          & \textbf{1797.79}    & 0.01   & 15       & \textbf{100 \%} & \textbf{16.65} & 0.04       \\
        \midrule
        \thead{Poisson \\ Regression} & 120 & 30 & 25 \% & 406.01 & 3.79 & 34 & 28 \% & 516.04 & 0.17 & 116 & \textbf{81 \%} & \textbf{81.74} & 0.4 & 29 & 24 \% & 885.84 & 0.22  & 81 & 63 \% & 297.89 & \textbf{0.04}  \\
        \midrule
        \thead{Pure \\ Integer \\ Portfolio}  & 210          & 158        & 75 \%         & 236.29     & \textbf{1.86}  & 196       & \textbf{93 \%}        & \textbf{26.97} & 1.89          & 30         & 14 \%         & 1162.41    & 54.59          & 10         & 5 \%          & 1656.61    & 0.01    & 173      & 82 \% & 266.83           & 1.89       \\
        \midrule
        \thead{Mixed \\ Integer \\ Portfolio} & 210          & 199        & 95 \%         & 19.47      & \textbf{0.72}  & 209       & \textbf{100 \%} & \textbf{2.82}    & 38.69         & 45         & 21 \%         & 903.36     & 10034.36       & 125        & 60 \%         & 137.68     & 64.96      & 62      & 30 \% & 842.95   & 38.69        \\
        \midrule
        \thead{Sparse \\ Regression}  & 160          & 147        & \textbf{92 \%}         & \textbf{25.14}      & \textbf{0.01}  & 137       & 86 \%        & 48.01     & \textbf{0.01} & 1        & 1 \% & 1720.39 & 0.35           & 56         & 35 \%         & 547.69     & \textbf{0.01} & 111     & 69 \% & 252.44   & \textbf{0.01}  \\
        \midrule
        \thead{Sparse Log \\ Regression} & 48           & 15         & \textbf{31 \%}         & \textbf{387.38}     & 0.19           & 14        & 29 \%        & 592.41    & 0.17          & 8        & 17 \% & 665.15 & 0.53           & 7          & 15 \%         & 1121.07    & \textbf{0.02}   & 42       & \textbf{88 \%} & \textbf{69.9}     & 0.08       \\
        \midrule
        \thead{Sparse Tailed \\ Regression}  & 160          & 160        & \textbf{100 \%} & \textbf{0.3} & \textbf{0.0}  & & & &   &  79   & 49 \% & 52.59 & \textbf{0.0}   & & & &    & & & &  \\
        \midrule
        \thead{Sparse Tailed \\ Log Regression}   & 48    & 48  & \textbf{100 \%} & \textbf{0.85} & \textbf{0.0} & & & &   & 36   & 75 \%     & 129.89     & 109.2   & & & &    & & & &  \\
        \bottomrule
    \end{tabular}
    \caption{\revision{Comparing the relative number of terminations and the time of all solver setups on all problems. The average time is taken using the geometric mean shifted by 1 second. Also, note that this is the average time over all instances in that group, i.e.,~it includes the time outs. \revisionTwo{\texttt{SHOT} is separated since it requires and uses the symbolic expression of the objective function as pointed out in the text.}}} 
\label{tab:SummaryOfComparison}
\end{table}

The performance of \package{} is overall good across a large range of instances.
For most problems, it is the best or second best in terms of instances solved and average solving time. 
The exceptions are the MIPLIB \texttt{pg5\_34} and \texttt{ran14x18-disj-8} instances of which none could be solved by \package.
For these problems, the corresponding MIP is expensive to solve and thus, \package{} is, as expected, not a good approach. 
Note though that all solver setups \revisionTwo{except \texttt{SHOT}} struggle with the two problems which point to them being hard to solve in general. 

On many instances of the Poisson Regression Problem and Sparse Log Regression Problem \package{} times out as well. 
Looking at progress plots for each problem in \cref{fig:ProgressPlots}, we see that the incumbent does not change drastically after the first 200 nodes.
Thus, the closing of the optimality gap has to be achieved by improving the lower bound. 
Observe that the lower bound is improving very slowly.
This points to the continuous relaxation in the nodes being too loose.
On the other hand, the progress of the lower bound could be improved by potentially exploiting either sharpness or strong convexity.

\begin{figure}
    \centering
    \begin{subfigure}[t]{0.49\textwidth}
        \centering
        \includegraphics[width=0.9\textwidth]{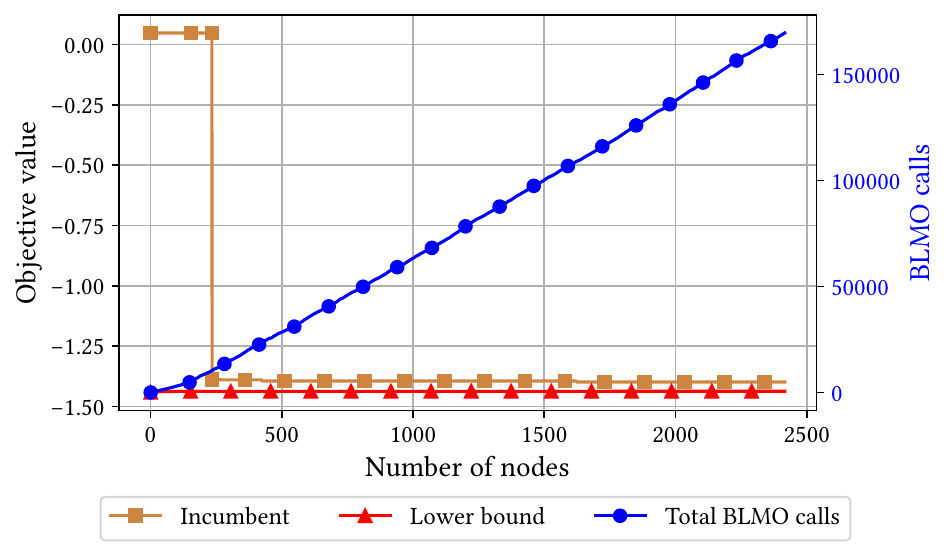}
        \caption{\revision{Poisson Regression of dimension 70.}}
        \label{fig:poissonRegProgress}
    \end{subfigure}
    \hfill
    \begin{subfigure}[t]{0.49\textwidth}
        \centering
        \includegraphics[width=0.9\textwidth]{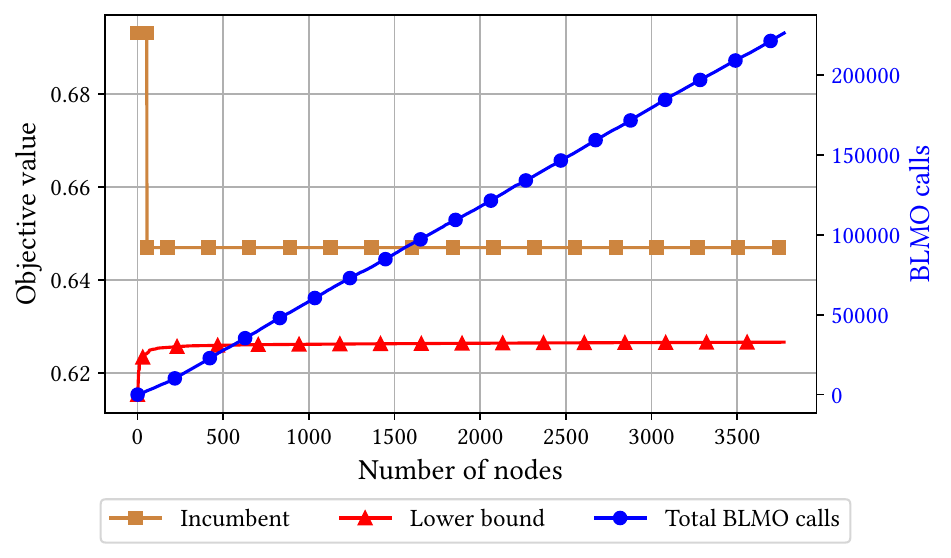}
        \caption{\revision{Sparse Log Regression of dimension 75.}}
        \label{fig:sparseLogRegProgress}
    \end{subfigure}
    \caption{\revision{Progress of the lower bound and the incumbent of \package{} over the number of nodes as well as the accumulated number of BLMO calls.}}
    \label{fig:ProgressPlots}
\end{figure}

The effect strong convexity can have on the solution process is showcased in \cref{fig:miplibStrongConvexity} on two MIPLIB \texttt{neos5} instances.
The lower bound is tightened and the solving time is thus reduced significantly, in the case of \cref{fig:miplibNeos55StC} even halved. 
The number of processed nodes decreases by roughly a quarter.

\begin{figure}
    \centering
    \begin{subfigure}[t]{0.49\textwidth}
        \centering
        \includegraphics[width=1.0\textwidth]{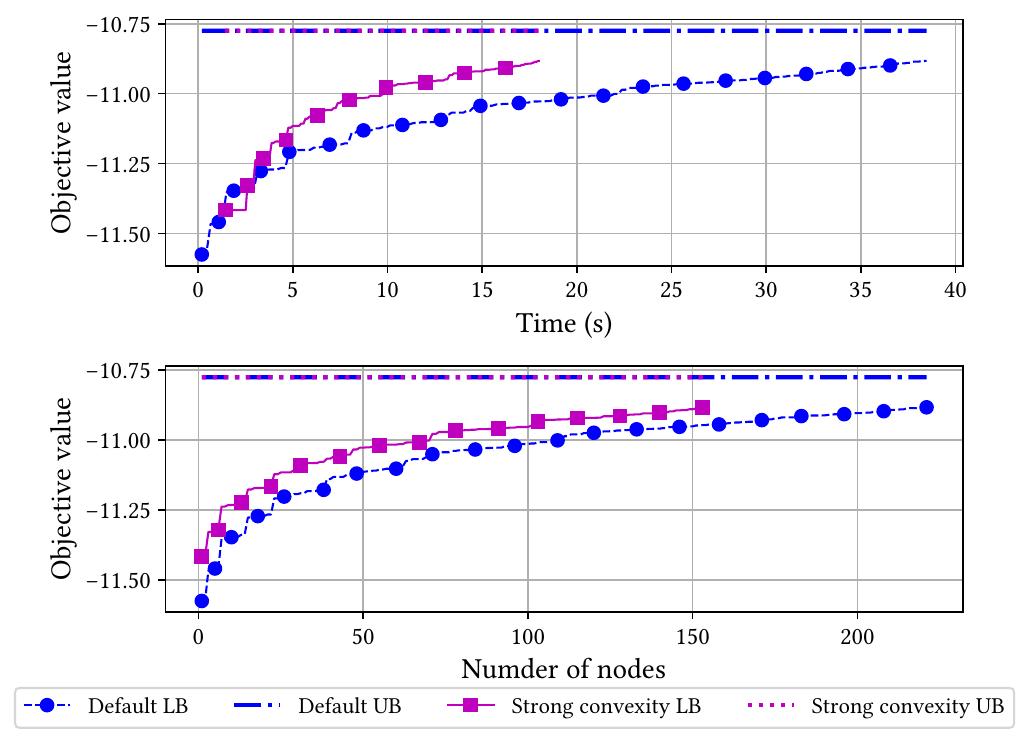}
        \caption{\revision{MIPLIB \texttt{neos5} example with the objective being formed from 5 vertices.}}
        \label{fig:miplibNeos55StC}
    \end{subfigure}
    \hfill
    \begin{subfigure}[t]{0.49\textwidth}
        \centering
        \includegraphics[width=1.0\textwidth]{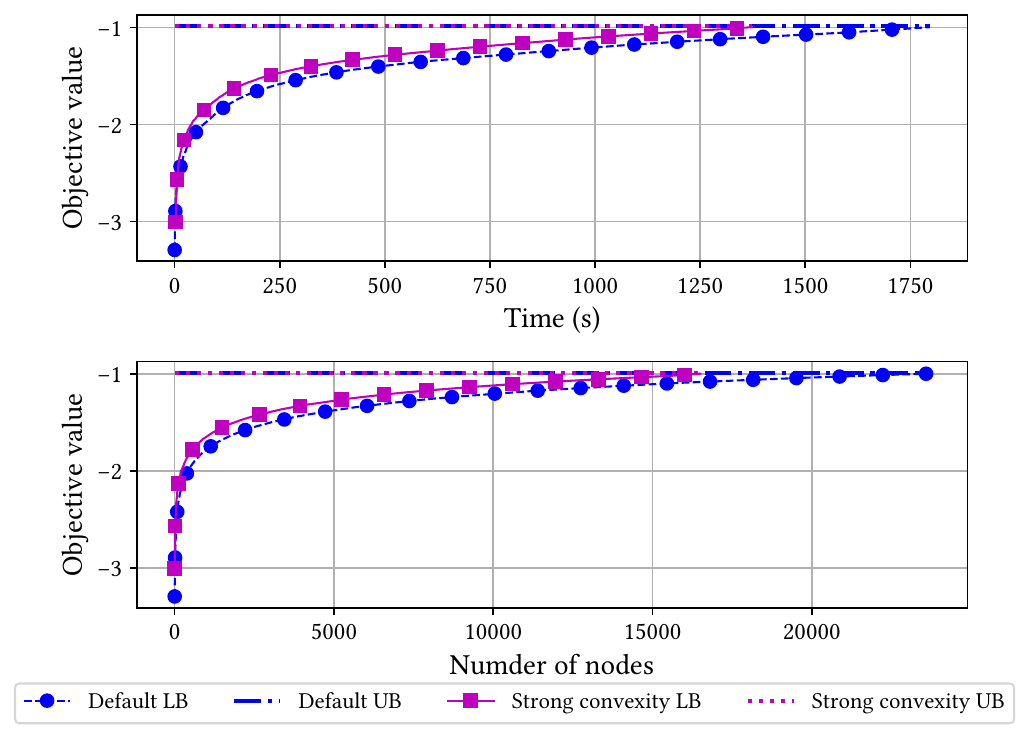}
        \caption{\revision{MIPLIB \texttt{neos5} example with the objective being formed from 8 vertices.}}
        \label{fig:miplibneos58StC}
    \end{subfigure}
    \caption{\revision{Primal-dual convergence comparison between \package{} with and without using strong convexity on MIPLIB \texttt{neos5} instances.}}
    \label{fig:miplibStrongConvexity}
\end{figure}

It is interesting to note that \texttt{BnB-Ipopt} shows better performance on both types of the portfolio problem, see \cref{fig:portfolioMixedTS} and \cref{tab:SummaryOfComparison}. 
The evaluation of an individual node is faster within \texttt{BnB-Ipopt} than in \package.
But the total number of nodes to evaluate is significantly smaller for \package{} than for \texttt{BnB-Ipopt}.
Notice that \package{} has to evaluate roughly two-thirds fewer nodes than \texttt{BnB-Ipopt}, see \cref{fig:miplibConvergence}.
Consequently, if \texttt{BnB-Ipopt} performs better on instances (of larger dimension), we can conclude that \package{} has to process a comparatively large number of nodes and hence, the solution time increases.
In \cref{fig:ProgressPortfolioInteger}, we can see the progress plot for a Pure Integer Portfolio instance. 
The incumbent continuously improves for the first 200 nodes, hinting that problem-specific heuristics could be beneficial.

\begin{figure}
    \centering
        \includegraphics[width=0.6\textwidth]{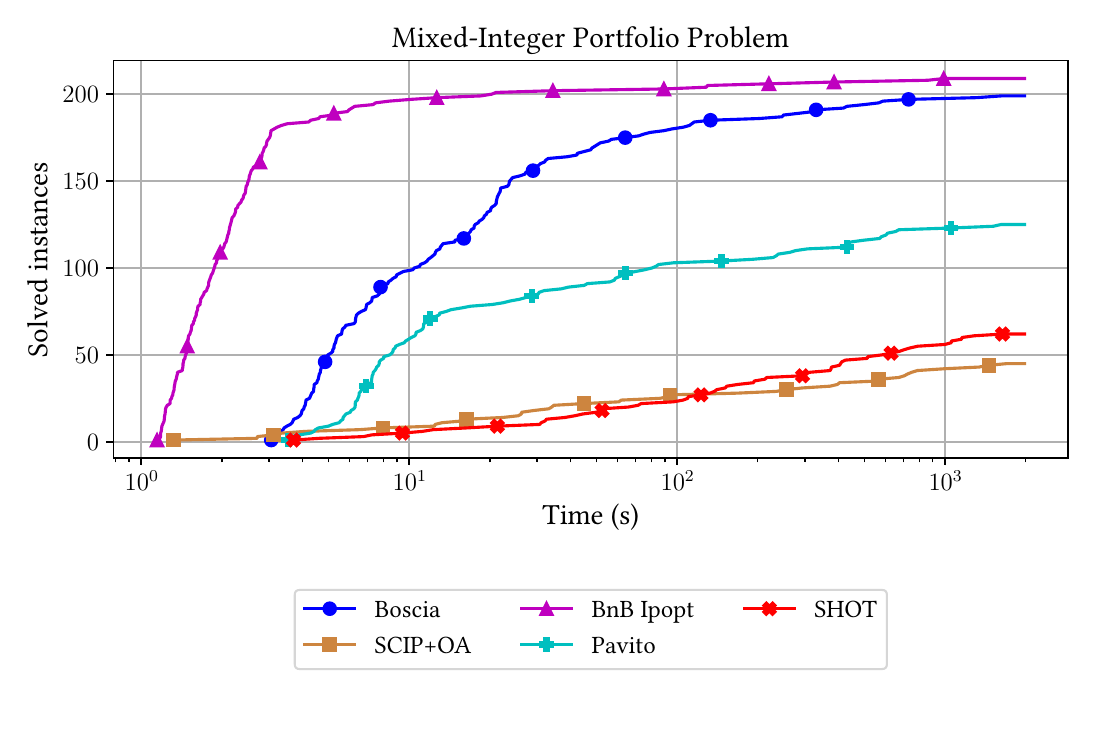}
        \caption{\revision{Termination on the Mixed Integer Portfolio instances}}
    \label{fig:portfolioMixedTS}
\end{figure}

\begin{figure}
    \centering
    \begin{subfigure}[t]{0.49\textwidth}
        \centering
        \includegraphics[width=1.0\textwidth]{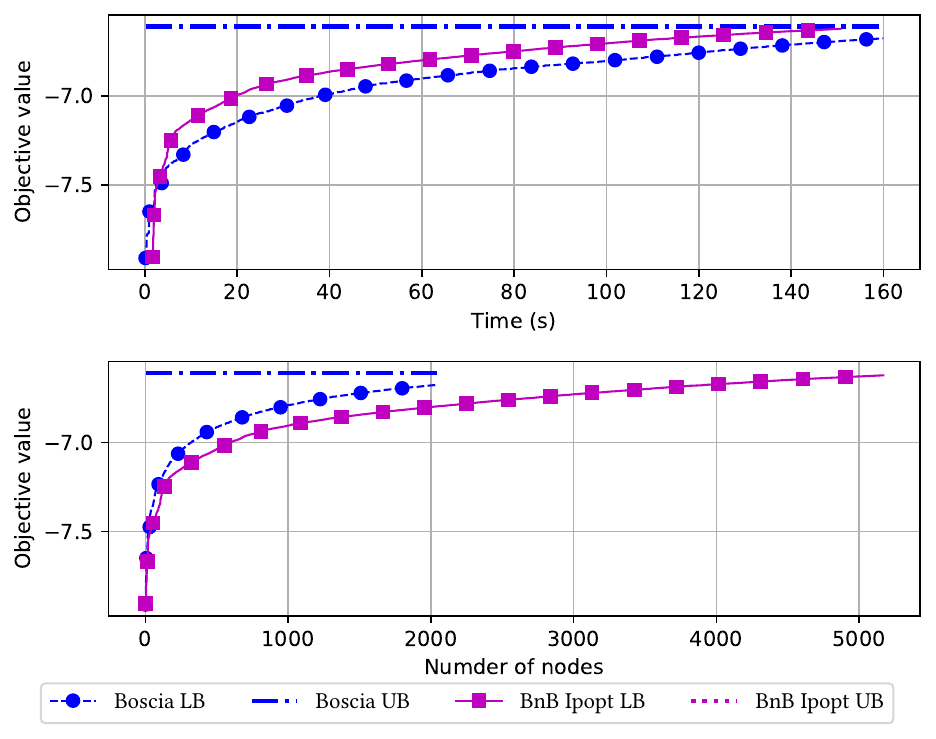}
        \caption{\revision{MIPLIB \texttt{neos5} example with the objective being formed from 6 vertices.}}
        \label{fig:miplibNeos56}
    \end{subfigure}
    \hfill
    \begin{subfigure}[t]{0.49\textwidth}
        \centering
        \includegraphics[width=1.0\textwidth]{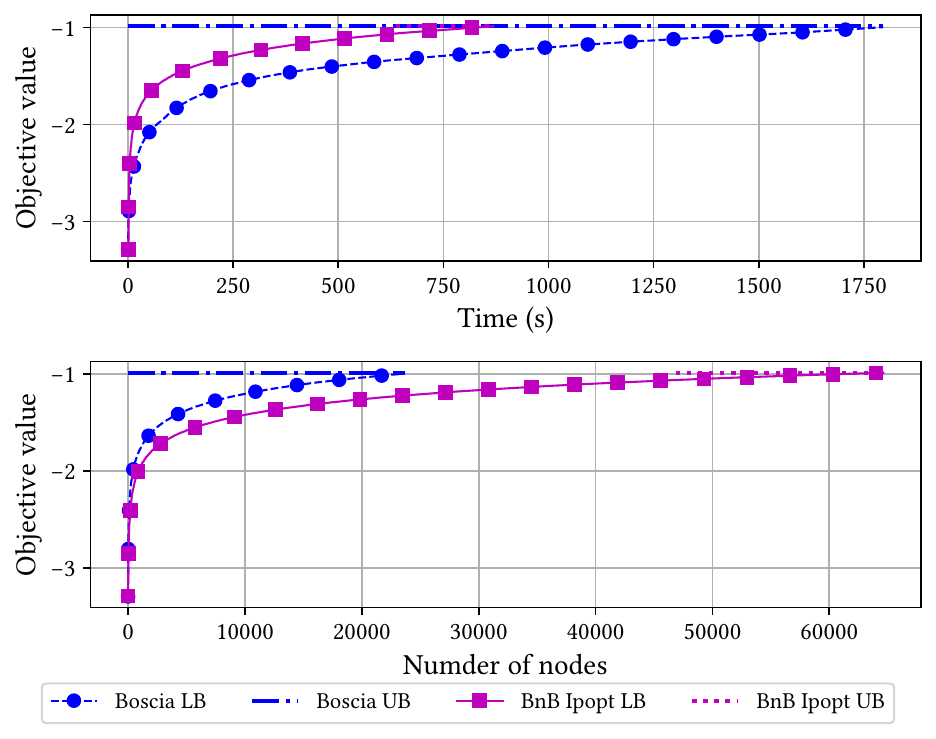}
        \caption{\revision{MIPLIB \texttt{neos5} example with the objective being formed from 8 vertices.}}
        \label{fig:miplibneos58}
    \end{subfigure}
    \caption{\revision{Primal-dual convergence comparison between \package{} and \texttt{BnB-Ipopt} on MIPLIB \texttt{neos5} instances.}}
    \label{fig:miplibConvergence}
\end{figure}

\begin{figure}
    \centering
    \includegraphics[width=0.6\textwidth]{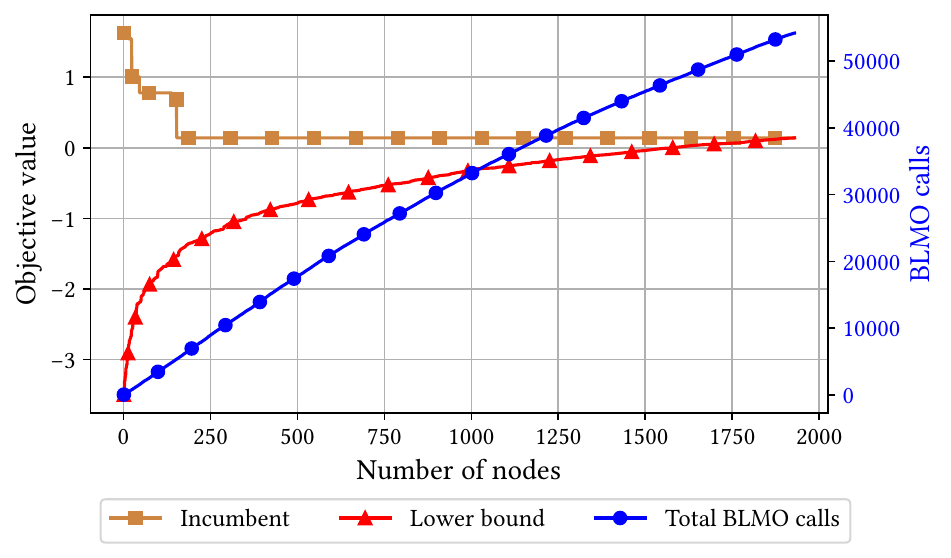}
    \caption{\revision{Progress of the lower bound, incumbent and accumulated number of LMO calls over the number of nodes for a Pure Integer Portfolio instance with 35 integer variables.}}
    \label{fig:ProgressPortfolioInteger}
\end{figure}

With regards to \texttt{SCIP+OA}, one main observation is that it performs well on small instances. 
For larger instances, the addition of cutting planes can lead to numerically ill-conditioned problems.
\revision{In particular for the Sparse Regression Problem, the Sparse Log Regression Problem and the Poisson Regression Problem, there were instances
for which SCIP reported optimality but \package{} had a better primal bound.
It seems that numerical instabilities lead to \texttt{SCIP+OA} cutting away the actual optimal solution and returning a suboptimal solution.
Note that we counted instances not solved by \texttt{SCIP+OA} if the reported solution had an absolute difference of 0.01 to the primal bound of \package.
Given the magnitude of the objectives, this is equivalent to a relative gap of 0.05.}

\texttt{SCIP+OA} shows a very good performance on the Poisson Regression instances, see \cref{fig:poissonRegTS}.
Given the occurrence of false positives in the Poisson Regression instances, we should take these results with a grain of salt. 
\begin{figure}
    \centering
        \includegraphics[width=0.6\textwidth]{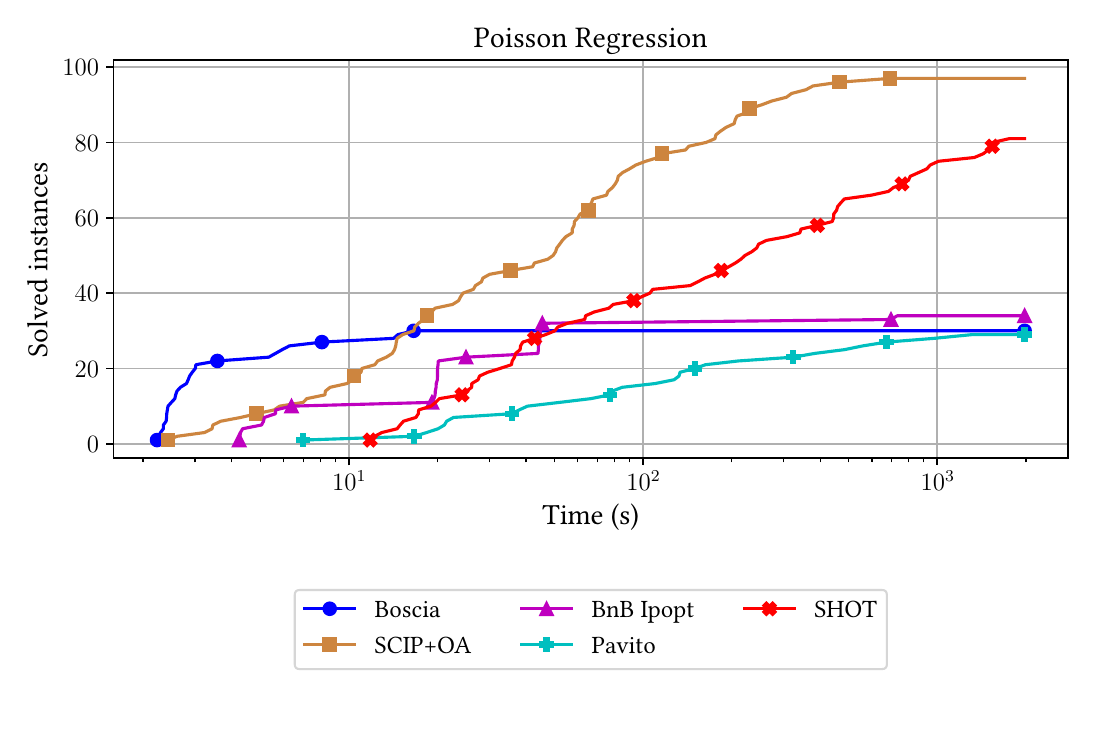}
        \caption{\revision{Termination on the Poisson Regression instances.}}
    \label{fig:poissonRegTS}
\end{figure}

\revision{\texttt{Pavito}'s performance is similar or worse to \texttt{SCIP+OA}.
Notable exceptions are the Sparse Regression instances and the Mixed Integer Portfolio instances 
on which \texttt{Pavito} performs better than \texttt{SCIP+OA}, see \cref{fig:portfolioMixedTS,fig:sparseRegTS}.
It outperforms \package{} only on the MIPLIB \texttt{ran14x18-disj-8} instance. 
Like \texttt{SCIP+OA}, the addition of cutting planes causes the internal MIP to become increasingly numerically difficult to solve.}

\revisionTwo{Lastly, \texttt{SHOT} performs very well on the MIPLIB instances and the sparse log regression problem on which it outperforms every other solver.
Since the MIPLIB instances \texttt{ran14x18-disj-8} and \texttt{pg5\_34} are hard to solve for all the other solvers,
we infer that \texttt{SHOT} benefits from its utilization of the objective function expression tree.
It is outperformed by \package{} on the Sparse Regression and Mixed Portfolio instances.}

\begin{figure}
    \centering
    \includegraphics[width=0.6\textwidth]{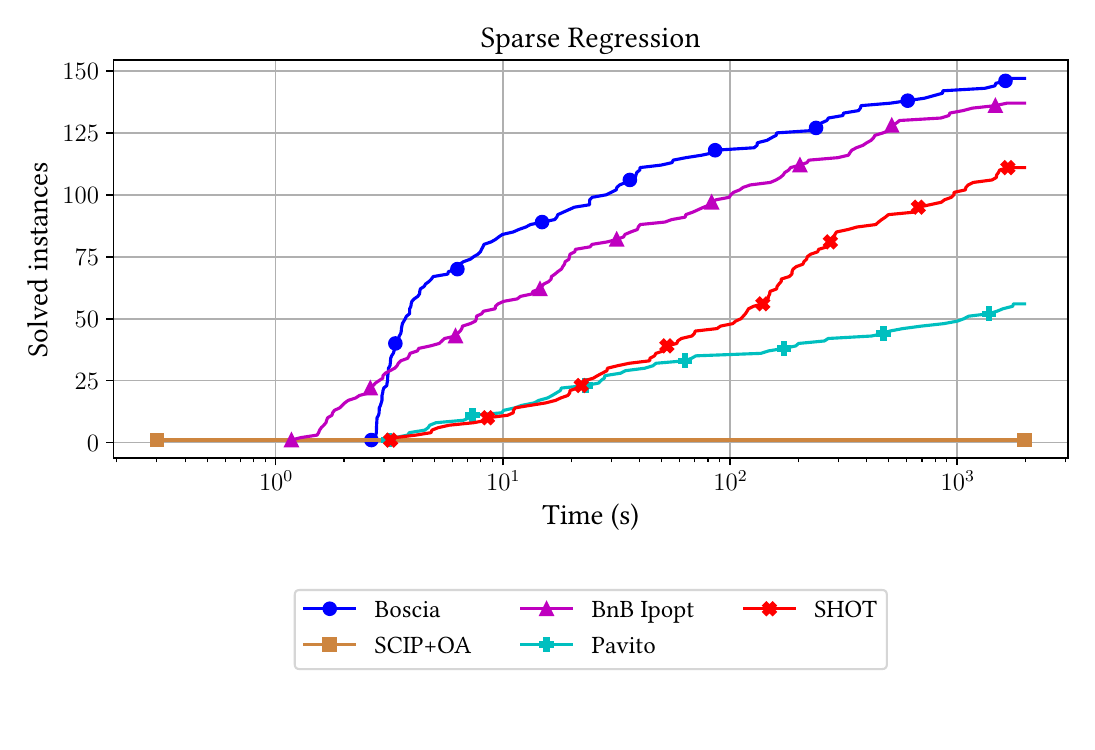}
    \caption{\revision{Termination on the Sparse Regression instances.}}
    \label{fig:sparseRegTS}
\end{figure}

\revisionTwo{We also experiment} with the various settings of \package{} to investigate their impact on the performance. 
The detailed results can be found in \cref{tab:SummaryOfTighteningAndStrongConvexity}, comparing the impact of tightening and strong convexity, \cref{tab:SummaryOfWarmStart}, comparing the effect of the warm-start settings, and \cref{tab:SummaryOfBranching}, comparing the performance of the different branching strategies \revisionTwo{in the \cref{app:DetailedExperiments}}. 

\cref{fig:WarmStartIntegerPortfolio} compares the number of solved instances for the different warm start and shadow set settings as well as comparing using BPCG vs using Away-Frank-Wolfe (AFW).
\revision{Disabling warm starting has a visible negative effect on the performance. 
The effect of not utilizing the shadow set is not so pronounced pointing to the BLMO being relatively cheap. 
Using AFW instead of BPCG impacts the performance significantly. 
This is expected since BPCG converges faster than AFW and requires fewer BLMO calls.}

\begin{figure}
    \centering
    \includegraphics[width=0.6\textwidth]{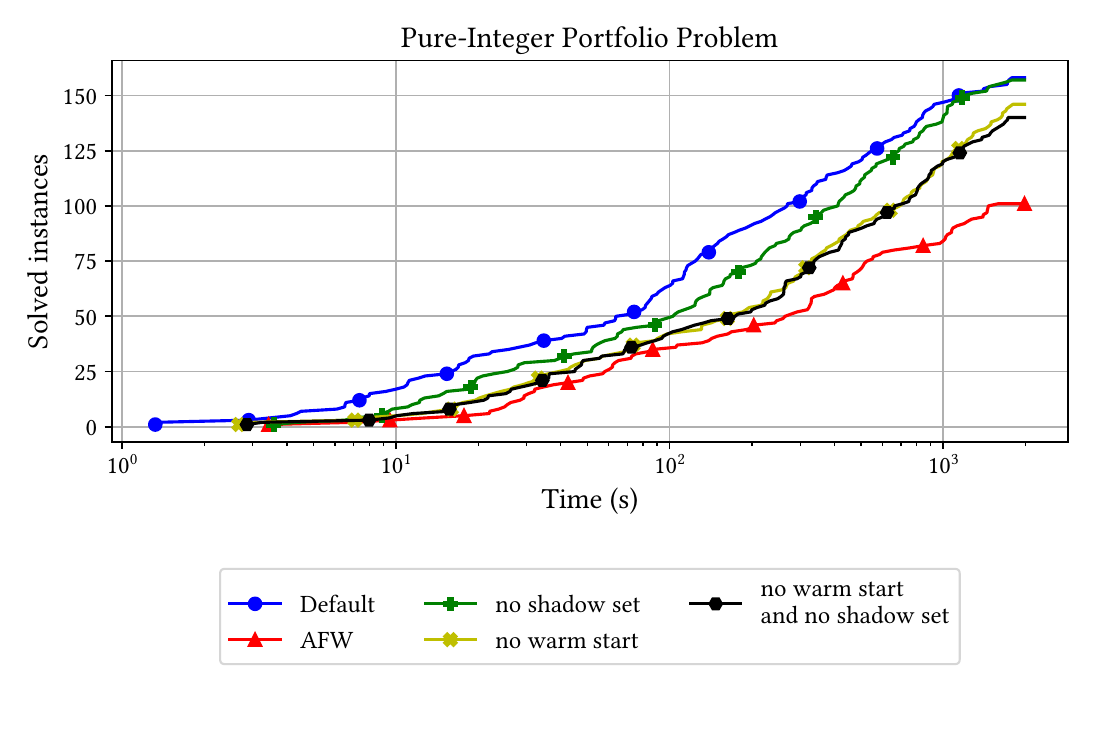}
    \caption{\revision{Comparing the effect of warm start, shadow and used Frank-Wolfe variant on the performance of \package{} on the Pure Integer Portfolio instances. Per default, \package{} uses both warm starting and the shadow set. The default Frank-Wolfe is BPCG.}}
    \label{fig:WarmStartIntegerPortfolio}
\end{figure}

The effects of the global and local tightening are problem-dependent as seen in \cref{fig:Tightenings}.
While local tightening has a positive effect on the Sparse Regression Problem, see \cref{fig:TighteningsSparseReg},
it increases both the solving time as well as the number of BLMO calls for the Pure Integer Portfolio problem (\cref{fig:TighteningsIntegerPortfolio}).
This is explained by the fact that tightening can render vertices of the shadow set infeasible and discard them,
these vertices cannot be used anymore and have to be replaced by new BLMO calls on the updated feasible region.

\begin{figure}
    \centering
    \begin{subfigure}[t]{0.49\textwidth}
        \centering
        \includegraphics[width=0.9\textwidth]{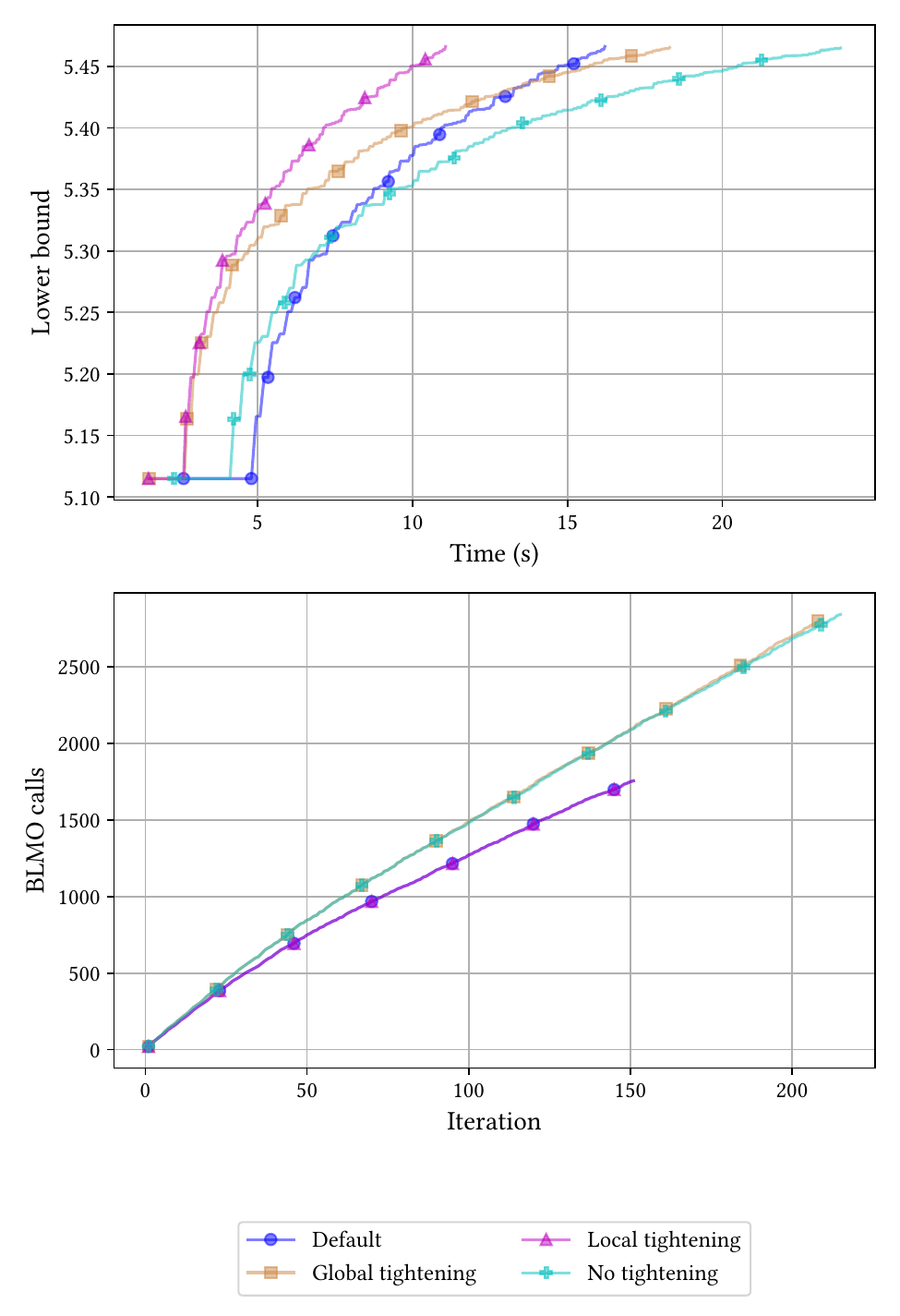}
        \caption{\revision{Sparse Regression instance with 115 integer variables.}}
        \label{fig:TighteningsSparseReg}
    \end{subfigure}
    \hfill
    \begin{subfigure}[t]{0.49\textwidth}
        \centering
        \includegraphics[width=0.9\textwidth]{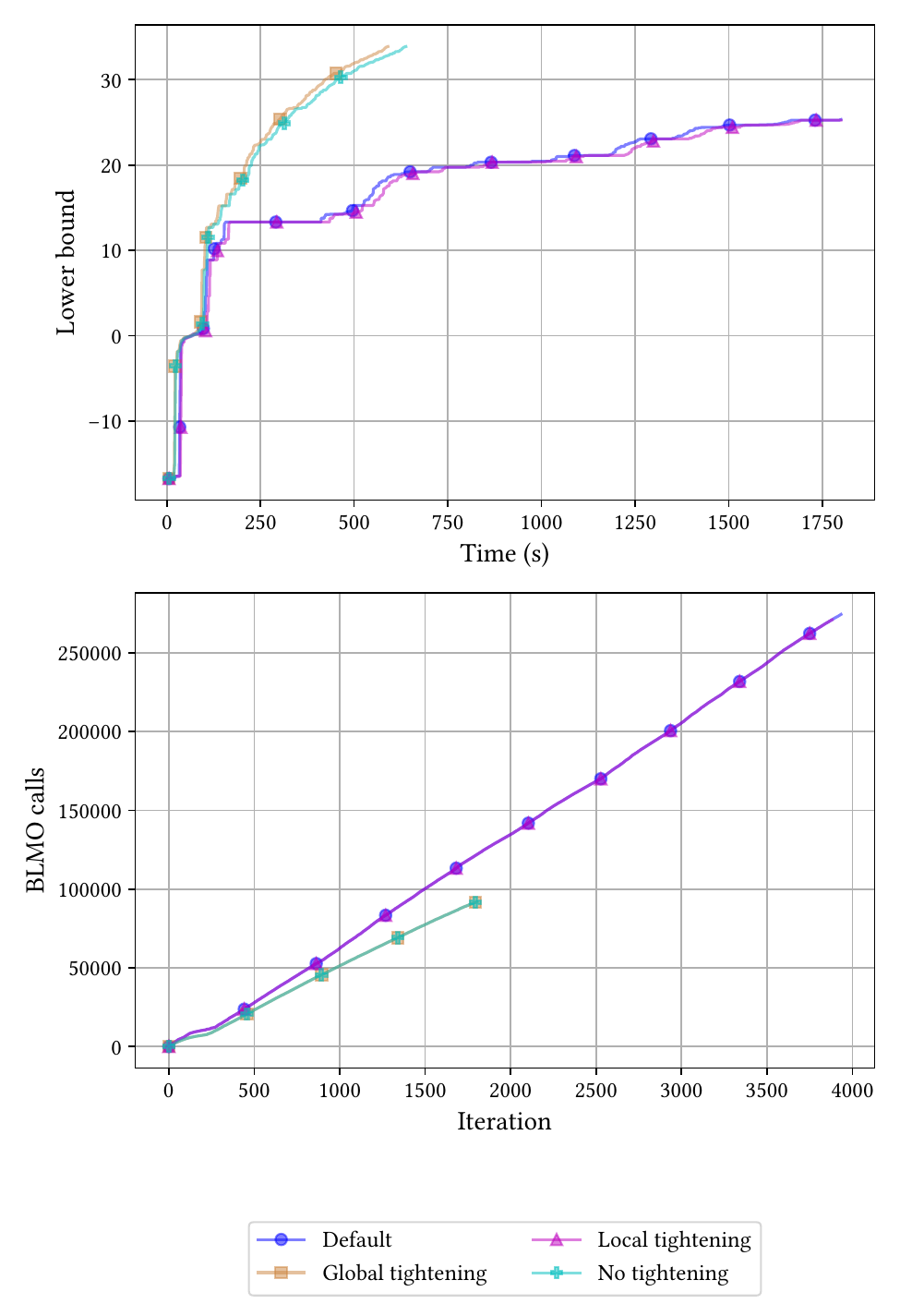}
        \caption{\revision{Pure Integer Portfolio instance with 120 variables.}}
        \label{fig:TighteningsIntegerPortfolio}
    \end{subfigure}
    \caption{\revision{Comparing the effect of the different tightening settings. The default setting uses both global and local tightening.}}
    \label{fig:Tightenings}
\end{figure}

In \cref{fig:BranchingMixedPortfolio}, the progress of the lower bound for the different branching strategies is displayed.
As expected, the strong branching rule is prohibitively expensive.
The hybrid branching rule fares better than the strong branching rule in terms of progress per node.
Nevertheless, it is expensive to perform even for a shallow depth.
Time-wise, the most fractional branching rule is still the best choice.
\revision{Although this is far from being the case in MILP, where most fractional branching performs on par with random branching,
we can highlight that previous studies \citep{gamrath2020exploratory} showed that dual degeneracy can partly explain the weakness of variable fractionality for branching.
Because of the nonlinear convexity of the MINLPs we consider, dual degeneracy is much less likely and fractionality is a reasonable choice. 
We leave studies of other branching rules applied to the \package{} algorithm to future work.}

\revisionTwo{To conclude, our newly proposed framework \package{} shows promising results.
In contrast to the Outer Approximation approaches, it maintains feasibility throughout and
there is no risk that the optimal solution is being cut off. 
It shows good performance on the Regression problems on which the Outer Approximation approaches struggle.
On the other hand, \package{} struggles if the LMO is expensive. 
The experiments have also shown that employing heuristics and further techniques to increase the lower bound progress will be advantagous for the framework.
All in all, \package{} is a good addition to the set of solution strategies for convex mixed-integer problems.
}

\begin{figure}
    \centering
    \begin{subfigure}[t]{0.49\textwidth}
        \centering
        \includegraphics[width=0.9\textwidth]{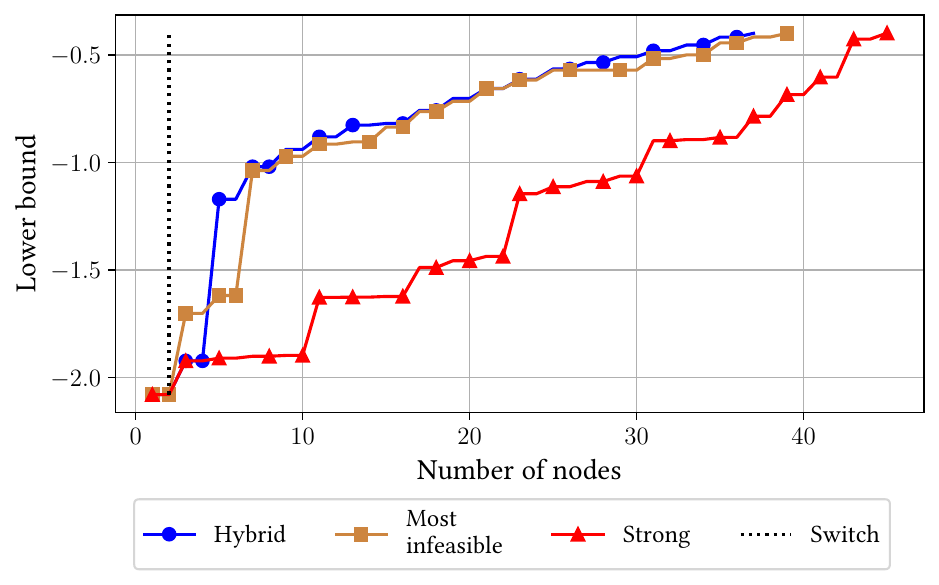}
        \caption{\revision{Lower bound progress over the number of nodes.}}
        \label{fig:BranchingNodesMixedPortfolio}
    \end{subfigure}
    \hfill
    \begin{subfigure}[t]{0.49\textwidth}
        \centering
        \includegraphics[width=0.9\textwidth]{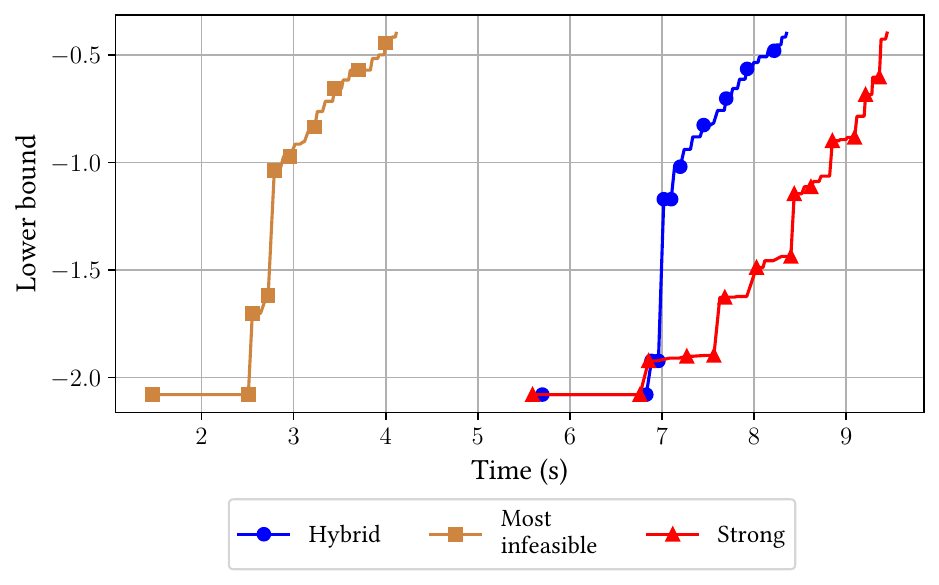}
        \caption{\revision{Lower bound progress over time.}}
        \label{fig:BranchingTimeMixedPortfolio}
    \end{subfigure}
    \caption{\revision{Comparing the branching strategies on a Mixed Integer Portfolio instance with 25 integer variables. In the hybrid branching, the depth setting is \# integer variables$/20$ is reached.}}
    \label{fig:BranchingMixedPortfolio}
\end{figure}

 \section{Conclusion}

In this paper, we proposed a novel algorithm for mixed-integer convex optimization
relying only on gradient and function evaluations of the objective, and linear optimization over the feasible set.
By embedding a FW-based subsolver within a branch-and-bound framework, our method
does not rely on separation or projection subproblems.
Since FW algorithms rely on LMO calls to handle the constraint set, we significantly
strengthen convex relaxations by optimizing over the local integer hull, leveraging the capabilities of modern MIP solvers.
Lazification techniques within and across nodes avoid expensive MIP solves by
exploiting all vertices that have been discovered and further MIP information,
while tightening and stronger dual bounds reduce the size of the branch-and-bound tree and tighten the feasible region.

\section*{Acknowledgments}

Research reported in this paper was partially supported through the Research Campus Modal funded by the German Federal Ministry of Education and Research (fund numbers 05M14ZAM,05M20ZBM) and the Deutsche Forschungsgemeinschaft (DFG) through the DFG Cluster of Excellence MATH+.

\bibliographystyle{icml2021}
\bibliography{refs}

\clearpage 
\newpage
\appendix
\input{supplementary.tex}

\end{document}

%% file: supplementary.tex
\section{Blended Pairwise Conditional Gradient}\label{app:bpcg}

We present in \cref{alg:bpcgshadow} the modified Lazy Blended Pairwise Conditional Gradient (L-BPCG) from \citet{tsuji2021sparser} to solve relaxations at each node,
combined with the adaptive step size proposed in \cite{pokutta2024frank} which dynamically estimates a local bound on the Lipschitz constant.
The procedure populates the shadow set with dropped vertices and starts from the warm-started active set.
The convergence follows from that of BPCG, we use the same progress measure on whether a shadow vertex
offers sufficient decrease as FW direction (vertex the iterates moves towards)
as we do for a classic pairwise step. The additional shadow vertex selection can be viewed as a special case of
a pairwise step where the initial weight in the active set is zero.

\begin{algorithm}
\caption{Lazy Blended Pairwise Conditional Gradient with Shadow Set and Cutoff}\label{alg:bpcgshadow}
\begin{algorithmic}[1]
\Require Starting active set $\mathcal{A}_0$ and weights $\lambda$, shadow set $\mathcal{S}_0$, function $f$, feasible set of current node $P$, $\varepsilon_{tol}$, accuracy $K \geq 1$, primal bound $\hat{f}$.
\Ensure Final iterate $\vx$ such that $f(\vx) \leq f^* + \varepsilon_{\text{tol}}$ or a dual bound greater than $\hat{f}$
\State $\vx_0 \gets \sum_{k = 1}^{|\mathcal{A}_0|} \lambda_k \vvv_k$
\State $\Phi_0 \gets \max_{\vvv\in P} \innp{\nabla f(\vx_0)}{\vvv}/2$
\State $t = 0$
\While{$\Phi_t > \varepsilon_{\text{tol}}$ and $f(\vx_t) - \Phi_t < \hat{f}$} \Comment{Near-optimal or pruned by the bound}
\State $a_t \gets \argmax_{\vvv \in \mathcal{A}_t} \innp{\nabla f(\vx_t)}{\vvv}$ \Comment{away vertex}
\State $s_t \gets \argmin_{\vvv \in \mathcal{A}_t} \innp{\nabla f(\vx_t)}{\vvv} $ \Comment{local forward vertex}
\If{$\innp{\nabla f(\vx_t)}{a_t - s_t} \geq \Phi_t$}
\State $d_t \gets a_t - s_t$
\State $\gamma_{\text{max}} \gets \texttt{weight\_of}(\mathcal{A}_t, a_t)$
\State $\gamma_t \gets \argmin{\gamma \in \left[0,\gamma_{\text{max}}\right]} f(x - \gamma d_t)$
\State $\vx_{t+1} \gets \vx_{t} - \gamma_t d_t$
\State $\Phi_{t+1} \gets \Phi_{t}$
\If{$\gamma_t < \gamma_{\text{max}}$} \Comment{descent step}
    \State $\mathcal{A}_{t+1} \gets \mathcal{A}_{t}$
\Else\Comment{drop step}
    \State $\mathcal{A}_{t+1} \gets \mathcal{A}_{t} \backslash \{a_t\} $
    \State $\mathcal{S}_{t+1} \gets \mathcal{S}_{t} \cup \{a_t\}$
\EndIf
\Else
\State $s_t \gets \argmin_{\vvv \in \mathcal{S}_t} \innp{\nabla f(\vx_t)}{\vvv} $ \Comment{forward vertex from dropped vertices}
\If{$\innp{\nabla f(\vx_t)}{a_t - s_t} \geq \Phi_t$}
\State $\mathcal{A}_{t+1} \gets \mathcal{A}_{t} \cup \{s_t\} $
\State $\mathcal{S}_{t+1} \gets \mathcal{S}_{t} \backslash \{s_t\}$
\State descent or drop step
\Else
    \State $w_t \gets \argmin_{\vvv \in P} \innp{\nabla f(\vx_t)}{\vvv}$ \Comment{global LMO}
    \If{ $\innp{\nabla f(\vx_t)}{\vx_t - w_t} \geq \Phi_t / K$}
        \State $d_t = \vx_t - w_t$
        \State $\gamma_t \gets \argmin_{\gamma \in \left[0,1\right]} f(x - \gamma d_t)$
        \State $\vx_{t+1} \gets \vx_{t} - \gamma_t d_t$
        \State $\Phi_{t+1} \gets \Phi_{t}$
        \State $\mathcal{A}_{t+1} \gets \mathcal{A}_{t} \cup \{w_t\} $
    \Else
        \State $\vx_{t+1} \gets \vx_{t}$
        \State $\Phi_{t+1} \gets \Phi_{t} / 2$
        \State $\mathcal{A}_{t+1} \gets \mathcal{A}_{t} $
    \EndIf
\EndIf
\EndIf
\State $t \gets t + 1$
\EndWhile
\end{algorithmic}
\end{algorithm}

\section{Tighter lower bound and dual tightening inequalities}\label{app:bounds}

In this appendix, we provide a modified lower-bound inequality when the function is sharp, and strengthen the dual tightening condition
when it is strongly convex.\\

For many objective functions of interest, strong convexity is too strong of a requirement, if they are not strongly convex or
if it offers a loose inequality. We derive an equivalent improved dual bound for functions presenting a \emph{Hölder error bound}
or \emph{sharpness} property \citep{hoffman2003approximate,nemirovskii1985optimal,polyak1979sharp}.
We also highlight that Frank-Wolfe algorithms benefit from accelerated rates under a sharpness assumption \citep{kerdreux2022restarting}.

\begin{definition}[Hölder Error Bound]
Let $f$ be a function $f: X\rightarrow\mathbb{R}$ and $C$ a compact neighborhood around the minimizer set $X^* := \argmin_{\vx\in X} f(\vx)$. The function $f$ satisfies a $(\theta,M)$ Hölder Error Bound on $C$ if $\exists\, \theta \in \left[0,1/2\right]$ and $M>0$, such that:
\begin{equation}\label{eq:holder}
\min_{\vx^*\in X^*} \|\vx - \vx^*\| \leq M (f(\vx) - f^*)^{\theta}.
\end{equation}
\end{definition}

\noindent
We will further assume that the optimum is unique $X^* := \{\vx^*\}$.
Denoting by $\vx^*_l$ 
the optimal solution after branching, we have:
\begin{align*}
    \norm{\vx^*_l - \vx^*} \leq & M (f(\vx^*_l) - f^*)^{\theta} \\
    f(\vx^*_l) \geq & M^{-1/\theta}  \norm{\vx^*_l - \vx^*}^{1/\theta} + f(\hat{\vx}) -  g(\hat{\vx}).
\end{align*}
Furthermore:
\begin{align*}
    \norm{\vx^*_l - \vx^*} & \geq \norm{\vx^*_l - \hat{\vx}} - \norm{\hat{\vx} - \vx^*}\\
    & \geq \norm{\vx^*_l - \hat{\vx}} - M (f(\hat{\vx}) - f^*)^{\theta}\\
    & \geq \norm{\vx^*_l - \hat{\vx}} - M (g(\hat{\vx}))^{\theta},
\end{align*}
resulting in the following bound:
\begin{align}
    f(\vx^*_l) \geq & M^{-1/\theta} (\norm{\hat{\vx} - \vx^*} - M (g(\hat{\vx}))^{\theta})^{1/\theta} + f(\hat{\vx}) -  g(\hat{\vx}),
\end{align}
where a lower bound on $\norm{\hat{\vx} - \vx^*}$ can be obtained similarly to the strong convexity result.
\noindent
We can also improve the dual bound tightening in case the function is $\mu$-strongly convex.

\begin{theorem}\label{the:dualstronger}
Let us assume that the bounds are $\left[\vl,\vu\right] \supseteq \Xset$ and that we have a relaxed solution $\vx^{(t)}$ and a variable $j \in J$
such that $\vx^{(t)}_j = \vl_j$ and $\nabla f(\vx^{(t)})_j \geq 0$.
Furthermore, if $f$ is $\mu$-strongly convex, then,
if there exists $M \in \{1, \dots, \vu_j-\vl_j\}$, such that:
\begin{align}\label{eq:dualtighteningstrong}
M \nabla f(\vx^{(t)})_j > \mathrm{UB} - f(\vx^{(t)}) + g_{-j}(\vx^{(t)}) - \frac{\mu}{2} M^2,
\end{align}
with $g_{-j}(\cdot)$ the masked Frank-Wolfe gap defined in \cref{the:dualtightening},
then the upper bound can be tightened to $\vx^{(t)}_j \leq \vl_j + M - 1$.    
\end{theorem}
\begin{proof}
Given the iterate $\vx^{(t)}$, starting from strong convexity:
\begin{align*}
& f(\vx) - f(\vx^{(t)}) \geq \innp{\nabla f(\vx^{(t)})}{\vx - \vx^{(t)}} + \frac{\mu}{2} \norm{\vx - \vx^{(t)}}^2 \,\,\forall \vx \in \Xset.
\end{align*}

Taking the minimum over all solutions $\vx \in \Xset_M$, we obtain:
\begin{align*}
    \min_{\vx \in \Xset_M} f(\vx) - f(\vx^{(t)}) \geq\, & \min_{\vx \in \Xset_M} (\vx - \vx^{(t)})_j \nabla f(\vx^{(t)})_j + \sum_{k\neq j} \nabla f(\vx^{(t)})_k (\vx - \vx^{(t)})_k +  \frac{\mu}{2} \norm{\vx - \vx^{(t)}}^2 \\
    \geq\, & \min_{\vx \in \Xset_M} (\vx - \vx^{(t)})_j \nabla f(\vx^{(t)})_j + \min_{\vx \in \Xset_M} \sum_{k\neq j} \nabla f(\vx^{(t)})_k (\vx - \vx^{(t)})_k + \min_{\vx \in \Xset_M} \frac{\mu}{2} \norm{\vx - \vx^{(t)}}^2 \\
    \geq\, & M \nabla f(\vx^{(t)})_j - g_{-j}(\vx^{(t)}) + \frac{\mu}{2} M^2,
\end{align*}
since the optimum in $\Xset_M$ has at least the $j$-th variable at distance $M$ from $\vx^{(t)}$.
This results in the following modification of inequality~\eqref{eq:validcomplement}:
\begin{align}\label{eq:validcomplementstrong}
& \mathrm{UB} - f(\vx^{(t)}) \geq  M \nabla f(\vx^{(t)})_j - g_{-j}(\vx^{(t)}) + \frac{\mu}{2} M^2.
\end{align}
Therefore, if \eqref{eq:validcomplementstrong} does not hold, i.e.:
\begin{align*}
    & M \nabla f(\vx^{(t)})_j > \mathrm{UB} - f(\vx^{(t)}) + g_{-j}(\vx^{(t)}) - \frac{\mu}{2} M^2,
\end{align*}
we can deduce $\vx_j \leq \vl_j + M - 1$.
\end{proof}

\section{Detailed and extended experiments results}\label{app:DetailedExperiments}

\revisionTwo{In this section, we present a more detailed overview of the computational experiments. 
\cref{tab:SummaryByDifficultyMIPLIP,tab:SummaryByDifficultyPortfolio,tab:SummaryByDifficultyRegression,tab:SummaryByDifficultyTailedReg} 
compare the performance of the different solver setups. }

All tables have the same format. The instances for each problem are split into increasingly smaller subsets depending on their minimum solve time, i.e.,~the minimum time any of the solvers took to solve it. 
The cut-offs are at 0 seconds (all problems), took at least 10 seconds to solve, 300 s, 600 s and lastly 1200 s. 
Note that if none of the solvers terminates on any instance of a subset, the corresponding row is omitted from the table. 
The same procedure applies if there are no changes between rows.
For each solver, the relative number of terminations, the average time and the average relative gap with respect to the lower bound of \package{} are displayed.
The average time is taken using the geometric mean shifted by 1 second. 
Also, note that this is the average time over all instances in that group, i.e.,~it includes the time outs.

In \cref{tab:SummaryOfTighteningAndStrongConvexity}, we compare the effect of the different dual tightening settings and strong convexity utilization.
Tightening can be performed locally as well as globally. Per default, both local and global tightening are activated. 
The tightening of the lower bound via strong \revisionTwo{convexity} was only tested on the MIPLIB instances.
As seen in \cref{fig:miplibStrongConvexity}, using strong convexity can reduce the number of nodes to be processed and thus the time. It seems that for instances of small size, however, it is not very effective. 
The effects of the tightenings differ between the problem classes. 
For some examples, local tightening seems to cause computational overhead. e.g.~for the Pure Integer and Mixed Integer Portfolio.
On the other hand, for the Sparse Regression Problem, it has a clear positive effect. 
For the many problems, using both types of tightening is advisable. 

\cref{tab:SummaryOfWarmStart} showcases the impact of the warm starting settings as well as comparing BPCG to Away-Frank-Wolfe as the node solver. 
Per default, both the warm start and the usage of the shadow set are activated. 
The results clearly show that enabling both warm start and the usage of the shadow set is a good choice. 
Disabling the shadow set can be useful if the BLMO is very cheap like in \cite{design_of_experiments_boscia_23} where the feasible region is simply a scaled probability simplex intersected with a hypercube.
In this scenario, both the storing of the shadow set as well as the inner product computation between the gradient and vertices in the shadow set are more expensive compared to the BLMO calls.

Lastly, in \cref{tab:SummaryOfBranching}, the performance of the different branching strategies is displayed. 
As expected, the most fractional branching rule performs the best. 
As mentioned before, the nonlinear convexity counteracts the dual degeneracy.
Hence, most fractional branching is a reasonable choice. 
The partial strong branching rule is still too expensive as expected. 
The hybrid strong branching rule with a shallow depth does seem more promising.
Nevertheless, it is time intensive. 
Our assumption is that problems with a very cheap BLMO might benefit from this branching rule.
Further experiments have to be done. 

\cref{fig:MIP22433Termination,fig:MIPneos5Termination,fig:PoissonTermination,fig:IntegerPortfolioTermination,fig:MixedPortfolioTermination,fig:SparseRegTermination,fig:SparseLogRegTermination,fig:TailedTermination,fig:TailedLogTermination} 
showcase the number of terminations over time for, at the top, the different solver setups, and, at the bottom, for the warm start settings and Away-Frank-Wolfe as node solver,
\revisionTwo{except for those examples for which one of plots has already featured in the main part. The corresponding plots are linked.}
In \cref{fig:ProgressSparseReg,fig:ProgressIntegerPortfolio,fig:ProgressMixedPortfolio,fig:ProgressPoisson}, the progress of the lower bound and the incumbent of the tree are displayed for some instances, on the left with the BLMO calls accumulated and on the right, the BLMO calls per node.
\revisionTwo{We can see a clear trend in that the number of BLMO calls is decreasing with the number of nodes. This is to be expected especially if the shadow set is activated.}
\cref{fig:SizeASSparseReg,fig:SizeASIntegerPortfolio,fig:SizeASMixedPortfolio} show the evolution in terms of size for the active set, shadow set and the accumulated number of BLMO calls over the node depth.
\revisionTwo{While oscilliating, the size of both vertex sets is gradually decrease with higher node depth.}
An example of the effect of the tightening is shown in \cref{fig:TighteningsMixedPortfolio}.
\revisionTwo{It is interesting to observe that the global tightening seems to be expensive both in terms of time as well as number of BLMO calls for this problem.
This could point to the subproblems being very differently from the problem at the root node.
Thus, it would not be beneficial to use information from the root node in the other nodes.}
In \cref{fig:BranchingSparseReg,fig:BranchingIntegerPortfolio}, the lower bound progress over the number of nodes and time is displayed for the most fractional branching rule, the partial strong branching rule and the hybrid strong branching rule with a depth of \# integer variables$/20$.
\revisionTwo{We see that the choice of branching rule is very dependent on the specific structure of the problem. 
For the Sparse Regression problem, most-infeasible branching is clearly the best choice. 
In the case of the Portfolio problem, the strong branching clearly performs better than most-infeasible branching.
Observe that regardless strong branching is expensive and should only be used if the BLMO is very cheap. 
Hybrid branching is a good alternative and the depth to which strong branching is performed can be changed easily depending on the cost of the BLMO.}

The effects of the initial tolerances for the Frank-Wolfe algorithm are shown in \cref{fig:DualDecaySparseReg} by plotting the total solving time for each pair of initial epsilon and dual gap decay factor. 
\revisionTwo{As expected, starting with a high precision and increasing it quickly is not the best approach.
The best result in terms of running time is achieved if the tolerance is very relaxed, i.e.~0.01, and decreased very slowly.
Note that the default setting of \package{} is 0.01 as the start Frank-Wolfe gap tolerance and a dual decay factor of 0.8. 
The plot suggest that the dual gap decay factor might be increased even more to 0.9 for this example.}

\revisionTwo{In summary, the default settings of \package{} are reasonable for most problems. 
Especially, warm starting and most-infeasible branching are useful in most case. 
For other settings like the tightenings, there seems to be a greater variety between the problems.
Note that the settings can be easily adapted.}

\begin{figure}
    \centering
    \begin{subfigure}[t]{1.0\textwidth}
        \centering
        \includegraphics[width=0.7\textwidth]{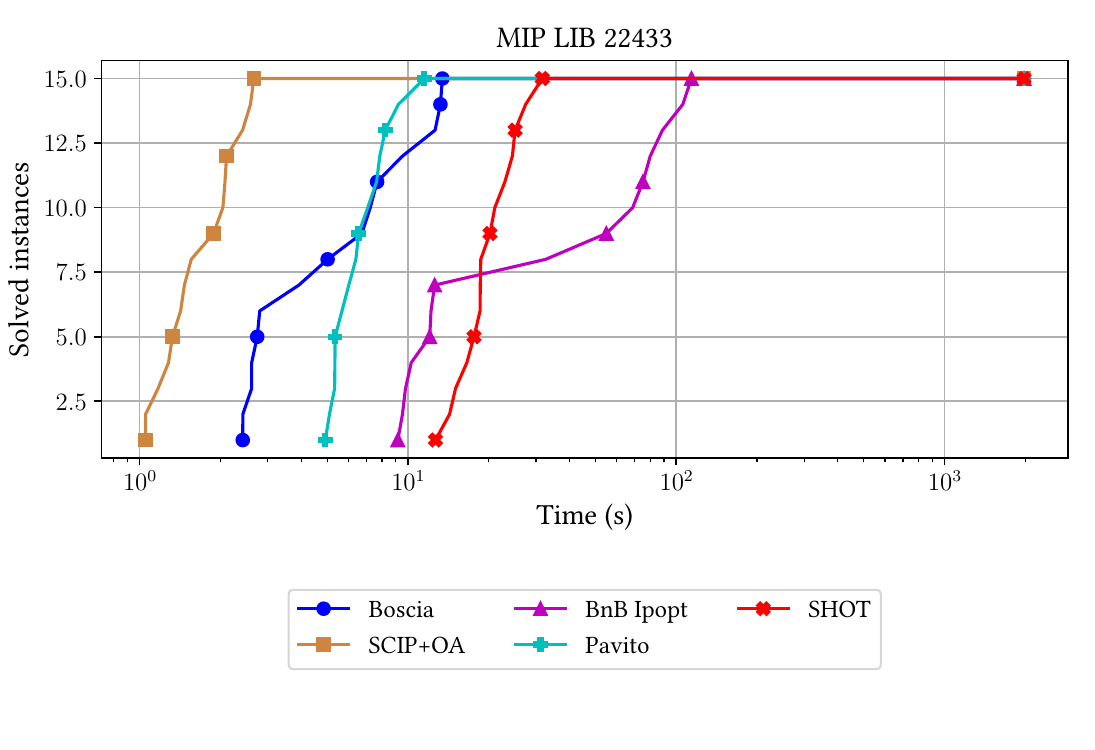}
    \end{subfigure}
    \hfill
    \begin{subfigure}[t]{1.0\textwidth}
        \centering
        \includegraphics[width=0.7\textwidth]{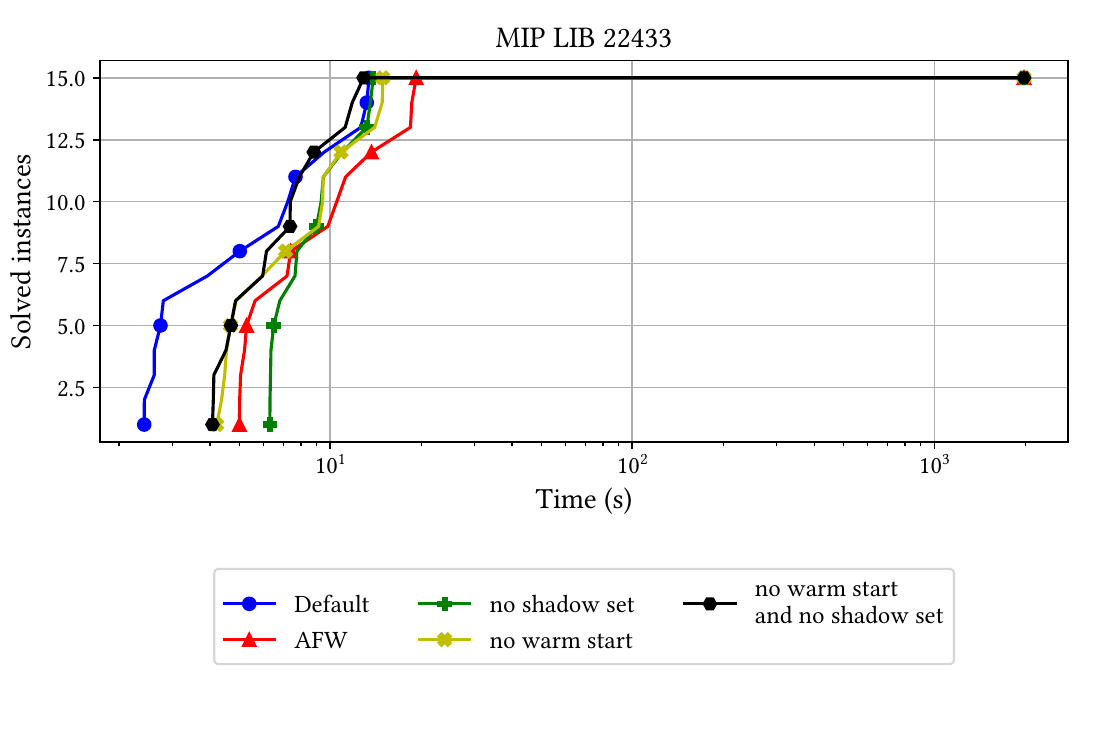}
    \end{subfigure}
    \caption{
    \revision{Comparing the number of terminations over time for the different solver set-ups and \package{} settings for the MIPLIB \texttt{22433} Problem.}
    }
    \label{fig:MIP22433Termination}
\end{figure}

\begin{figure}
    \centering
    \begin{subfigure}[t]{1.0\textwidth}
        \centering
        \includegraphics[width=0.7\textwidth]{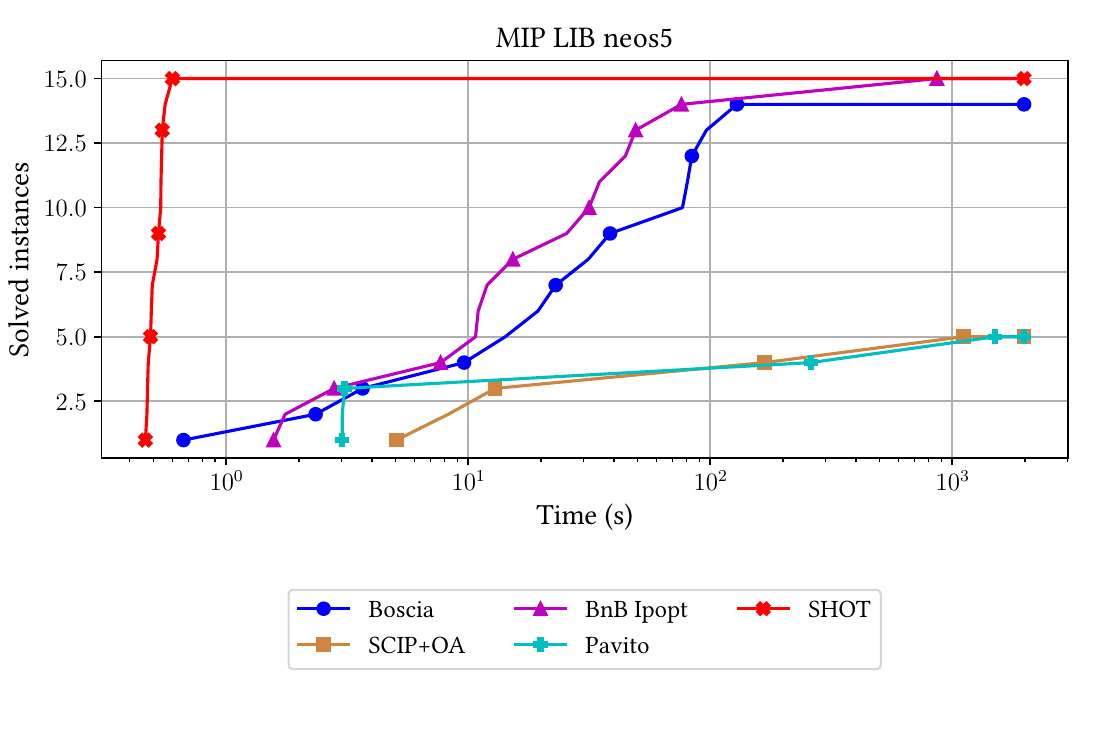}
    \end{subfigure}
    \hfill
    \begin{subfigure}[t]{1.0\textwidth}
        \centering
        \includegraphics[width=0.7\textwidth]{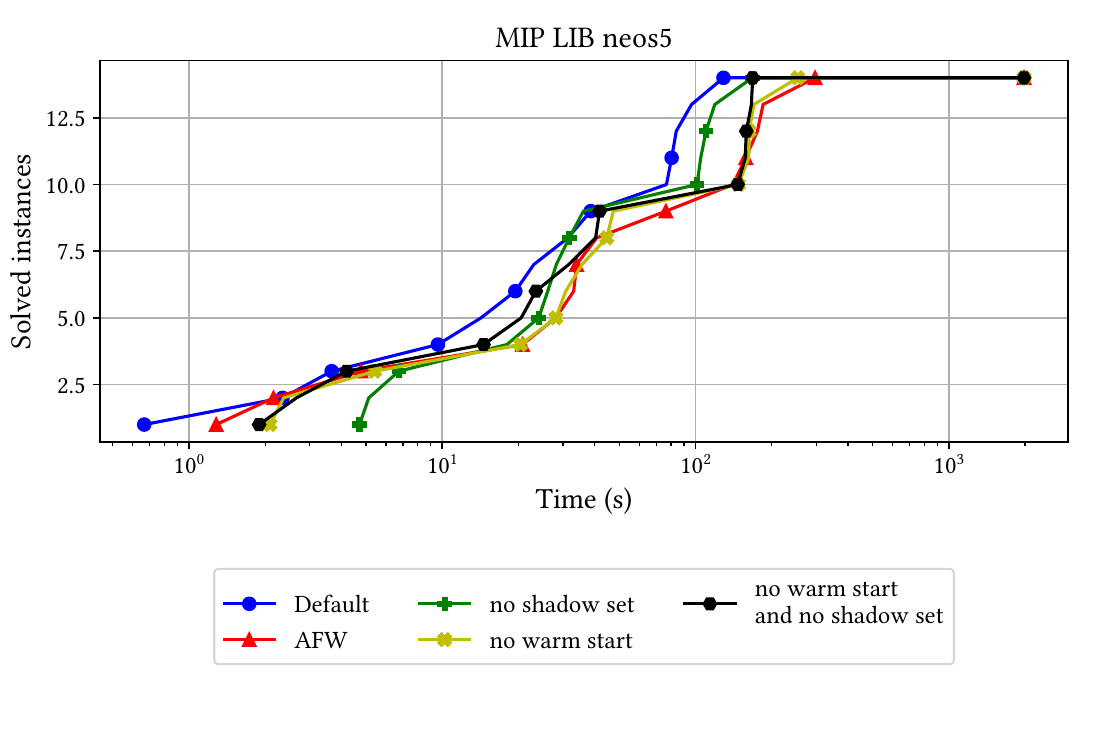}
    \end{subfigure}
    \caption{
    \revision{Comparing the number of terminations over time for the different solver set-ups and \package{} settings for the MIPLIB \texttt{neos5} Problem.}
    }
    \label{fig:MIPneos5Termination}
\end{figure}
    
\begin{figure}
    \centering
    \begin{subfigure}[t]{1.0\textwidth}
        \centering
        \includegraphics[width=0.7\textwidth]{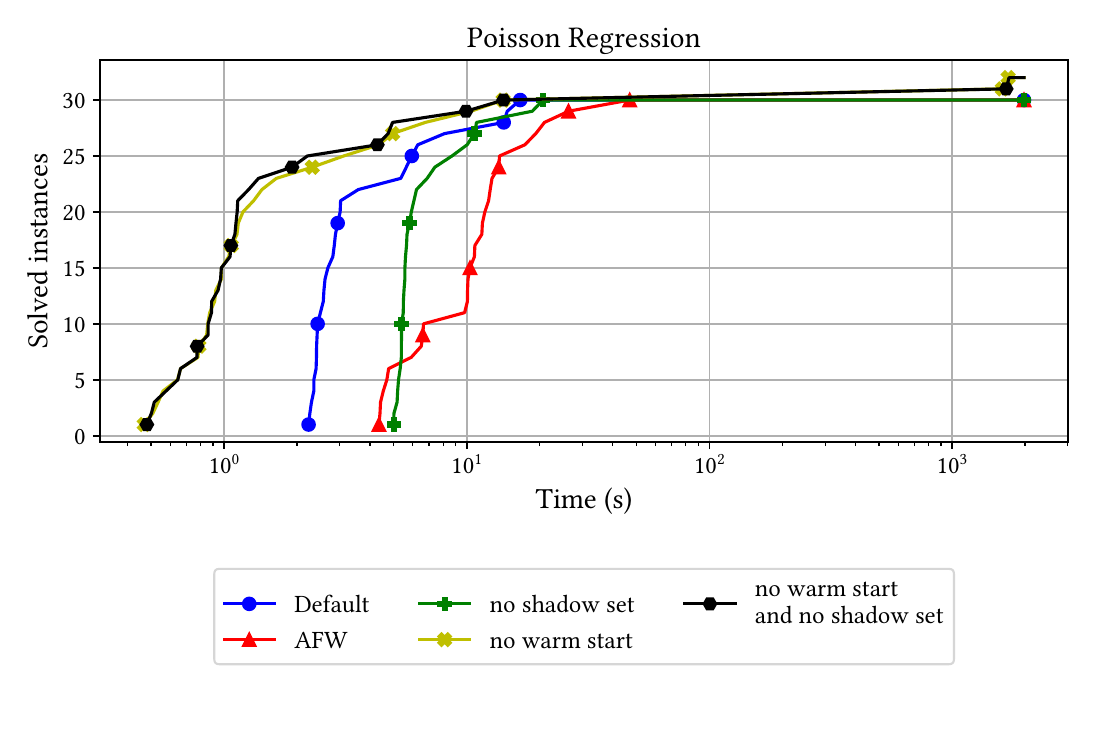}
    \end{subfigure}
    \caption{
    \revisionTwo{Comparing the number of terminations over time for the \package{} settings for the Poisson Regression Problem. For comparison between solvers, see \cref{fig:poissonRegTS}}
    }
    \label{fig:PoissonTermination}
\end{figure}

\begin{figure}
    \centering
    \begin{subfigure}[t]{1.0\textwidth}
        \centering
        \includegraphics[width=0.7\textwidth]{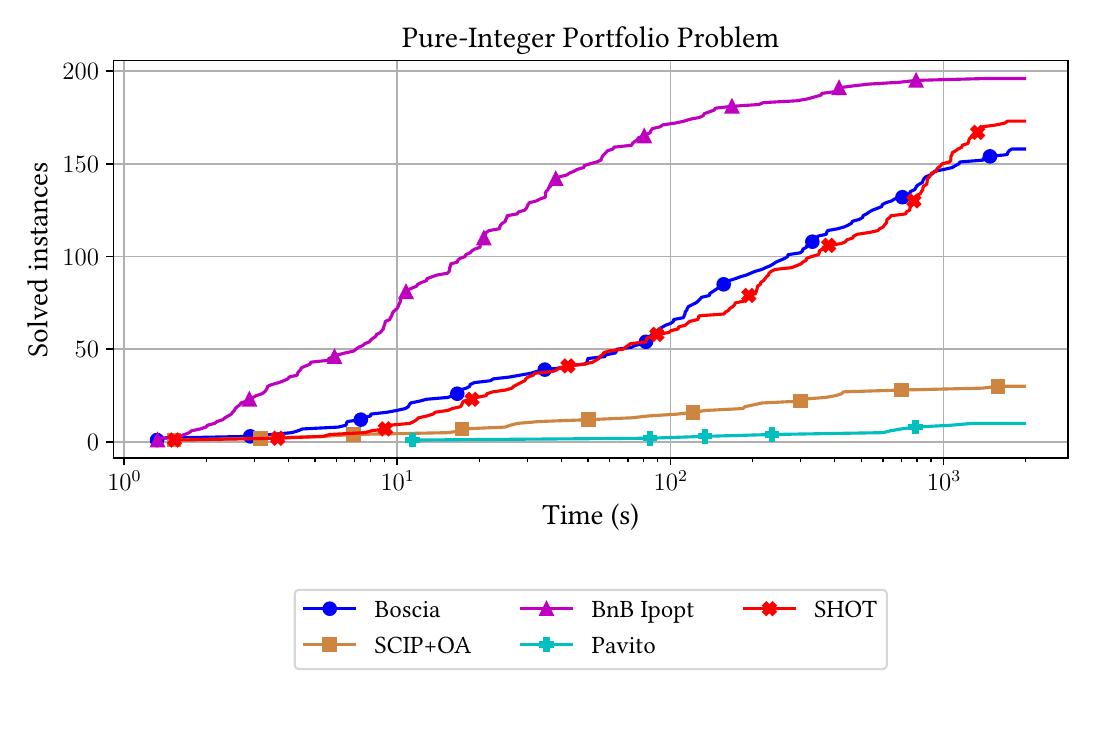}
    \end{subfigure}
    \caption{
    \revisionTwo{Comparing the number of terminations over time for the different solver set-ups for the Pure Integer Portfolio Problem. For the settings of \package{}, see \cref{fig:WarmStartIntegerPortfolio}.}
    }
    \label{fig:IntegerPortfolioTermination}
\end{figure}

\begin{figure}
    \centering
    \begin{subfigure}[t]{1.0\textwidth}
        \centering
        \includegraphics[width=0.7\textwidth]{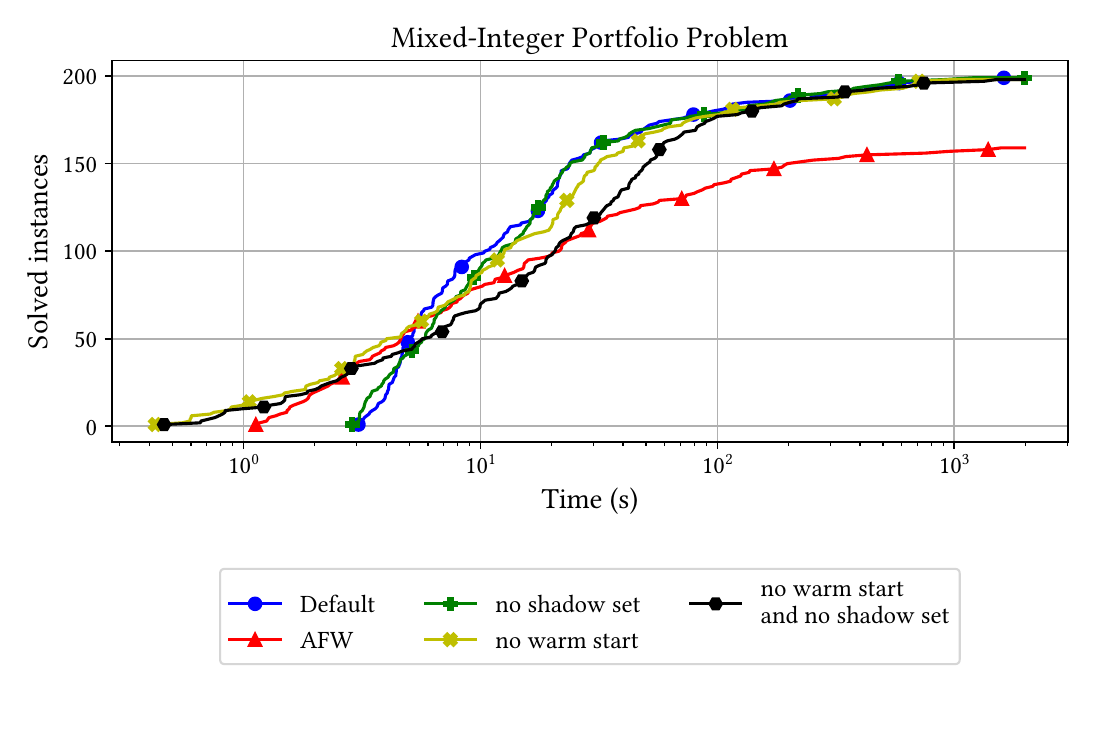}
    \end{subfigure}
    \caption{
    \revisionTwo{Comparing the number of terminations over time for the \package{} settings for the Mixed Integer Portfolio Problem. For comparison between solvers, see \cref{fig:portfolioMixedTS}.}
    }
    \label{fig:MixedPortfolioTermination}
\end{figure}

\begin{figure}
    \centering
    \begin{subfigure}[t]{1.0\textwidth}
        \centering
        \includegraphics[width=0.7\textwidth]{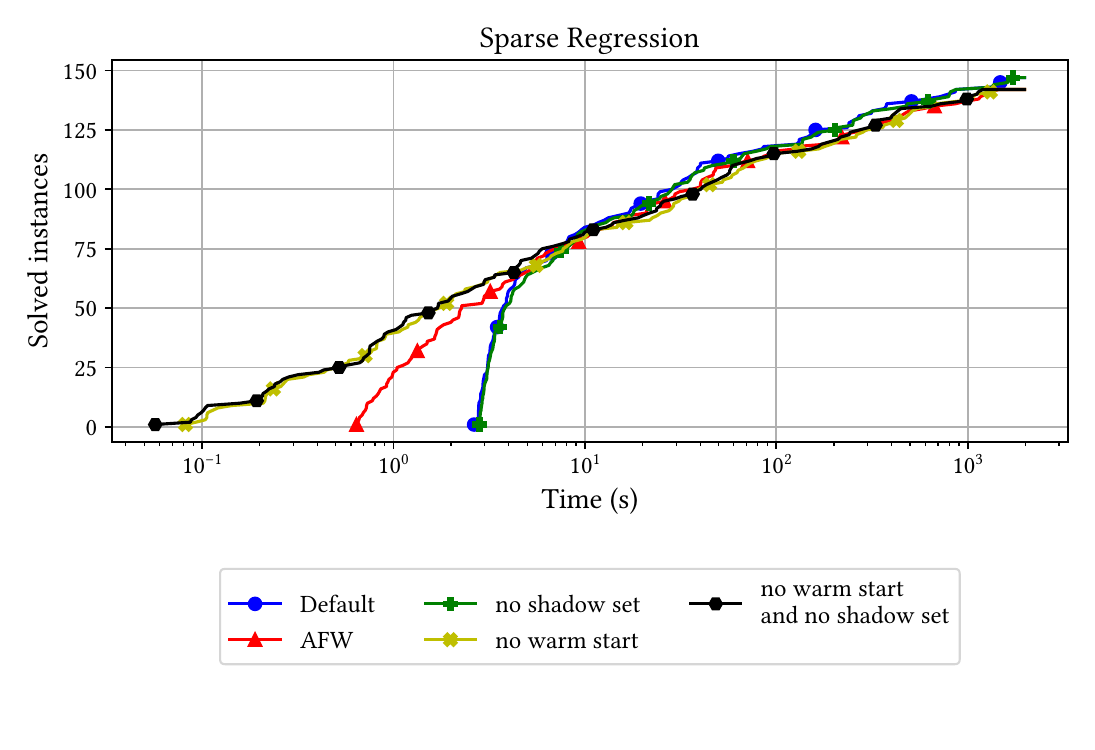}
    \end{subfigure}
    \caption{
    \revisionTwo{Comparing the number of terminations over time for the \package{} settings for the Sparse Regression Problem. For comparison between solvers, see \cref{fig:sparseRegTS}.}
    }
    \label{fig:SparseRegTermination}
\end{figure}

\begin{figure}
    \centering
    \begin{subfigure}[t]{1.0\textwidth}
        \centering
        \includegraphics[width=0.7\textwidth]{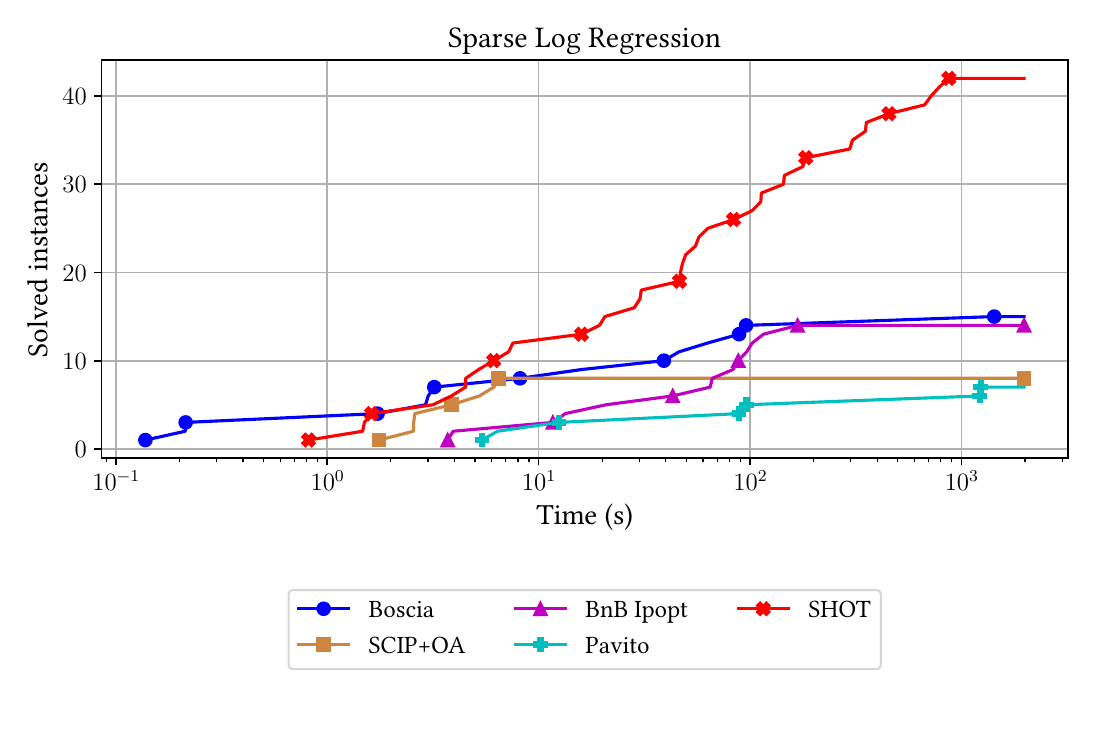}
    \end{subfigure}
    \hfill
    \begin{subfigure}[t]{1.0\textwidth}
        \centering
        \includegraphics[width=0.7\textwidth]{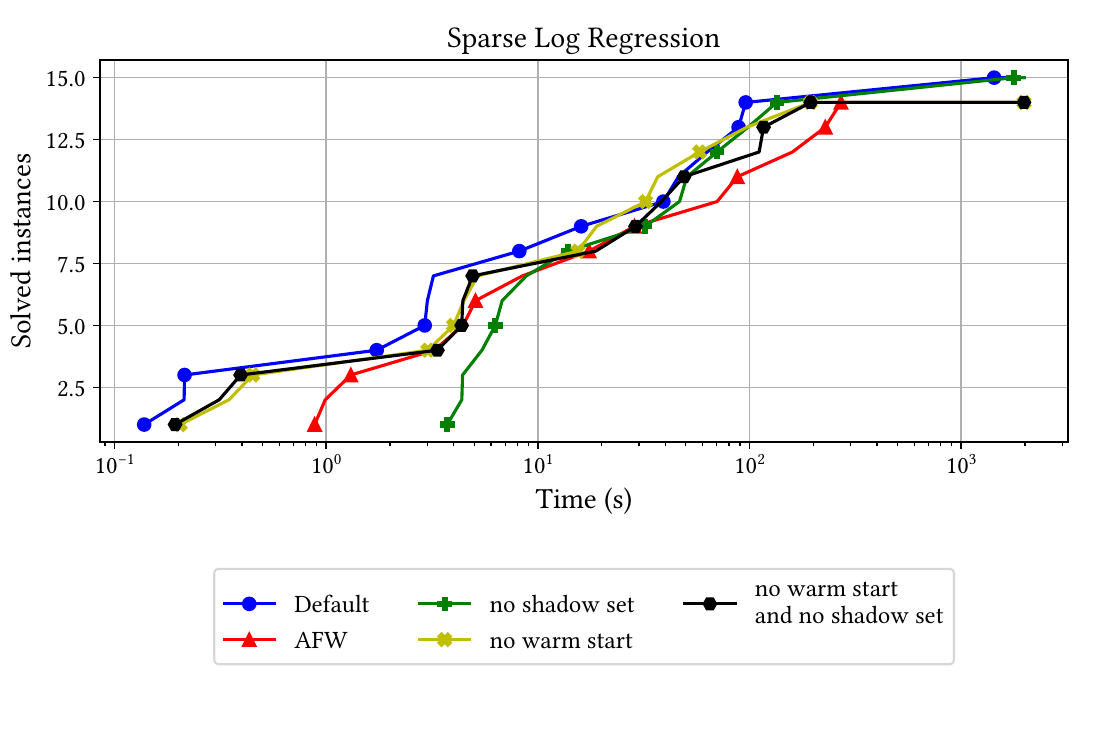}
    \end{subfigure}
    \caption{
    \revision{Comparing the number of terminations over time for the different solver set-ups and \package{} settings for the Sparse Log Regression Problem.}
    }
    \label{fig:SparseLogRegTermination}
\end{figure}

\begin{figure}
    \centering
    \begin{subfigure}[t]{1.0\textwidth}
        \centering
        \includegraphics[width=0.7\textwidth]{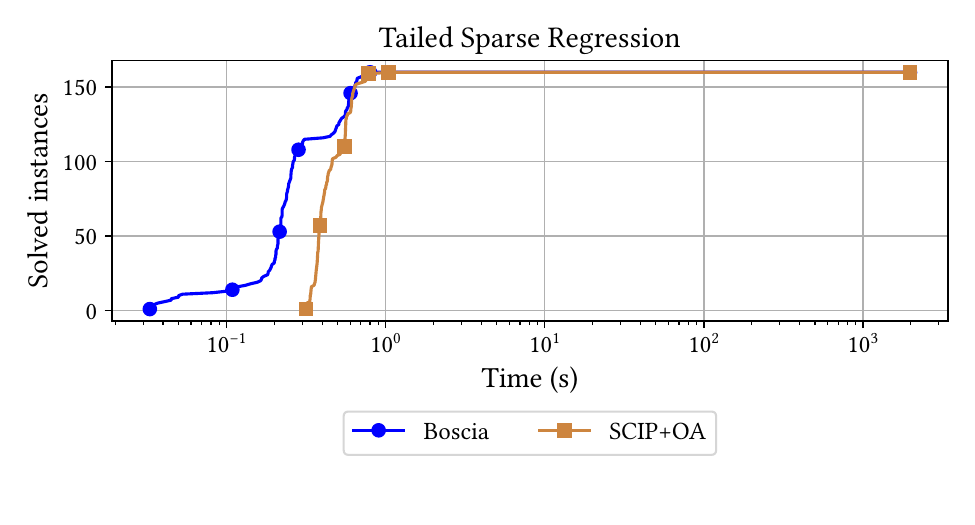}
    \end{subfigure}
    \hfill
    \begin{subfigure}[t]{1.0\textwidth}
        \centering
        \includegraphics[width=0.7\textwidth]{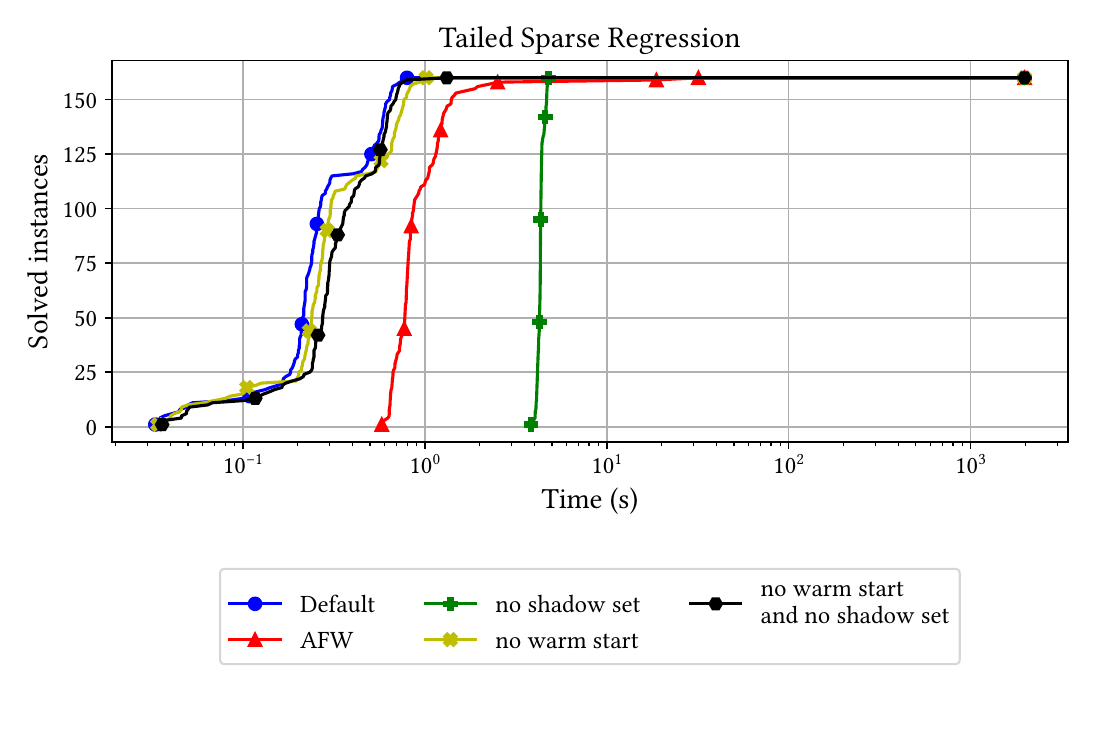}
    \end{subfigure}
    \caption{
    \revision{Comparing the number of terminations over time for the different solver set-ups and \package{} settings for the Tailed Sparse Regression Problem.}
    }
    \label{fig:TailedTermination}
\end{figure}

\begin{figure}
    \centering
    \begin{subfigure}[t]{1.0\textwidth}
        \centering
        \includegraphics[width=0.7\textwidth]{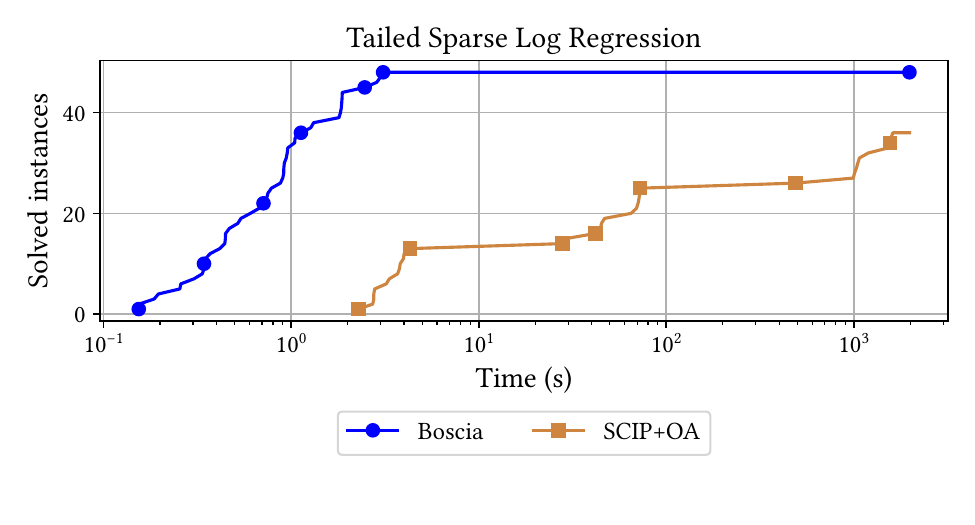}
    \end{subfigure}
    \hfill
    \begin{subfigure}[t]{1.0\textwidth}
        \centering
        \includegraphics[width=0.7\textwidth]{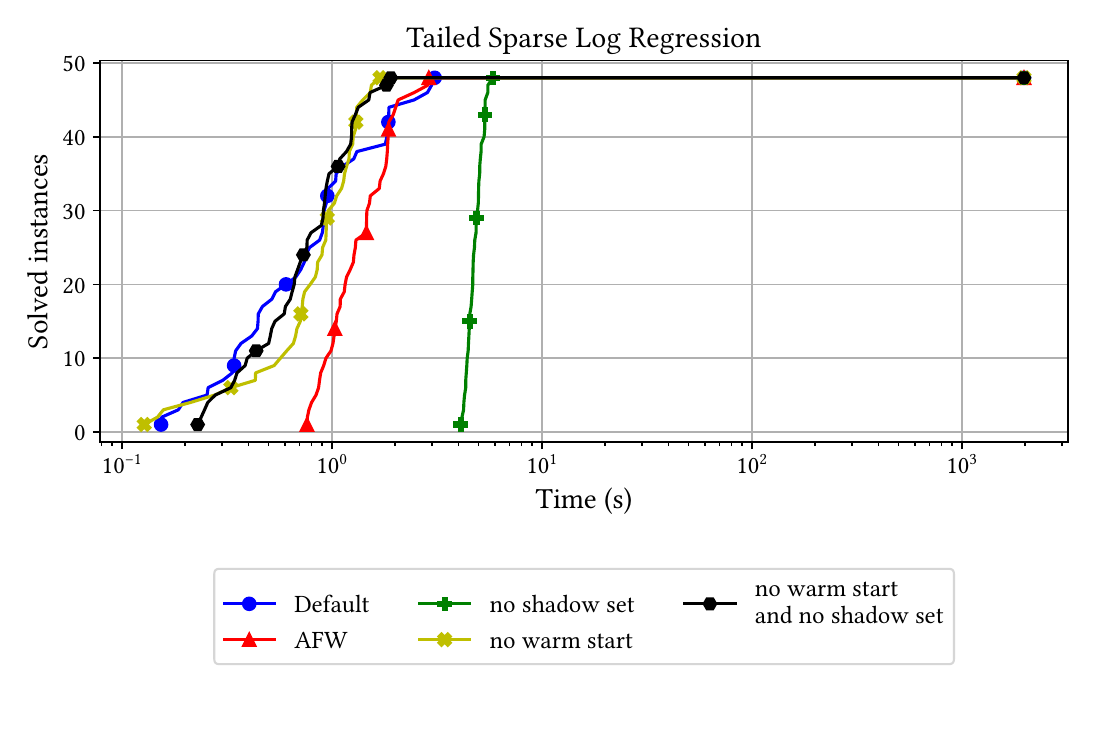}
    \end{subfigure}
    \caption{
    \revision{Comparing the number of terminations over time for the different solver set-ups and \package{} settings for the Tailed Sparse Log Regression Problem.}
    }
    \label{fig:TailedLogTermination}
\end{figure}

\begin{figure}
    \centering
    \begin{subfigure}[t]{0.49\textwidth}
        \centering
        \includegraphics[width=0.90\textwidth]{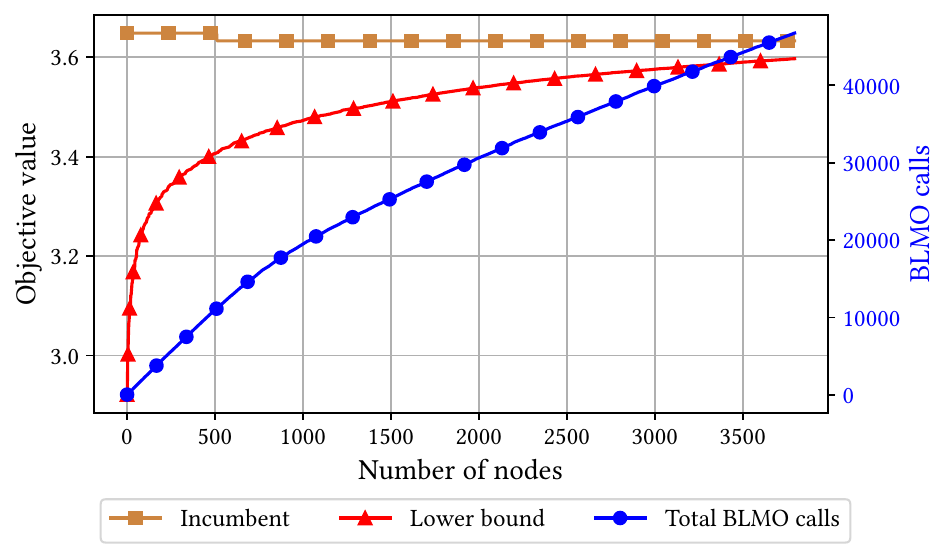}
        \caption{Number of integers variables $=95$}
    \end{subfigure}
    \hfill
    \begin{subfigure}[t]{0.49\textwidth}
        \centering
        \includegraphics[width=0.90\textwidth]{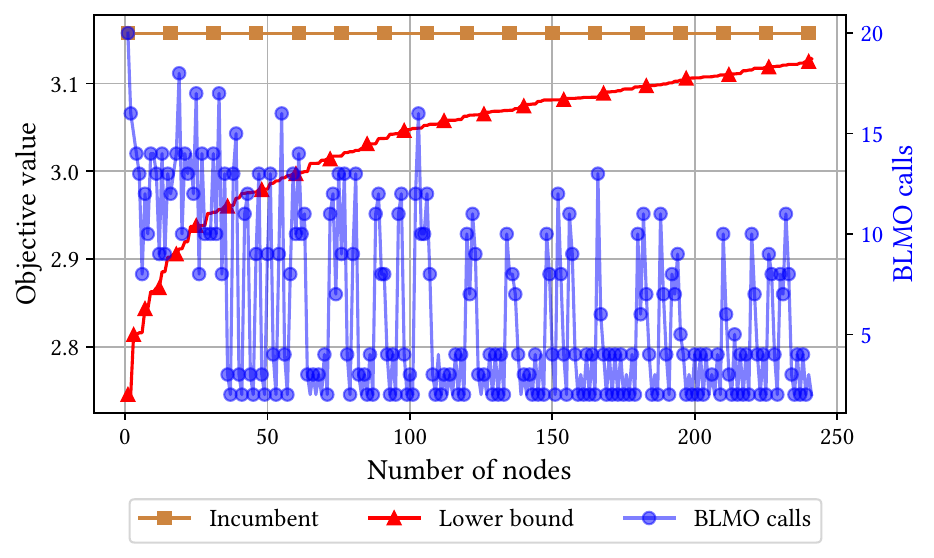}
        \caption{Number of integers variables $=75$}
    \end{subfigure}
    \caption{
    \revision{Progress plot of upper and lower bound for instances of the Sparse Regression Problem. On the left with the accumulated number of BLMO calls, and on the right with the BLMO calls per node.}
    }
    \label{fig:ProgressSparseReg}
\end{figure}

\begin{figure}
    \centering
    \begin{subfigure}[t]{0.49\textwidth}
        \centering
        \includegraphics[width=0.90\textwidth]{new_images/dual_gap_default_35_10_integer_portfolio.pdf}
        \caption{Number of integers variables $=35$}
    \end{subfigure}
    \hfill
    \begin{subfigure}[t]{0.49\textwidth}
        \centering
        \includegraphics[width=0.90\textwidth]{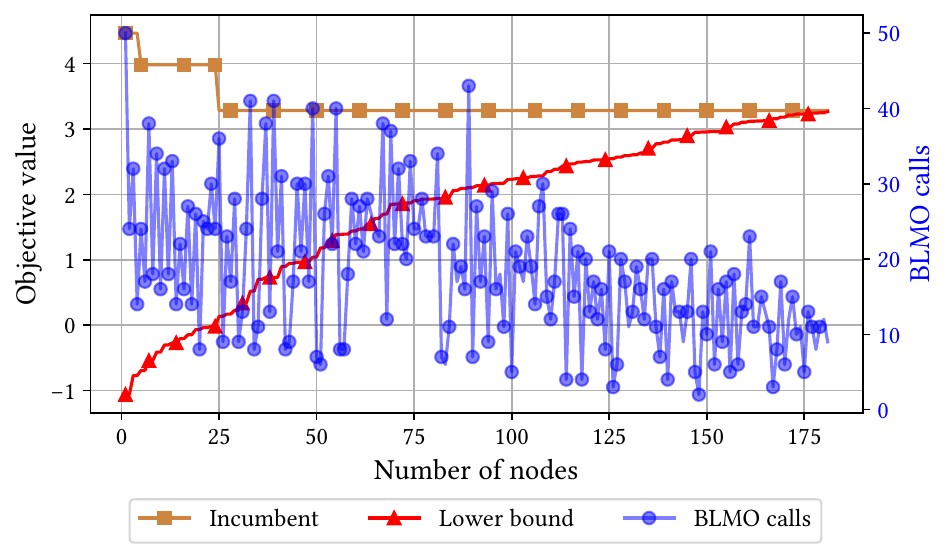}
        \caption{Number of integers variables $=20$}
    \end{subfigure}
    \caption{
    \revision{Progress plot of upper and lower bound for instances of the Pure Integer Portfolio Problem. On the left with the accumulated number of BLMO calls, and on the right with the BLMO calls per node.}
    }
    \label{fig:ProgressIntegerPortfolio}
\end{figure}

\begin{figure}
    \centering
    \begin{subfigure}[t]{0.49\textwidth}
        \centering
        \includegraphics[width=0.90\textwidth]{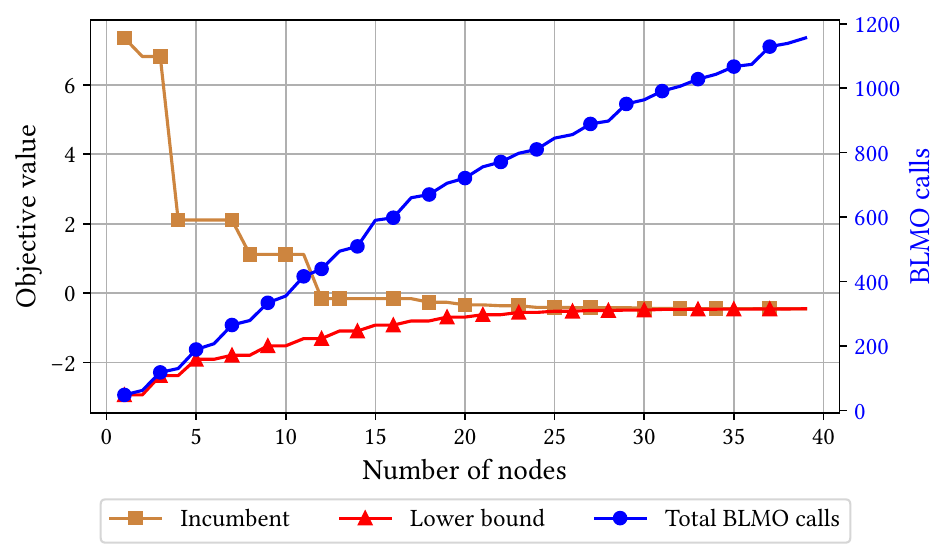}
        \caption{Number of integers variables $=22$}
    \end{subfigure}
    \hfill
    \begin{subfigure}[t]{0.49\textwidth}
        \centering
        \includegraphics[width=0.90\textwidth]{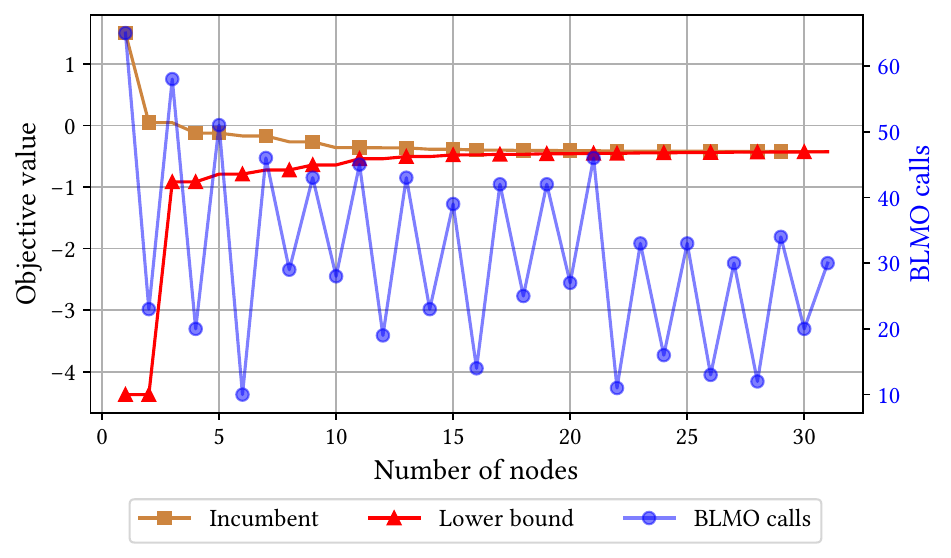}
        \caption{Number of integers variables $=27$}
    \end{subfigure}
    \caption{
    \revision{Progress plot of upper and lower bound for instances of the Mixed Integer Portfolio Problem. On the left with the accumulated number of BLMO calls, and on the right with the BLMO calls per node.}
    }
    \label{fig:ProgressMixedPortfolio}
\end{figure}

\begin{figure}
    \centering
    \begin{subfigure}[t]{0.49\textwidth}
        \centering
        \includegraphics[width=0.90\textwidth]{new_images/dual_gap_default_poisson_70_1.0-70_35.0_1.pdf}
        \caption{Number of integers variables $70$ and big M $=1.0$}
    \end{subfigure}
    \hfill
    \begin{subfigure}[t]{0.49\textwidth}
        \centering
        \includegraphics[width=0.90\textwidth]{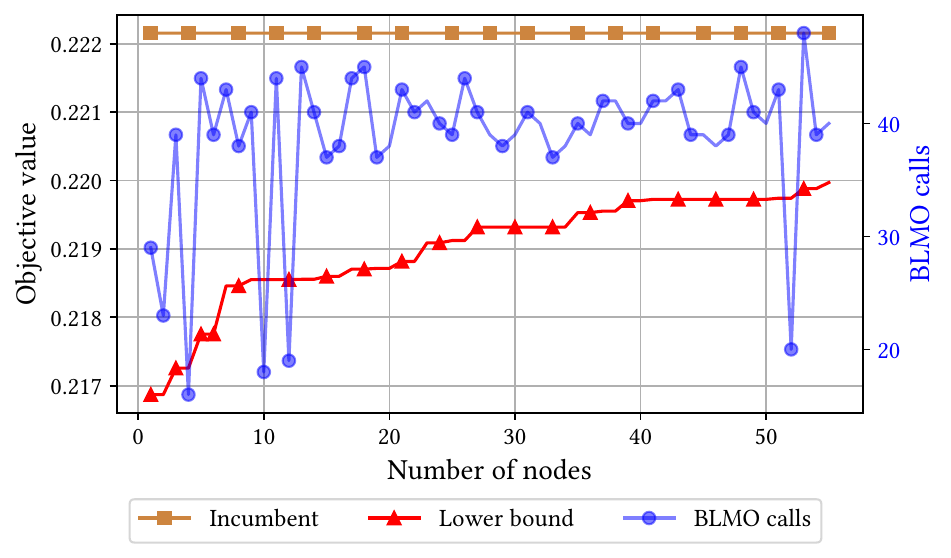}
        \caption{Number of integers variables $=70$ and big M $=10.0$}
    \end{subfigure}
    \caption{
    \revision{Progress plot of upper and lower bound for instances of the Poisson Regression Problem. On the left with the accumulated number of BLMO calls, and on the right with the BLMO calls per node.}
    }
    \label{fig:ProgressPoisson}
\end{figure}

\begin{figure}
    \centering
    \begin{subfigure}[t]{0.49\textwidth}
        \centering
        \includegraphics[width=0.90\textwidth]{new_images/dual_gap_default_sparse_log_regression_15_1.0-1_7.pdf}
        \caption{Number of integers variables $=75$}
    \end{subfigure}
    \hfill
    \begin{subfigure}[t]{0.49\textwidth}
        \centering
        \includegraphics[width=0.90\textwidth]{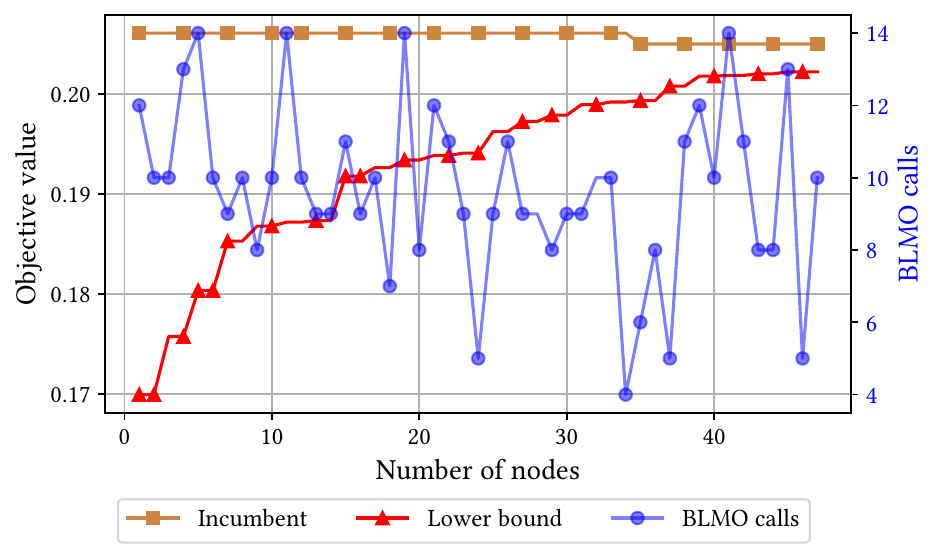}
        \caption{Number of integers variables $=25$}
    \end{subfigure}
    \caption{
    \revision{Progress plot of upper and lower bound. On the left with the accumulated number of BLMO calls, and on the right with the BLMO calls per node.}
    }
    \label{fig:ProgressSparseLogReg}
\end{figure}

\begin{figure}
    \centering
    \includegraphics[width=0.7\textwidth]{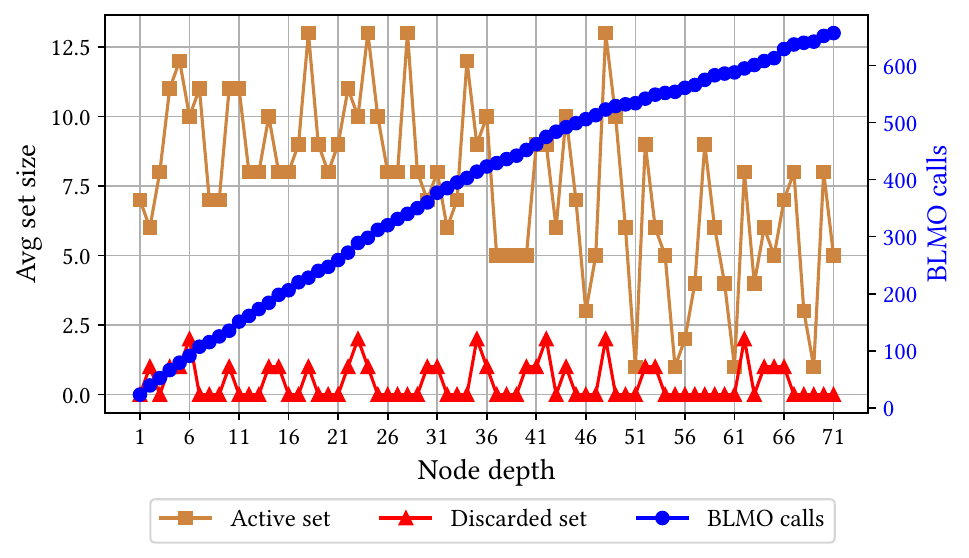}
    \caption{\revision{Evolution of the size of the active and discarded set in \package{} and the accumulated number of BLMO calls for the Sparse Regression Problem with $75$ integer variables.}}
    \label{fig:SizeASSparseReg}
\end{figure}

\begin{figure}
    \centering
    \includegraphics[width=0.7\textwidth]{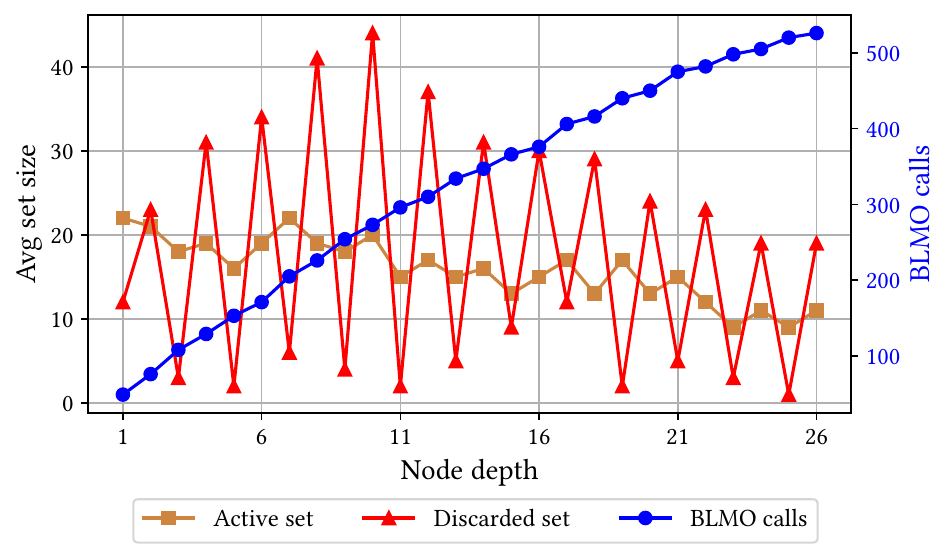}
    \caption{\revision{Evolution of the size of the active and discarded set in \package{} and the accumulated number of BLMO calls for the Mixed Integer Portfolio Problem with $15$ integer variables.}}
    \label{fig:SizeASMixedPortfolio}
\end{figure}

\begin{figure}
    \centering
    \includegraphics[width=0.7\textwidth]{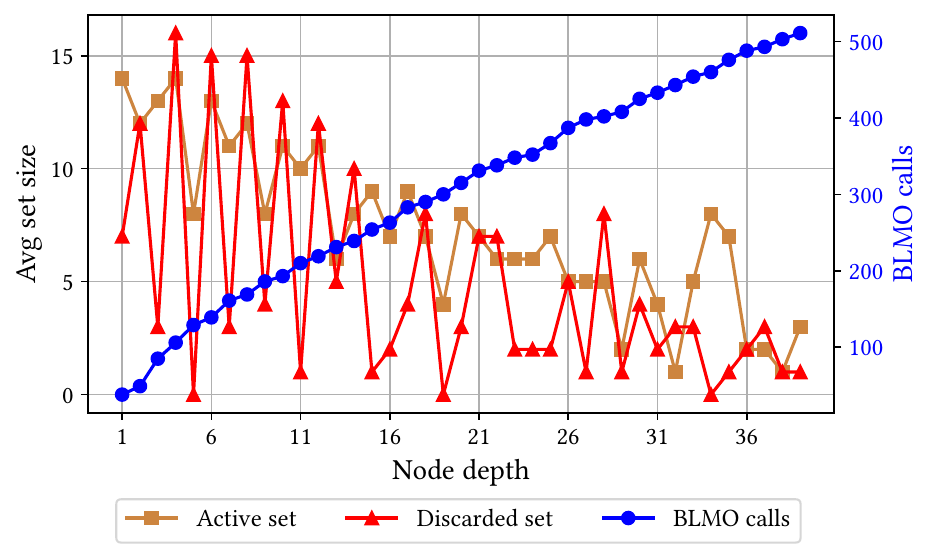}
    \caption{\revision{Evolution of the size of the active and discarded set in \package{} and the accumulated number of BLMO calls for the Pure Integer Portfolio Problem with $20$ integer variables.}}
    \label{fig:SizeASIntegerPortfolio}
\end{figure}

\begin{figure}
    \centering
    \includegraphics[width=0.7\textwidth]{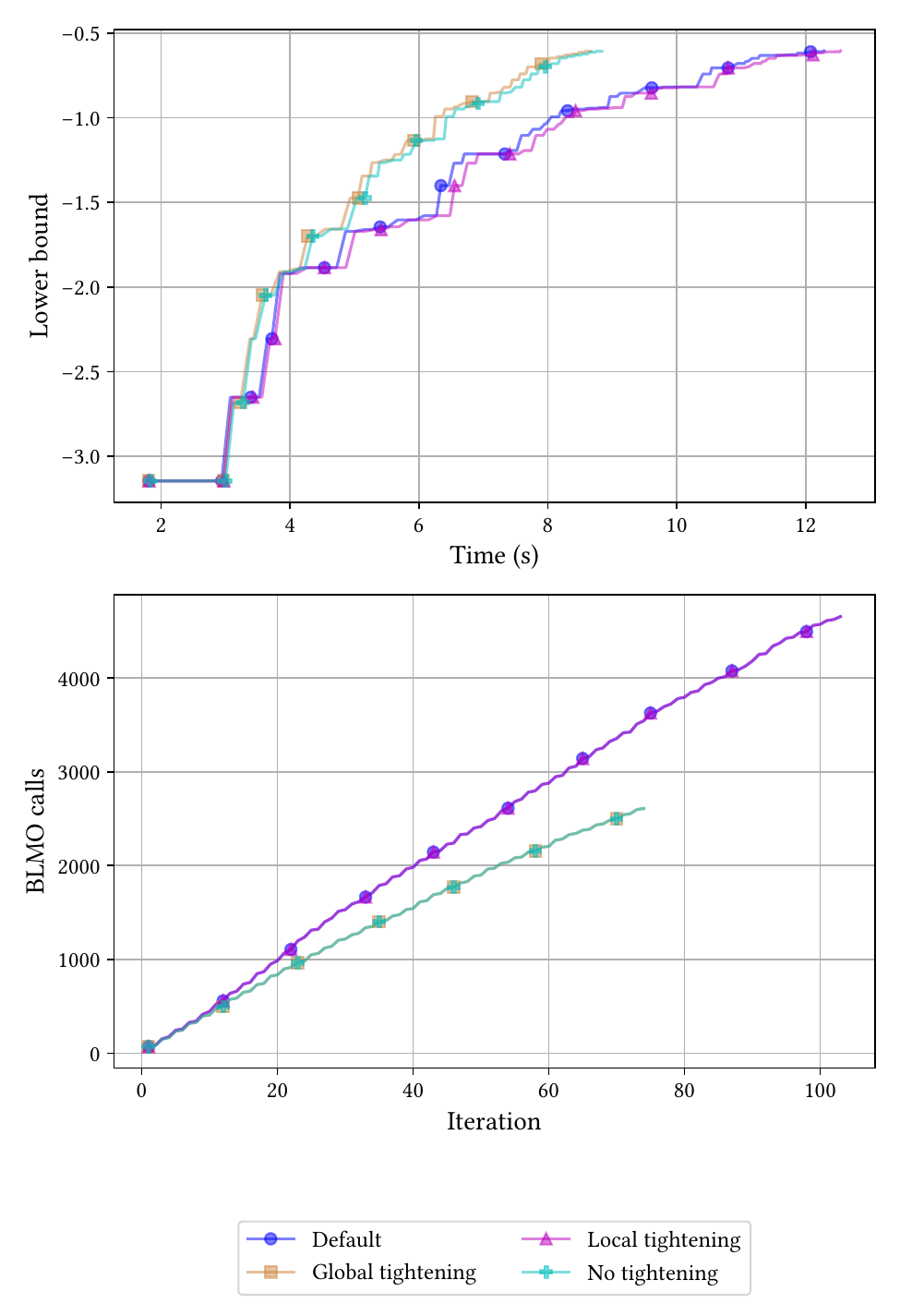}
    \caption{\revision{Comparing the effect of tightening on the lower bound and the accumulated number of BLMO calls for the Mixed Integer Portfolio Problem with $37$ integer variables.}}
    \label{fig:TighteningsMixedPortfolio}
\end{figure}

\begin{figure}
    \centering
    \includegraphics[width=0.7\textwidth]{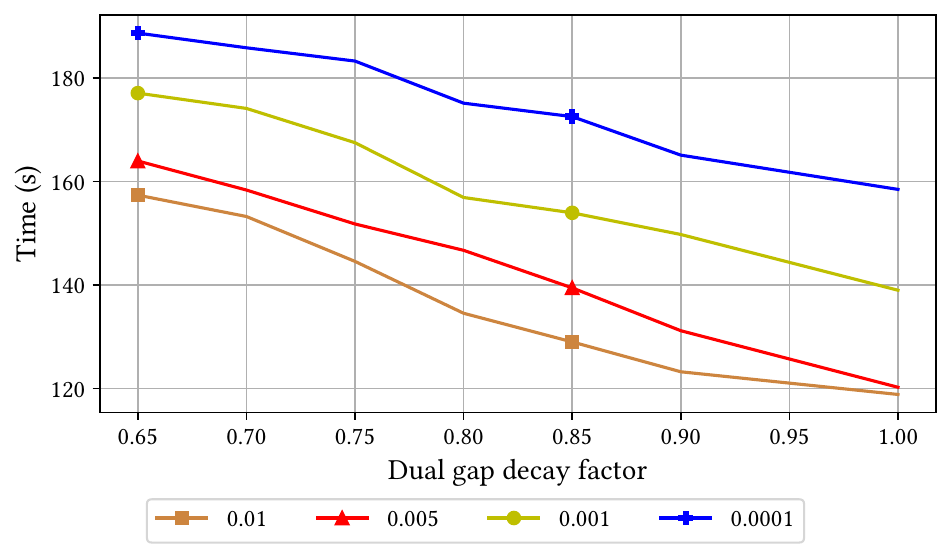}
    \caption{\revision{Comparing the effect of initial tolerance for Frank-Wolfe and dual decay factor on the solving time of the Sparse Regression Problem with $135$ integer variables.}}
    \label{fig:DualDecaySparseReg}
\end{figure}

\begin{figure}
\centering
\begin{subfigure}[t]{0.49\textwidth}
    \centering
    \includegraphics[width=0.9\textwidth]{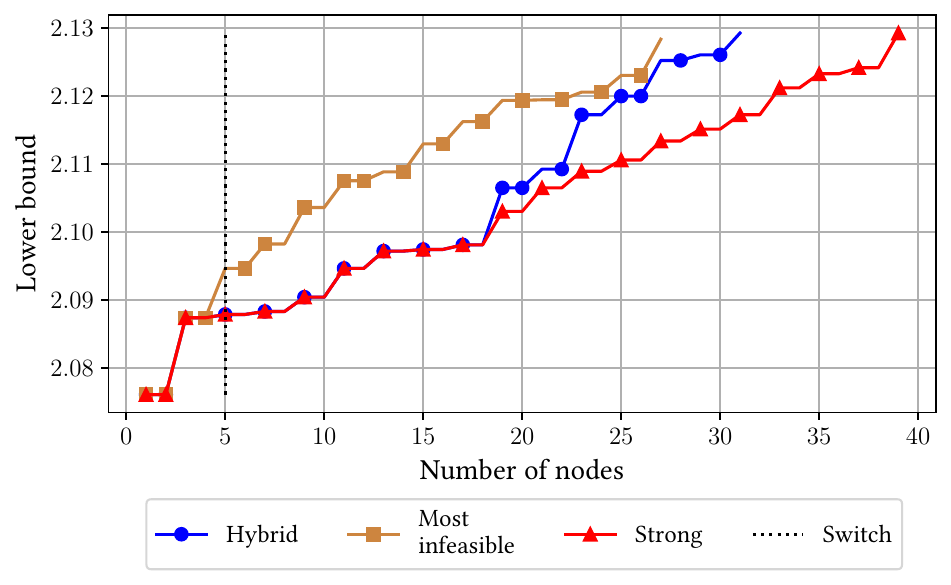}
    \caption{\revision{Lower bound progress over the number of nodes.}}
    \label{fig:BranchingNodesSparseReg}
\end{subfigure}
\hfill
\begin{subfigure}[t]{0.49\textwidth}
    \centering
    \includegraphics[width=0.9\textwidth]{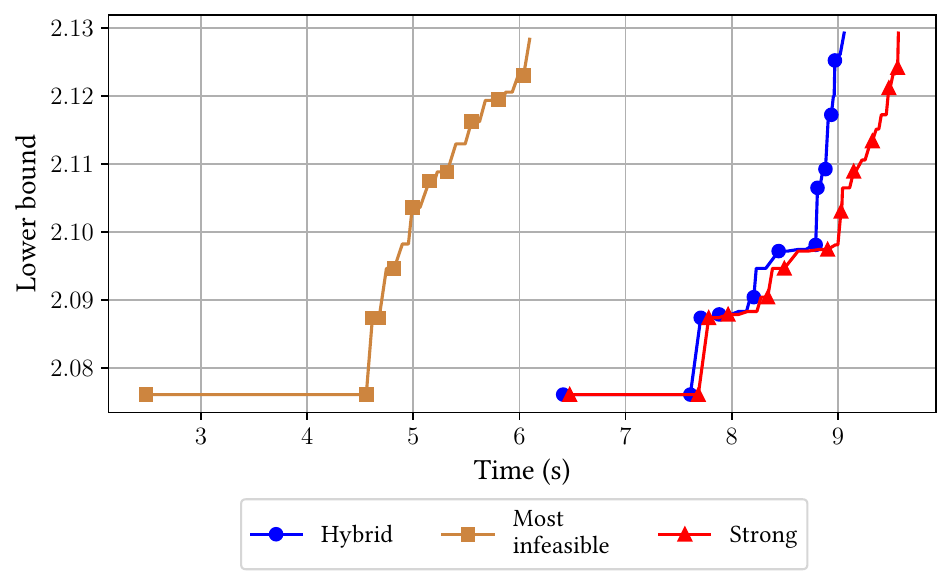}
    \caption{\revision{Lower bound progress over time.}}
    \label{fig:BranchingTimeSparseReg}
\end{subfigure}
\caption{\revision{Comparing the branching strategies on a Sparse Regression instance with 110 integer variables. In the hybrid branching, the depth setting is \# integer variables$/20$ is reached.}}
\label{fig:BranchingSparseReg}
\end{figure}

\begin{figure}
\centering
\begin{subfigure}[t]{0.49\textwidth}
    \centering
    \includegraphics[width=0.9\textwidth]{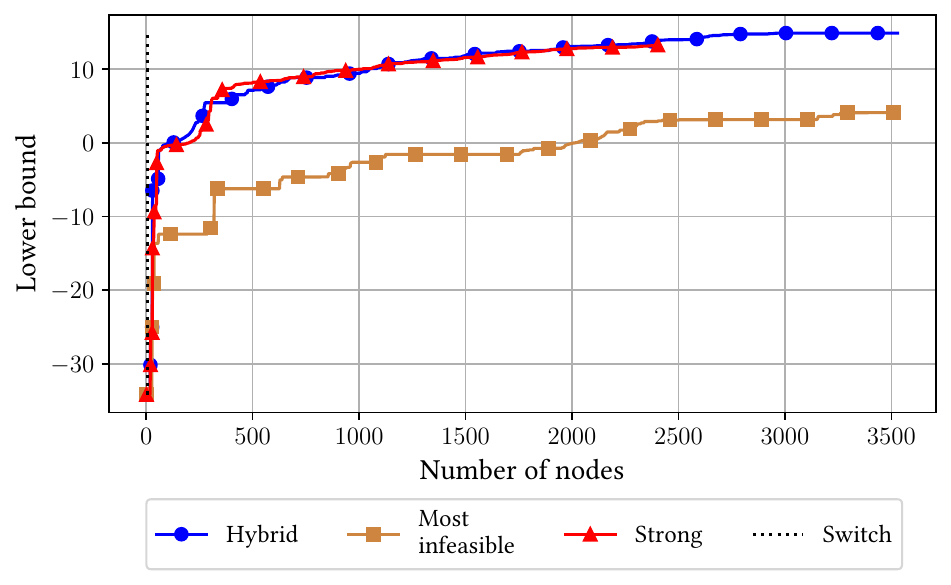}
    \caption{\revision{Lower bound progress over the number of nodes.}}
    \label{fig:BranchingNodesIntegerPortfolio}
\end{subfigure}
\hfill
\begin{subfigure}[t]{0.49\textwidth}
    \centering
    \includegraphics[width=0.9\textwidth]{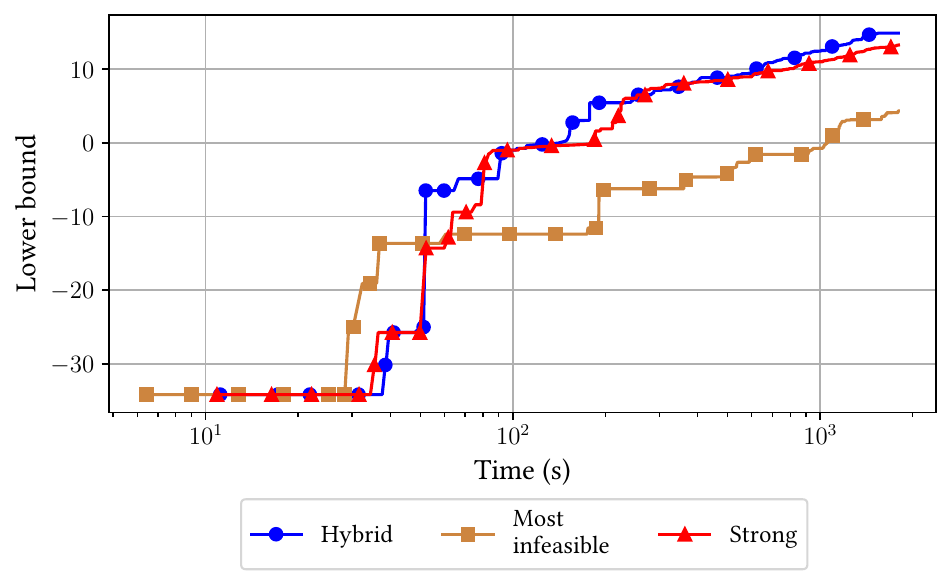}
    \caption{\revision{Lower bound progress over time.}}
    \label{fig:BranchingTimeIntegerPortfolio}
\end{subfigure}
\caption{\revision{Comparing the branching strategies on a Pure Integer Portfolio instance with 120 integer variables. In the hybrid branching, the depth setting is \# integer variables$/20$ is reached.}}
\label{fig:BranchingIntegerPortfolio}
\end{figure}

\newpage 

\begin{sidewaystable} 
    \centering
    \begin{tabular}{lll Hrrr Hrrr Hrrr Hrrr | Hrrr} 
        \toprule
        \multicolumn{3}{l}{MIPLIB} & \multicolumn{4}{c}{Boscia} & \multicolumn{4}{c}{B\&B Ipopt} & \multicolumn{4}{c}{SCIP+OA} & \multicolumn{4}{c}{Pavito} & \multicolumn{4}{c}{SHOT} \tabularnewline 
        
        \cmidrule(lr){4-7}
        \cmidrule(lr){8-11}
        \cmidrule(lr){12-15}
        \cmidrule(lr){16-19}
        \cmidrule(lr){20-23}

        \thead{Problem} & \thead{Solved \\ after \\ (s)} & \thead{\# \\ inst.} & \thead{\# \\ solved} & \thead{\% \\ solved} & \thead{Time (s)} & \thead{Relative \\ Gap} & \thead{\# \\ solved} & \thead{\% \\ solved} & \thead{Time (s)} & \thead{Relative \\ Gap} & \thead{\# \\ solved} & \thead{\% \\ solved} & \thead{Time (s)} & \thead{Relative \\ Gap} & \thead{\# \\ solved} & \thead{\% \\ solved} & \thead{Time (s)} & \thead{Relative \\ Gap} & \thead{\# \\ solved} & \thead{\% \\ solved} & \thead{Time (s)} & \thead{Relative \\ Gap} \\

        \midrule

        \texttt{22433} & 0       & 15           & 15         & \textbf{100 \%}        & 5.33       & 0.0            & 15        & \textbf{100 \%}       & 30.32     & 0.0             & 15         & \textbf{100 \%}        & \textbf{1.69} & 0.0  & 15         & \textbf{100 \%}        & 6.64       & 0.0            & 15       & \textbf{100 \%}      & 19.77    & 0.0         \\
        \midrule
        \texttt{neos 5} & 0       & 15           & 14         & 93 \%         & 30.66      & \textbf{0.01}  & 15        & \textbf{100 \%} & \textbf{19.0}      & \textbf{0.01} & 5          & 33 \%         & 517.4      & 0.04           & 5          & 33 \%         & 460.85     & 0.05           & 15       & \textbf{100 \%} & \textbf{0.51} & \textbf{0.01} \\
        \midrule
        \texttt{pg5\_34} &         0       & 15           & 0          & 0 \%          & 1800.0     & 0.02           & 0         & 0 \%         & 1800.07   & Inf           & 0          & 0 \%          & 1800.04    & 0.18           & 0          & 0 \%          & 1800.0     & Inf            & 12       & \textbf{80 \%} & \textbf{685.26} & \textbf{0.01} \\
         & 300     & 13           & 0          & 0 \%          & 1800.0     & 0.02           & 0         & 0 \%         & 1800.08   & Inf           & 0          & 0 \%          & 1800.04    & 0.18           & 0          & 0 \%          & 1800.0     & Inf            & 10       & \textbf{77 \%} & \textbf{795.4}  & \textbf{0.01} \\
         & 600     & 10           & 0          & 0 \%          & 1800.0     & 0.02           & 0         & 0 \%         & 1800.08   & Inf           & 0          & 0 \%          & 1800.04    & 0.18           & 0          & 0 \%          & 1800.0     & Inf            & 7        & \textbf{70 \%} & \textbf{896.47} & \textbf{0.01} \\
        \midrule
        \texttt{ran14x18-disj-8} &        0       & 15           & 0          & 0 \%          & 1802.86    & 0.09           & 0         & 0 \%         & 1800.16   & Inf           & 0          & 0 \%          & 1800.09    & 0.61           & 1          & \textbf{7 \%}          & \textbf{1797.79}    & 0.01           & 15       & \textbf{100 \%} & \textbf{16.65} & 0.04         \\
         & 10      & 13           & 0          & 0 \%          & 1802.67    & 0.08           & 0         & 0 \%         & 1800.18   & Inf           & 0          & 0 \%          & 1800.1     & 0.65           & 1          & \textbf{8 \%}          & \textbf{1797.45}    & 0.01           & 13       & \textbf{100 \%} & \textbf{18.19} & 0.03         \\
         \bottomrule
    \end{tabular}
    \caption{\footnotesize \revisionTwo{Four of the MIPLIB instances with a quadratic objective formed by 4 and 8 randomly chosen vertices. Instance \texttt{22433} is a mixed binary problem with 231 binary variables and 429 variables in total. Problem \texttt{neos 5} is a mixed binary problem with 53 binary variables and 63 variables in total. \texttt{pg5\_34} is a mixed binary problem with 100 binary variables and 2600 variables in total. Lastly, \texttt{ran14x18-disj-8} is a mixed binary problem with 252 binary variables and 504 variables in total.}} 
    \label{tab:SummaryByDifficultyMIPLIP}
\vspace*{1cm}
    \centering
    \begin{tabular}{lll Hrrr Hrrr Hrrr Hrrr | Hrrr} 
        \toprule
        \multicolumn{3}{l}{Portfolio} & \multicolumn{4}{c}{Boscia} & \multicolumn{4}{c}{B\&B Ipopt} & \multicolumn{4}{c}{SCIP+OA} & \multicolumn{4}{c}{Pavito} & \multicolumn{4}{c}{SHOT} \tabularnewline 
        
        \cmidrule(lr){4-7}
        \cmidrule(lr){8-11}
        \cmidrule(lr){12-15}
        \cmidrule(lr){16-19}
        \cmidrule(lr){20-23}

        \thead{Type} & \thead{Solved \\ after \\ (s)} & \thead{\# \\ inst.} & \thead{\# \\ solved} & \thead{\% \\ solved} & \thead{Time (s)} & \thead{Relative \\ Gap} & \thead{\# \\ solved} & \thead{\% \\ solved} & \thead{Time (s)} & \thead{Relative \\ Gap} & \thead{\# \\ solved} & \thead{\% \\ solved} & \thead{Time (s)} & \thead{Relative \\ Gap} & \thead{\# \\ solved} & \thead{\% \\ solved} & \thead{Time (s)} & \thead{Relative \\ Gap} & \thead{\# \\ solved} & \thead{\% \\ solved} & \thead{Time (s)} & \thead{Relative \\ Gap} \\

        \midrule
        Pure Integer & 0       & 210          & 158        & 75 \%         & 236.29     & \textbf{1.86}  & 196       & \textbf{93 \%}        & \textbf{26.97} & 1.89          & 30         & 14 \%         & 1162.41    & 54.59          & 10         & 5 \%          & 1656.61    & 0.01           & 173      & 82 \% & 266.83           & 1.89         \\
         & 10      & 138          & 86         & 62 \%         & 707.09     & \textbf{2.94}  & 124       & \textbf{90 \%}        & \textbf{64.69} & 3.05          & 0          & 0 \%          & 1800.0     & 88.38          & 0          & 0 \%          & 1800.0     & Inf            & 101      &  73 \% & 708.59           & 2.99         \\
         & 300     & 26           & 0          & 0 \%          & 1800.0     & 1.42           & 12        & \textbf{46 \%}        & \textbf{985.21}         & 1.47          & 0          & 0 \%          & 1800.0     & 1068.68        & 0          & 0 \%          & 1800.0     & Inf            & 4       & 15 \% & 1671.56 & 1.81         \\
         & 600     & 17           & 0          & 0 \%          & 1800.0     & 1.41           & 3         & \textbf{18 \%}        & \textbf{1598.82}        & 1.72          & 0          & 0 \%          & 1800.0     & 2844.43        & 0          & 0 \%          & 1800.0     & Inf            & 17       & 0 \% & 1800.0           & 2.04         \\
         & 1200    & 15           & 0          & 0 \%          & 1800.0     & 1.37           & 1         & \textbf{7 \%}         & \textbf{1771.51}        & Inf           & 0          & 0 \%          & 1800.0     & 4265.04        & 0          & 0 \%          & 1800.0     & Inf            & 15       & 0 \% & 1800.0           & 2.08  \\
         \midrule
         Mixed Integer & 0       & 210          & 199        & 95 \%         & 19.47      & \textbf{0.72}  & 209       & \textbf{100 \%} & \textbf{2.82}    & 38.69         & 45         & 21 \%         & 903.36     & 10034.36       & 125        & 60 \%         & 137.68     & 64.96          & 62      & 30 \% & 842.95   & 38.69         \\
          & 10      & 12           & 5          & 42 \%         & 889.18     & 0.36           & 11        & \textbf{92 \%}           & \textbf{128.81}  & \textbf{0.19} & 0          & 0 \%          & 1800.0     & 1919.83        & 0          & 0 \%          & 1800.0     & Inf            & 1       & 8 \% & 1750.16  & 0.2           \\
          & 300     & 4            & 0          & 0 \%          & 1800.0     & 0.28           & 3         & \textbf{75 \%}           & \textbf{876.48}  & \textbf{0.21} & 0          & 0 \%          & 1800.0     & 1919.83        & 0          & 0 \%          & 1800.0     & Inf            & 0        & 0 \% & 1800.0   & \textbf{0.21} \\
          & 600     & 3            & 0          & 0 \%          & 1800.0     & 0.35           & 2         & \textbf{67 \%}           & \textbf{1152.68} & \textbf{0.26} & 0          & 0 \%          & 1800.0     & Inf            & 0          & 0 \%          & 1800.0     & Inf            & 0        & 0 \% & 1800.0   & 0.27          \\
          & 1200    & 1            & 0          & 0 \%          & 1800.0     & 0.31           & 0         & 0 \%            & 1800.0           & \textbf{0.23} & 0          & 0 \%          & 1800.0     & Inf            & 0          & 0 \%          & 1800.0     & Inf            & 0        & 0 \% & 1800.0   & 0.25   \\
        \bottomrule
    \end{tabular}
    \caption{\footnotesize \revisionTwo{The two types of Integer Portfolio problems. The Pure Integer Portfolio Problem has between 20 to 120 integral variables. The Mixed Integer Portfolio problem has 20 to 120 variables in total, half of them being integer.}} 
    \label{tab:SummaryByDifficultyPortfolio}
\end{sidewaystable}

\begin{sidewaystable}
    \centering
    \begin{tabular}{lll Hrrr Hrrr Hrrr Hrrr | Hrrr} 
        \toprule
        \multicolumn{3}{l}{Regression} & \multicolumn{4}{c}{Boscia} & \multicolumn{4}{c}{B\&B Ipopt} & \multicolumn{4}{c}{SCIP+OA} & \multicolumn{4}{c}{Pavito} & \multicolumn{4}{c}{SHOT} \tabularnewline 
        
        \cmidrule(lr){4-7}
        \cmidrule(lr){8-11}
        \cmidrule(lr){12-15}
        \cmidrule(lr){16-19}
        \cmidrule(lr){20-23}

        \thead{type} & \thead{Solved \\ after \\ (s)} & \thead{\# \\ inst.} & \thead{\# \\ solved} & \thead{\% \\ solved} & \thead{Time (s)} & \thead{Relative \\ Gap} & \thead{\# \\ solved} & \thead{\% \\ solved} & \thead{Time (s)} & \thead{Relative \\ Gap} & \thead{\# \\ solved} & \thead{\% \\ solved} & \thead{Time (s)} & \thead{Relative \\ Gap} & \thead{\# \\ solved} & \thead{\% \\ solved} & \thead{Time (s)} & \thead{Relative \\ Gap} & \thead{\# \\ solved} & \thead{\% \\ solved} & \thead{Time (s)} & \thead{Relative \\ Gap} \\

        \midrule
        Poisson Regression & 0 & 120 & 30 & 25 \% & 406.01 & 3.79 & 34 & 28 \% & 516.04 & 0.17 & 116 & \textbf{81 \%} & \textbf{81.74} & 0.4 & 29 & 24 \% & 885.84 & 0.22 & 81 & 63 \% & 297.89 & \textbf{0.04} \\
         & 10 & 80 & 3 & 4 \% & 1508.16 & 6.09 & 7 & 9 \% & 1343.59 & 0.26 & 76 & \textbf{80 \%} & \textbf{125.57} & 0.58 & 26 & 32 \% & 708.2 & 0.24 & 67 & \textbf{84 \%} & 247.27 & \textbf{0.05} \\
         & 300 & 7 & 0 & 0 \% & 1800.03 & 75.92 & 0 & 0 \% & 1800.27 & 0.95 & 4 & \textbf{57 \%} & \textbf{732.71} & 3.66 & 1 & 14 \% & 1649.29 & 4.91 & 6 & \textbf{86 \%} & 774.04 & \textbf{0.2} \\
         & 600 & 2 & 0 & 0 \% & 1800.04 & 151.49 & 0 & 0 \% & 1800.25 & 1.32 & 0 & 0 \% & 1800.0 & 5.24 & 1 & 50 \% & 1325.4 & 4.91 & 2 & \textbf{100 \%} & \textbf{1190.21} & \textbf{0.46} \\
        \midrule
        Sparse Regression & 0       & 160          & 147        & \textbf{92 \%}         & \textbf{25.14}      & \textbf{0.01}  & 137       & 86 \%        & 48.01     & \textbf{0.01} & 1        & 1 \% & 1720.39 & 0.35           & 56         & 35 \%         & 547.69     & \textbf{0.01}  & 111      & 69 \% & 252.44   & \textbf{0.01} \\
        \midrule
        Sparse Log Regression & 0       & 48           & 15         & \textbf{31 \%}         & \textbf{387.38}     & 0.19           & 14        & 29 \%        & 592.41    & 0.17          & 8         & 17 \% & 665.15 & 0.53           & 7          & 15 \%         & 1121.07    & \textbf{0.02}  & 42       & \textbf{88 \%} & \textbf{69.9}     & 0.08         \\
        \bottomrule
    \end{tabular}
    \caption{\footnotesize \revisionTwo{The Regression problems. The Sparse Regression problem has between 150 and 300 variables with half of them being binary variables. The Sparse Log Regression problem has between 50 and 200 variables with half of them being binary variables. The Poisson Regression problem has between 50 and 100 binary variables and the same number plus 1 of continuous variables.}} 
    \label{tab:SummaryByDifficultyRegression}
\vspace*{1cm}
    \centering
    \begin{tabular}{lll Hrrrr Hrrrr } 
        \toprule
        \multicolumn{3}{l}{Tailed Regression} & \multicolumn{5}{c}{Boscia} & \multicolumn{5}{c}{SCIP+OA}  \tabularnewline 
        
        \cmidrule(lr){4-8}
        \cmidrule(lr){9-13}

        \thead{Type} & \thead{Solved \\ after \\ (s)} & \thead{\# \\ inst.} & \thead{\# \\ solved} & \thead{\% \\ solved} & \thead{Time (s)} & \thead{Relative \\ Gap} & \thead{\# \\ nodes \\ cuts} & \thead{\# \\ solved} & \thead{\% \\ solved} & \thead{Time (s)} & \thead{Relative \\ Gap} & \thead{\# \\ nodes \\ cuts} \\

        \midrule
        Tailed Sparse & 0       & 160          & 160        & \textbf{100 \%} & \textbf{0.3} & \textbf{0.0}   & 59               & 79       &  49 \% & 52.59       & \textbf{0.0}    & 41               \\
        \midrule
        Tailed Sparse Log & 0       & 48           & 48         & \textbf{100 \%} & \textbf{0.85} & \textbf{0.0}   & 36               & 36         & 75 \%         & 129.89     & 109.2          & 7130             \\
        \bottomrule   
    \end{tabular}
    \caption{\footnotesize \revisionTwo{The Tailed Regression problems. The Tailed Sparse Regression problem has 2 or 3 binary variables and 4 to 9 variables in total. The Tailed Sparse Log Regression problem has between 20 and 100 binary variables and 60 to 300 variables in total.}} 
    \label{tab:SummaryByDifficultyTailedReg}
\end{sidewaystable}

\begin{sidewaystable}
    \centering
    \begin{tabular}{cl Hrrr Hrrr Hrrr Hrrr Hrrr} 
        \toprule
        \multicolumn{2}{l}{} & \multicolumn{4}{c}{Default} & \multicolumn{4}{c}{\thead{Global \\ Tightening}} & \multicolumn{4}{c}{\thead{Local \\ Tightening}} & \multicolumn{4}{c}{\thead{No \\ Tightening}} & \multicolumn{4}{c}{\thead{Strong \\ Convexity}}\tabularnewline 
        
        \cmidrule(lr){3-6}
        \cmidrule(lr){7-10}
        \cmidrule(lr){11-14}
        \cmidrule(lr){15-18}
        \cmidrule(lr){19-22}

        Problem & \thead{\# \\ inst.} & \thead{\# \\ solved} & \thead{\% \\ solved} & \thead{Time (s)} & \thead{Relative \\ Gap} & \thead{\# \\ solved} & \thead{\% \\ solved} & \thead{Time (s)} & \thead{Relative \\ Gap} & \thead{\# \\ solved} & \thead{\% \\ solved} & \thead{Time (s)} & \thead{Relative \\ Gap} & \thead{\# \\ solved} & \thead{\% \\ solved} & \thead{Time (s)} & \thead{Relative \\ Gap} & \thead{\# \\ solved} & \thead{\% \\ solved} & \thead{Time (s)} & \thead{Relative \\ Gap}  \\

        \midrule
        \thead{MIPLIB \texttt{22433}}  & 15 & 15 & \textbf{100 \%} & \textbf{5.33} & 0.0 & 15 & \textbf{100 \%} & 9.81 & 0.0 & 15 & \textbf{100 \%} & 10.26 & 0.0 & 15 & \textbf{100 \%} & 10.04 & 0.0 & 15 & \textbf{100 \%} & 9.73 & 0.0 \\
        \midrule
        \thead{MIPLIB \texttt{neos5}}  & 15 & 14 & \textbf{93 \%} & \textbf{30.66} & 0.01 & 14 & \textbf{93 \%} & 44.21 & 0.01 & 14 & \textbf{93 \%} & 42.82 & 0.01 & 14 & \textbf{93 \%} & 45.73 & 0.01 & 14 & \textbf{93 \%} & 33.28 & 0.01 \\
        \midrule
        \thead{MIPLIB \texttt{pg5\_34}}  & 15 & 0 & 0 \% & 1800.0 & 0.02 & 0 & 0 \% & 1800.0 & 0.02 & 0 & 0 \% & 1800.0 & 0.02 & 0 & 0 \% & 1800.0 & 0.02 & 0 & 0 \% & 1800.0 & 0.02 \\
        \midrule
        \thead{MIPLIB \texttt{ran14x18-disj-8}} & 15 & 0 & 0 \% & 1802.86 & 0.09 & 0 & 0 \% & 1804.77 & 0.09 & 0 & 0 \% & 1804.7 & 0.09 & 0 & 0 \% & 1804.79 & 0.09 & 0 & 0 \% & 1804.71 & 0.09 \\
        \midrule
        \thead{Poisson \\ Regression} & 120 & 30 & \textbf{25 \%} & \textbf{406.01} & 3.79 & 30 & \textbf{25 \%} & 451.8 & \textbf{0.31} & 30 & \textbf{25 \%} & 447.97 & \textbf{0.31} & 30 & \textbf{25 \%} & 501.87 & \textbf{0.31}  & & & & \\
        \midrule
        \thead{Integer \\ Portfolio}  & 210 & 158 & 75 \% & 236.29 & 1.86 & 175 & \textbf{83 \%} & \textbf{208.29} & \textbf{1.86} & 154 & 73 \% & 312.1 & 1.86 & 175 & \textbf{83 \%} & 214.79 & \textbf{1.86}  & & & &  \\
        \midrule
        \thead{Mixed  \\ Portfolio} & 210 & 199 & 95 \% & 19.47 & 0.72 & 202 & \textbf{96 \%} & \textbf{14.7} & 0.17 & 199 & 95 \% & 19.67 & 0.72 & 202 & \textbf{96 \%} & 16.49 & 0.17   & & & & \\
        \midrule
        \thead{Sparse \\ Regression} & 160 & 147 & \textbf{92 \%} & \textbf{25.14} & \textbf{0.01} & 144 & 90 \% & 38.81 & \textbf{0.01} & 147 & \textbf{92 \%} & 26.18 & \textbf{0.01} & 145 & 91 \% & 34.21 & \textbf{0.01}  & & & &  \\
        \midrule
        \thead{Sparse Log \\ Regression} & 48 & 15 & \textbf{31 \%} & \textbf{387.38} & 0.19 & 14 & 29 \% & 477.29 & 0.24 & 14 & 29 \% & 471.22 & 0.24 & 14 & 29 \% & 446.59 & 0.24   & & & & \\
        \midrule
        \thead{Tailed Sparse \\ Regression} & 160 & 160 & \textbf{100 \%} & \textbf{0.3} & 0.0 & 160 & \textbf{100 \%} & 4.6 & 0.0 & 160 & \textbf{100 \%} & 4.35 & 0.0 & 160 & \textbf{100 \%} & 5.6 & 0.0   & & & & \\
        \midrule
        \thead{Tailed Sparse \\ Log Regression}  & 48 & 48 & \textbf{100 \%} & 0.85 & 0.0 & 48 & \textbf{100 \%} & 4.78 & 0.0 & 48 & \textbf{100 \%} & 4.81 & 0.0 & 48 & \textbf{100 \%} & 6.63 & 0.0   & & & & \\
        \bottomrule
    \end{tabular}
    \caption{\footnotesize \revision{Investigating the effect of the tightening settings and exploitation of strong convexity on the performance.}} 
\label{tab:SummaryOfTighteningAndStrongConvexity}
\end{sidewaystable}

\begin{sidewaystable}
    \centering
    \begin{tabular}{cl Hrrr Hrrr Hrrr Hrrr Hrrr} 
        \toprule
        \multicolumn{2}{l}{} & \multicolumn{4}{c}{Default} & \multicolumn{4}{c}{\thead{Away \\ Frank-Wolfe}} & \multicolumn{4}{c}{\thead{No Warm Start \\ No Shadow Set}} & \multicolumn{4}{c}{\thead{No Warm Start}} & \multicolumn{4}{c}{\thead{No Shadow Set}}\tabularnewline 
        
        \cmidrule(lr){3-6}
        \cmidrule(lr){7-10}
        \cmidrule(lr){11-14}
        \cmidrule(lr){15-18}
        \cmidrule(lr){19-22}

        Problem & \thead{\# \\ inst.} & \thead{\# \\ solved} & \thead{\% \\ solved} & \thead{Time (s)} & \thead{Relative \\ Gap} & \thead{\# \\ solved} & \thead{\% \\ solved} & \thead{Time (s)} & \thead{Relative \\ Gap} & \thead{\# \\ solved} & \thead{\% \\ solved} & \thead{Time (s)} & \thead{Relative \\ Gap} & \thead{\# \\ solved} & \thead{\% \\ solved} & \thead{Time (s)} & \thead{Relative \\ Gap} & \thead{\# \\ solved} & \thead{\% \\ solved} & \thead{Time (s)} & \thead{Relative \\ Gap}  \\

        \midrule
        \thead{MIPLIB \texttt{22433}}  & 15 & 15 & \textbf{100 \%} & \textbf{5.33} & 0.0 & 15 & \textbf{100 \%} & 8.74 & 0.0 & 15 & \textbf{100 \%} & 6.6 & 0.0 & 15 & \textbf{100 \%} & 7.45 & 0.0 & 15 & \textbf{100 \%} & 8.57 & 0.0 \\
        \midrule
        \thead{MIPLIB \texttt{neos5}}  & 15 & 14 & \textbf{93 \%} & \textbf{30.66} & 0.01 & 14 & \textbf{93 \%} & 49.89 & 0.01 & 14 & \textbf{93 \%} & 43.26 & 0.01 & 14 & \textbf{93 \%} & 49.27 & 0.01 & 14 & \textbf{93 \%} & 43.56 & 0.01\\
        \midrule
        \thead{MIPLIB \texttt{pg5\_34}}  & 15 & 0 & 0 \% & 1800.0 & 0.02 & 0 & 0 \% & 1800.0 & 0.02 & 0 & 0 \% & 1800.0 & 0.02 & 0 & 0 \% & 1800.0 & 0.02 & 0 & 0 \% & 1800.0 & 0.02\\
        \midrule
        \thead{MIPLIB \texttt{ran14x18-disj-8}} & 15 & 0 & 0 \% & 1802.86 & 0.09 & 0 & 0 \% & 1804.68 & 0.11 & 0 & 0 \% & 1800.49 & 0.13 & 0 & 0 \% & 1803.26 & 0.12 & 0 & 0 \% & 1805.01 & \textbf{0.08}  \\
        \midrule
        \thead{Poisson \\ Regression} & 120 & 30 & 25 \% & 406.01 & 3.79 & 30 & 25 \% & 503.92 & \textbf{0.31} & 32 & \textbf{27 \%} & \textbf{345.31} & \textbf{0.31} & 32 & \textbf{27 \%} & 348.17 & \textbf{0.31} & 30 & 25 \% & 459.33 & \textbf{0.31} \\
        \midrule
        \thead{Pure \\ Integer \\ Portfolio}  & 210 & 158 & \textbf{75 \%} & \textbf{236.29} & \textbf{1.86} & 101 & 48 \% & 616.09 & 1.87 & 140 & 67 \% & 448.87 & 1.88 & 146 & 70 \% & 437.36 & 2.24 & 157 & \textbf{75 \%} & 314.28 & \textbf{1.86} \\
        \midrule
        \thead{Mixed \\ Integer \\ Portfolio} & 210 & 199 & \textbf{95 \%} & 19.47 & 0.72 & 159 & 76 \% & 48.11 & 1.75 & 198 & 94 \% & 25.59 & 0.89 & 198 & 94 \% & \textbf{19.46} & 0.83 & 199 & \textbf{95 \%} & 20.24 & \textbf{0.1}\\
        \midrule
        \thead{Sparse \\ Regression} & 160 & 147 & \textbf{92 \%} & \textbf{25.14} & \textbf{0.01} & 142 & 89 \% & 23.25 & \textbf{0.01} & 142 & 89 \% & 20.13 & \textbf{0.01} & 142 & 89 \% & 21.4 & 0.02 & 147 & \textbf{92 \%} & 26.35 & \textbf{0.01} \\
        \midrule
        \thead{Sparse Log \\ Regression} & 48 & 15 & \textbf{31 \%} & \textbf{387.38} & 0.19 & 14 & 29 \% & 469.42 & \textbf{0.18} & 14 & 29 \% & 428.38 & 0.29 & 14 & 29 \% & 411.88 & 0.29 & 15 & \textbf{31 \%} & 472.98 & 0.24  \\
        \midrule
        \thead{Tailed Sparse \\ Regression} & 160 & 160 & \textbf{100 \%} & \textbf{0.3} & 0.0 & 160 & \textbf{100 \%} & 0.97 & 0.0 & 160 & \textbf{100 \%} & 0.36 & 0.0 & 160 & \textbf{100 \%} & 0.35 & 0.0 & 160 & \textbf{100 \%} & 4.33 & 0.0\\
        \midrule
        \thead{Tailed Sparse \\ Log Regression}  & 48 & 48 & \textbf{100 \%} & 0.85 & 0.0 & 48 & \textbf{100 \%} & 1.36 & 0.0 & 48 & \textbf{100 \%} & \textbf{0.77} & 0.0 & 48 & \textbf{100 \%} & 0.84 & 0.0 & 48 & \textbf{100 \%} & 4.8 & 0.0 \\
        \bottomrule
    \end{tabular}
    \caption{\footnotesize \revision{Investigating the effect of the warm start settings on the performance.}} 
\label{tab:SummaryOfWarmStart}
\end{sidewaystable}

\begin{sidewaystable} \scriptsize
    \centering
    \begin{tabular}{ll Hrrr Hrrr Hrrr Hrrr Hrrr Hrrr Hrrr} 
        \toprule
        \multicolumn{2}{l}{} & \multicolumn{4}{c}{\thead{Most \\ Fractional \\ Branching}} & \multicolumn{4}{c}{\thead{Strong \\ Branching}} & \multicolumn{4}{c}{\thead{hybrid \\ branching \\ depth $=$ \#int vars$/1$}} & \multicolumn{4}{c}{\thead{hybrid \\ branching \\ depth $=$ \#int vars$/2$}} & \multicolumn{4}{c}{\thead{hybrid \\ branching \\ depth $=$ \#int vars$/5$}} & \multicolumn{4}{c}{\thead{hybrid \\ branching \\ depth $=$ \#int vars$/10$}} & \multicolumn{4}{c}{\thead{hybrid \\ branching \\ depth $=$ \#int vars$/20$}} \tabularnewline 
        
        \cmidrule(lr){3-6}
        \cmidrule(lr){7-10}
        \cmidrule(lr){11-14}
        \cmidrule(lr){15-18}
        \cmidrule(lr){19-22}
        \cmidrule(lr){23-26}
        \cmidrule(lr){27-30}

        Problem & \thead{\# \\ inst.} & \thead{\# \\ solved} & \thead{\% \\ solved} & \thead{Time (s)} & \thead{Relative \\ Gap} & \thead{\# \\ solved} & \thead{\% \\ solved} & \thead{Time (s)} & \thead{Relative \\ Gap} & \thead{\# \\ solved} & \thead{\% \\ solved} & \thead{Time (s)} & \thead{Relative \\ Gap} & \thead{\# \\ solved} & \thead{\% \\ solved} & \thead{Time (s)} & \thead{Relative \\ Gap} & \thead{\# \\ solved} & \thead{\% \\ solved} & \thead{Time (s)} & \thead{Relative \\ Gap} & \thead{\# \\ solved} & \thead{\% \\ solved} & \thead{Time (s)} & \thead{Relative \\ Gap} & \thead{\# \\ solved} & \thead{\% \\ solved} & \thead{Time (s)} & \thead{Relative \\ Gap}  \\

        \midrule
        \thead{MIPLIB \\ \texttt{22433}}  & 15 & 15 & \textbf{100 \%} & \textbf{5.33} & 0.0 & 15 & \textbf{100 \%} & 9.11 & 0.0 & 15 & \textbf{100 \%} & 9.84 & 0.0 & 15 & \textbf{100 \%} & 9.93 & 0.0 & 15 & \textbf{100 \%} & 9.94 & 0.0 & 15 & \textbf{100 \%} & 10.69 & 0.0 & 15 & \textbf{100 \%} & 9.83 & 0.0 \\
        \midrule
        \thead{MIPLIB \\ \texttt{neos5}}  & 15 & 14 & \textbf{93 \%} & \textbf{30.66} & 0.01 & 14 & \textbf{93 \%} & 76.12 & 0.01 & 14 & \textbf{93 \%} & 64.19 & 0.01 & 14 & \textbf{93 \%} & 72.81 & 0.01 & 14 & \textbf{93 \%} & 67.63 & 0.01 & 14 & \textbf{93 \%} & 55.4 & 0.01 & 14 & \textbf{93 \%} & 51.68 & 0.01 \\
        \midrule
        \thead{MIPLIB \\ \texttt{pg5\_34}}  & 15 & 0 & 0 \% & 1800.0 & 0.02 & 0 & 0 \% & 1800.0 & 0.02 & 0 & 0 \% & 1800.0 & 0.02 & 0 & 0 \% & 1800.0 & 0.02 & 0 & 0 \% & 1800.0 & 0.02 & 0 & 0 \% & 1800.0 & 0.02 & 0 & 0 \% & 1800.0 & 0.02 \\
        \midrule
        \thead{MIPLIB \\ \texttt{ran14x18-} \\ \texttt{disj-8}} & 15 & 0 & 0 \% & 1802.86 & 0.09 & 0 & 0 \% & 1805.62 & 0.1 & 0 & 0 \% & 1806.27 & 0.1 & 0 & 0 \% & 1806.1 & 0.1 & 0 & 0 \% & 1806.04 & 0.1 & 0 & 0 \% & 1806.04 & 0.1 & 0 & 0 \% & 1805.96 & 0.1  \\
        \midrule
        \thead{Poisson \\ Regression} & 120 & 30 & \textbf{25 \%} & \textbf{406.01} & 3.79 & 30 & \textbf{25 \%} & 576.79 & 0.34 & 30 & \textbf{25 \%} & 580.72 & 0.34 & 30 & \textbf{25 \%} & 559.74 & 0.34 & 30 & \textbf{25 \%} & 571.63 & 0.33 & 30 & \textbf{25 \%} & 567.96 & \textbf{0.32} & 30 & \textbf{25 \%} & 529.2 & \textbf{0.32} \\
        \midrule
        \thead{Pure \\ Integer \\ Portfolio}  & 210 & 158 & \textbf{75 \%} & \textbf{236.29} & \textbf{1.86} & 137 & 65 \% & 449.74 & 1.96 & 130 & 62 \% & 459.39 & 1.88 & 138 & 66 \% & 454.27 & 1.96 & 153 & 73 \% & 342.97 & 1.88 & 158 & \textbf{75 \%} & 315.37 & \textbf{1.86} & 158 & \textbf{75 \%} & 294.04 & 1.87 \\
        \midrule
        \thead{Mixed \\ Integer \\ Portfolio} & 210 & 199 & 95 \% & \textbf{19.47} & 0.72 & 201 & \textbf{96 \%} & 30.38 & 0.38 & 202 & \textbf{96 \%} & 31.33 & 0.41 & 201 & \textbf{96 \%} & 30.13 & 0.39 & 200 & 95 \% & 28.15 & 0.48 & 200 & 95 \% & 27.83 & 0.94 & 199 & 95 \% & 25.57 & \textbf{0.03} \\
        \midrule
        \thead{Sparse \\ Regression} & 160 & 147 & \textbf{92 \%} & \textbf{25.14} & \textbf{0.01} & 136 & 85 \% & 62.92 & \textbf{0.01} & 136 & 85 \% & 58.94 & 0.02 & 136 & 85 \% & 56.74 & 0.02 & 138 & 86 \% & 54.64 & 0.02 & 146 & 91 \% & 45.25 & 0.02 & 148 & \textbf{92 \%} & 38.46 & \textbf{0.01}  \\
        \midrule
        \thead{Sparse Log \\ Regression} & 48 & 15 & \textbf{31 \%} & \textbf{387.38} & \textbf{0.19} & 14 & 29 \% & 821.55 & 0.45 & 14 & 29 \% & 705.96 & 0.45 & 14 & 29 \% & 627.81 & 0.45 & 14 & 29 \% & 532.43 & 0.42 & 14 & 29 \% & 571.82 & 0.41 & 15 & \textbf{31 \%} & 545.52 & 0.23   \\
        \midrule
        \thead{Tailed \\ Sparse \\ Regression} & 160 & 160 & \textbf{100 \%} & \textbf{0.3} & 0.0 & 160 & \textbf{100 \%} & 12.84 & 0.0 & 160 & \textbf{100 \%} & 9.95 & 0.0 & 160 & \textbf{100 \%} & 9.33 & 0.0 & 160 & \textbf{100 \%} & 7.28 & 0.0 & 160 & \textbf{100 \%} & 7.27 & 0.0 & 160 & \textbf{100 \%} & 7.29 & 0.0 \\
        \midrule
        \thead{Tailed  \\ Sparse Log \\ Regression}  & 48 & 48 & \textbf{100 \%} & \textbf{0.85} & 0.0 & 48 & \textbf{100 \%} & 10.37 & 0.0 & 48 & \textbf{100 \%} & 9.34 & 0.0 & 48 & \textbf{100 \%} & 9.26 & 0.0 & 48 & \textbf{100 \%} & 9.34 & 0.0 & 48 & \textbf{100 \%} & 9.33 & 0.0 & 48 & \textbf{100 \%} & 9.37 & 0.0 \\
        \bottomrule
    \end{tabular}
    \caption{\footnotesize \revision{Investigating the effect of the different branching strategies on the performance.}} 
\label{tab:SummaryOfBranching}
\end{sidewaystable}